\documentclass[reqno,12pt]{amsart}

\usepackage{amsmath}
\usepackage{amscd}
\usepackage{amssymb}
\usepackage{enumerate}
\usepackage{amsfonts}
\usepackage{graphicx}
\usepackage[all]{xy}
\usepackage{mathrsfs}
\usepackage[hypertex]{hyperref}

\usepackage{bbm}

\newcommand{\e}{\mathbbm{1}}

\newcommand{\limto}{{\displaystyle\lim_{\longrightarrow}}}
\newcommand{\rightlim}{\mathop{\limto}}

\newcommand{\leftlim}{\mathop{\displaystyle\lim_{\longleftarrow}}}
\newcommand{\limfromn}{\leftlim\limits_{\raise3pt\hbox{$n$}}}
\newcommand{\limton}{\rightlim\limits_{\raise3pt\hbox{$n$}}}

\newcommand{\rightlimit}[1]{\mathop{\lim\limits_{\longrightarrow}}\limits%
                   _{\raise3pt\hbox{$\scriptstyle #1$}}}
\newcommand{\leftlimit}[1]{\mathop{\lim\limits_{\longleftarrow}}\limits%
                   _{\raise3pt\hbox{$\scriptstyle #1$}}}

\numberwithin{equation}{section}

\newcommand{\rar}[1]{\stackrel{#1}{\longrightarrow}}
\newcommand{\xrar}[1]{\xrightarrow{#1}}

\newcommand{\into}{\hookrightarrow}

\newcommand{\al}{\alpha}
\newcommand{\be}{\beta}
\newcommand{\ga}{\gamma}
\newcommand{\Ga}{\Gamma}
\newcommand{\de}{\delta}

\newcommand{\la}{\lambda}

\newcommand{\sg}{\sigma}

\newcommand{\vp}{\varphi}

\newcommand{\bC}{{\mathbb C}}
\newcommand{\bD}{{\mathbb D}}

\newcommand{\bF}{{\mathbb F}}
\newcommand{\bG}{{\mathbb G}}

\newcommand{\bK}{{\mathbb K}}
\newcommand{\bL}{{\mathbb L}}

\newcommand{\bN}{{\mathbb N}}

\newcommand{\bQ}{{\mathbb Q}}

\newcommand{\bZ}{{\mathbb Z}}

\newcommand{\cC}{{\mathcal C}}

\newcommand{\cF}{{\mathcal F}}
\newcommand{\cG}{{\mathcal G}}
\newcommand{\cH}{{\mathcal H}}
\newcommand{\cI}{{\mathcal I}}

\newcommand{\cL}{{\mathcal L}}
\newcommand{\cM}{{\mathcal M}}

\newcommand{\cO}{{\mathcal O}}

\newcommand{\sC}{{\mathscr C}}
\newcommand{\sD}{{\mathscr D}}
\newcommand{\sE}{{\mathscr E}}

\newcommand{\sL}{{\mathscr L}}
\newcommand{\sM}{{\mathscr M}}

\newcommand{\fG}{{\mathfrak G}}

\newcommand{\fI}{{\mathfrak I}}

\newcommand{\fX}{{\mathfrak X}}
\newcommand{\fY}{{\mathfrak Y}}

\newcommand{\kbar}{{\overline{k}}}

\newcommand{\abs}[1]{\lvert #1\rvert}

\newcommand{\Loc}{\operatorname{Loc}}
\newcommand{\IrrLoc}{\operatorname{IrrLoc}}

\newcommand{\Ker}{\operatorname{Ker}}

\newcommand{\End}{\operatorname{End}}
\newcommand{\Hom}{\operatorname{Hom}}

\newcommand{\Aut}{\operatorname{Aut}}
\newcommand{\Spec}{\operatorname{Spec}}

\newcommand{\id}{\operatorname{id}}

\newcommand{\pr}{\mathrm{pr}}

\newcommand{\Ad}{\operatorname{Ad}}

\newcommand{\Ind}{\operatorname{Ind}}
\newcommand{\ind}{\operatorname{ind}}
\newcommand{\Res}{\operatorname{Res}}
\newcommand{\tr}{\operatorname{tr}}
\newcommand{\Rep}{\operatorname{Rep}}

\newcommand{\Av}{\operatorname{Av}}
\newcommand{\av}{\operatorname{av}}

\newcommand{\tens}{\otimes}

\newcommand{\st}{\,\big\vert\,}

\newcommand{\sbr}{\smallbreak}
\newcommand{\mbr}{\medbreak}

\newcommand{\bra}{\langle}
\newcommand{\ket}{\rangle}

\newcommand{\eval}[2]{\bra #1\,\big\vert\,#2\ket}

\newcommand{\qab}{\bQ^{\operatorname{ab}}}

\newcommand{\can}{\operatorname{can}}

\newtheorem{thm}{Theorem}[section]
\newtheorem{cor}[thm]{Corollary}
\newtheorem{lem}[thm]{Lemma}

\newtheorem{prop}[thm]{Proposition}

\theoremstyle{remark}
\newtheorem{rem}[thm]{Remark}
\newtheorem{rems}[thm]{Remarks}
\newtheorem{example}[thm]{Example}

\newtheorem{defin}[thm]{Definition}
\newtheorem{defins}[thm]{Definitions}

\newcommand{\Fun}{\operatorname{Fun}}

\newcommand{\Fr}{\operatorname{Fr}}

\newcommand{\Gal}{\operatorname{Gal}}

\newcommand{\ql}{\overline{\bQ}_\ell}

\newcommand{\cs}{CS}

\newcommand{\Perv}{\operatorname{Perv}}

\newcommand{\igz}{\ind_{G'_0}^{G_0}}
\newcommand{\Igz}{\Ind_{G'_0}^{G_0}}

\newcommand{\ig}{\ind_{G'}^G}
\newcommand{\Ig}{\Ind_{G'}^G}
\newcommand{\iga}{\ind_{\Ga'}^{\Ga}}

\newcommand{\var}{\!-\!\operatorname{var}}

\newcommand{\Gt}{\widetilde{G}}

\newcommand{\Ht}{\widetilde{H}}

\newcommand{\Ut}{{\widetilde{U}}}

\newcommand{\fun}[1]{\Fun(#1,\qab)^{#1}}

\newcommand{\tensqn}{\tens_{\bF_q}\bF_{q^n}}

\title[Character sheaves and characters]{Character sheaves and characters of unipotent groups over finite fields}

\author{Mitya Boyarchenko}

\date{\today}

\thanks{\textit{Address:} University of Michigan, Department of Mathematics, 530 Church Street, 2074 East Hall, Ann Arbor, MI 48109--1043. \textit{Email}: {\tt mityab@umich.edu} \\ Partially supported by the NSF grants DMS-0701106, DMS-0703679 and DMS-1001769.}

\begin{document}

\begin{abstract}
Let $G_0$ be a connected unipotent group over a finite field
$\bF_q$, and let $G=G_0\otimes_{\bF_q}\overline{\bF}_q$, equipped
with the Frobenius endomorphism $\Fr_q:G\rar{}G$. For every
character sheaf $M$ on $G$ such that $\Fr_q^*M\cong M$, we prove
that $M$ comes from an irreducible perverse sheaf $M_0$ on $G_0$
such that $M_0$ is pure of weight $0$ (as an $\ell$-adic complex) and for each integer $n\geq 1$ the ``trace of Frobenius'' function $t_{M_0\tensqn}$ on
$G_0(\bF_{q^n})$ takes values in $\bQ^{ab}$, the abelian closure of
$\bQ$.

\sbr

We further show that as $M$ ranges over all $\Fr_q^*$-invariant
character sheaves on $G$, the functions $t_{M_0}$ form a basis for the space of all conjugation-invariant
functions $G_0(\bF_q)\rar{}\bQ^{ab}$, and are orthonormal with
respect to the standard \emph{unnormalized} Hermitian inner product on this space. The matrix relating this basis to the basis formed by the irreducible characters is block-diagonal, with blocks corresponding to the $\Fr_q^*$-invariant $\bL$-packets (of characters or, equivalently, of character sheaves).

\sbr

We also formulate and prove a suitable generalization of this result
to the case where $G_0$ is a possibly disconnected unipotent group
over $\bF_q$. (In general, $\Fr_q^*$-invariant character sheaves on
$G$ are related to the irreducible characters of the groups of
$\bF_q$-points of all pure inner forms of $G_0$ over $\bF_q$.)
\end{abstract}

\maketitle

\setcounter{tocdepth}{1}

\tableofcontents


\section*{Introduction}

This article is a sequel to \cite{characters} and \cite{foundations}. In \cite{characters} we defined and studied $\bL$-packets of irreducible characters of unipotent groups over finite fields. In \cite{foundations}
the notion of a character sheaf on a unipotent group over
an algebraically closed field of characteristic $p>0$ was defined, and several basic properties of character sheaves and $\bL$-packets
thereof were established. In the present paper we study the relationship between
irreducible characters of a group of the form $G_0(\bF_q)$, where
$G_0$ is a unipotent group over a finite field $\bF_q$, and
character sheaves on the group $G=G_0\tens_{\bF_q}\overline{\bF}_q$
that are invariant under the Frobenius endomorphism
$\Fr_q:G\rar{}G$.

\mbr

This relationship is easier to formulate when $G_0$ is connected.
In this case the number of $\Fr_q^*$-invariant character sheaves $M$
on $G$ equals the number of irreducible characters of $G_0(\bF_q)$.
Moreover, every such $M$ comes from an irreducible perverse sheaf $M_0$ on $G_0$ such that $M_0$ is pure of weight $0$ and the corresponding ``trace of Frobenius'' function $t_{M_0}$ on $G_0(\bF_q)$ takes values in $\qab$, the abelian closure of $\bQ$. With this choice of $M_0$ we prove (cf.~Theorem \ref{t:main-connected}) that as $M$ ranges over the set of $\Fr_q^*$-invariant character sheaves on $G$, the functions $t_{M_0}$ form a basis of the space of conjugation-invariant functions $G_0(\bF_q)\rar{}\qab$, which is orthonormal with respect to the standard unnormalized inner product $\eval{f_1}{f_2}=\sum_{g\in G_0(\bF_q)} f_1(g)\overline{f_2(g)}$. The matrix relating this basis to the basis formed by the irreducible characters of $G_0(\bF_q)$ is block-diagonal, with blocks labeled by the $\bL$-packets (Theorem \ref{t:main}(d)).

\mbr

For example, suppose $G_0$ is a connected \emph{commutative} unipotent group over $\bF_q$. For each multiplicative $\ql$-local system $\cL_0$ on $G_0$, the function $t_{\cL_0}:G_0(\bF_q)\rar{}\ql^\times$ is a group homomorphism, and the map $\cL_0\longmapsto t_{\cL_0}$ is a bijection between the set of isomorphism classes of multiplicative local systems on $G_0$ and $\Hom(G_0(\bF_q),\ql^\times)$. The pure perverse sheaves on $G_0$ of weight $0$, which were mentioned in the previous paragraph, are the complexes\footnote{If $\dim G_0$ is odd, one needs to choose a square root of $q$, which also leads to the choice of a square root (with respect to the tensor product of local systems) of the Tate sheaf $\ql(1)$.} $M_0=\cL_0[\dim G_0](\dim G_0/2)$. Note that whereas the functions $t_{\cL_0}$ are orthonormal for the normalized inner product, the corresponding functions $t_{M_0}=q^{-\dim G_0/2}t_{\cL_0}$ are orthonormal for the unnormalized inner product.

\mbr

We remark that even if ultimately one is only interested in the
connected case, studying arbitrary unipotent groups seems to be necessary because that is what opens up the possibility of proving the main results by induction on the size of $G_0$. Indeed, as we saw in
\cite{characters}, the study of $\bL$-packets of irreducible
characters\footnote{This notion is recalled in
\S\ref{ss:L-packets-characters}.} of $G_0(\bF_q)$ leads us to
consider characters of subgroups $G_0'(\bF_q)\subset G_0(\bF_q)$ for
certain closed subgroups $G_0'\subset G_0$ that may be
disconnected. Similarly, the down-to-earth construction of character
sheaves on $G$ that was presented in \cite{foundations} also
involves induction from possibly disconnected closed subgroups of
$G$.

\mbr

Thus our main goal in the present article is to formulate and prove
a theorem on the relationship between character sheaves and
characters when $G_0$ is an arbitrary unipotent group over $\bF_q$.
In general, there may be more $\Fr_q^*$-invariant character sheaves
on $G$ than there are irreducible characters of $G_0(\bF_q)$. For
example, if $G_0$ is a \emph{finite discrete} unipotent group over
$\bF_q$, then $G$ can be viewed as an abstract finite $p$-group
$\Ga$ on which the Frobenius endomorphism acts trivially. Character
sheaves on $\Ga$ are the same as irreducible objects of the category
of $\Ga$-equivariant sheaves of vector spaces on $\Ga$ (with respect to the conjugation action), and they are
automatically $\Fr_q^*$-invariant. Unless $\Ga$ is trivial, the
number of such objects is strictly greater than the number of
irreducible characters of the group $G_0(\bF_q)=\Ga$.

\mbr

The example mentioned in the previous paragraph is discussed in more detail in \S\ref{ss:finite-discrete}. It naturally leads us to consider all pure inner forms\footnote{Note that if $G_0$ is connected, then by Lang's theorem \cite{lang}, the Galois cohomology $H^1(\bF_q,G_0)$ is trivial, and thus $G_0$ only has one pure inner form, namely, $G_0$ itself.} of $G_0$ over $\bF_q$. This phenomenon is also related to the fact that if $M_0$ is a conjugation-equivariant $\ell$-adic complex on $G_0$, then for each $\al\in H^1(\bF_q,G_0)$, we get an $\ell$-adic complex $M_0^\al$ on the corresponding pure inner form $G_0^\al$, and hence a conjugation-invariant function $t_{M_0^\al}:G_0^\al(\bF_q)\rar{}\ql$. Taking into account the fact that a $\Fr_q^*$-invariant character sheaf on $G$ determines a function on $G_0^\al(\bF_q)$ for each $\al\in H^1(\bF_q,G_0)$ allows us to find the correct formulation of the relationship between character sheaves and irreducible characters for a possibly disconnected unipotent group over $\bF_q$ (Theorem \ref{t:main}).

\mbr

At the end of \S\ref{ss:main-result} we discuss the organization of the remainder of the article.

\subsection*{Acknowledgments} The results on the relationship between character sheaves and characters of unipotent groups over finite fields, which are proved in this article, were conjectured several years ago by Vladimir Drinfeld.

\mbr

I am very thankful to him for introducing me to this area of research, for his continued support and advice, and for teaching me the techniques used in this article. I am grateful to him and to Alexander Beilinson for suggesting numerous improvements in the exposition. I am especially indebted to Drinfeld for suggesting an approach to showing that there are ``sufficiently many'' character sheaves on a unipotent group $G_0$ over $\bF_q$ (so that the associated ``trace of Frobenius'' functions span the space of all conjugation-invariant functions on $G_0(\bF_q)$), which is shorter and clearer than my original argument.

\mbr

I thank Martin Olsson for helpful email correspondence and for suggesting to me the reference \cite{sun}.

\mbr

I must also acknowledge great intellectual debt to George Lusztig, both because he conjectured in \cite{lusztig} the existence of a theory of character sheaves on unipotent groups, and because his work on character sheaves and characters of reductive groups over finite fields inspired, to a large extent, the results of the present article.

\subsection*{Notation} We work with two fixed primes $p\neq\ell$. We choose an algebraic closure $\bF$ of a field with $p$ elements and an algebraic closure $\ql$ of the field $\bQ_\ell$ of $\ell$-adic numbers. All our geometric objects are defined over a field of characteristic $p>0$, and all derived categories are those of constructible $\ql$-complexes. If $M_0$ is such a complex on a scheme $X_0$ of finite type (or a perfect variety) over a finite field $\bF_q$, the corresponding ``trace of Frobenius'' function is denoted by $t_{M_0}:X_0(\bF_q)\rar{}\ql$. We write $M_0\tensqn$ for the pullback of $M_0$ via the projection $X_0\tensqn\rar{}X_0$.

\mbr

We will denote by $\qab$ the abelian closure of $\bQ$ in $\ql$, and by $z\mapsto\overline{z}$ the complex conjugation automorphism of $\qab$. If $r\in\bN$, we write $\mu_r\subset\ql^\times$ for the subgroup consisting of $r$-th roots of unity. We also write $\bF_{p^r}\subset\bF$ for the unique subfield consisting of $p^r$ elements.


\section{The connected case}\label{s:connected}

In this section we review the definitions of character sheaves and $\bL$-packets
thereof and then state our main result for connected unipotent groups
over finite fields.

\subsection{Perfect schemes and groups}\label{ss:perfect-recollections} We recall that a scheme $S$ in characteristic $p$ (i.e., such
that $p$ annihilates the structure sheaf $\cO_S$ of $S$) is said to
be \emph{perfect} if the morphism $\cO_S\rar{}\cO_S$, given by
$f\longmapsto f^p$ on the local sections of $\cO_S$, is an
isomorphism of sheaves. As explained in
\cite[\S1.9]{foundations}, it is more convenient to work with
perfect schemes and perfect algebraic groups when developing the
theory of character sheaves on unipotent groups (see
\emph{loc.~cit.} for more details). We follow the same approach in this article and, for brevity, introduce the following terms.

\begin{defins}\label{defs:perfect}
Let $k$ be a perfect field of characteristic $p>0$.
 \sbr
\begin{enumerate}[(1)]
\item A \emph{perfect variety} over $k$ is a perfect scheme over $k$ that is isomorphic to the perfectization \cite{greenberg} of a scheme of finite type over $k$.
 \sbr
\item A \emph{perfect group} over $k$ is a group object in the category of perfect varieties over $k$; equivalently \cite[Lemma A.7]{characters}, it is a group scheme over $k$ that is isomorphic to the perfectization of a group scheme of finite type over $k$.
 \sbr
\item A \emph{perfect unipotent group} over $k$ is a group scheme over $k$ that is isomorphic to the perfectization of a unipotent algebraic group over $k$.
\end{enumerate}
\end{defins}

\begin{rem}
Even though we formulate and prove the results of the present article in the context of perfect unipotent groups over finite fields, the same statements are also true for a usual unipotent group $G_0$ over $\bF_q$. The reason is that if $G_0^{perf}$ is the perfectization of $G_0$, then the natural morphism $G_0^{perf}\rar{}G_0$ induces a group isomorphism $G_0^{perf}(\bF_{q^n})\rar{}G_0(\bF_{q^n})$ for every $n\in\bN$, as well as an equivalence\footnote{In particular, character sheaves on $G_0$ and on $G_0^{perf}$ are ``the same.''} between the \'etale topos of $G_0$ and that of $G_0^{perf}$.
\end{rem}

\subsection{Recollections on character sheaves}\label{ss:recollections}
Fix a perfect field $k$ of characteristic $p>0$. If $X$ is a perfect variety over $k$, we write
$\sD(X)=D^b_c(X,\ql)$ for the bounded derived category of
constructible complexes of $\ql$-sheaves on $X$. If $G$ is a perfect
unipotent group over $k$ acting on $X$, one can define the equivariant derived
category $\sD_G(X)$ as in \cite[Def.~4.5]{characters} or \cite[Def.~1.3]{foundations}. (It would have been better to define $\sD_G(X)$ as the $\ell$-adic derived category of the quotient stack $G\backslash X$ \cite{Las-Ols06}, but the \emph{ad hoc} approach used in \cite{characters,foundations} suffices for our purposes.)

\mbr

The notation $\sD_G(G)$ always refers to the conjugation action of
$G$ on itself.

\mbr

The categories $\sD(G)$ and $\sD_G(G)$ are monoidal with respect to
the bifunctor of \emph{convolution with compact supports}, which we
denote by\footnote{Following \cite[\S1.2]{foundations}, we almost always
omit the letters ``R'' and ``L'' from our notation for the six
functors defined on the derived categories $\sD(X)$.}
\[
(M,N) \longmapsto M*N=\mu_!\bigl((p_1^* M)\tens(p_2^* N)\bigr),
\]
where $\mu:G\times G\rar{}G$ is the multiplication morphism and
$p_1,p_2:G\times G\rar{}G$ are the first and second projections. The
unit object in each of the categories is the delta-sheaf at the
identity element of $G$, which will be denoted by $\e$.

\mbr

An object $e\in\sD_G(G)$ is said to be a \emph{closed idempotent} if
there exists an arrow $\e\rar{}e$ that becomes an isomorphism after
convolution with $e$. It is further said to be a \emph{minimal}
closed idempotent if $e\neq 0$ and for every closed idempotent
$e'\in\sD_G(G)$, we have either $e*e'\cong e$, or $e*e'=0$.

\begin{rem}\label{r:different-notions-of-minimality}
In \cite{foundations}, the notion of a \emph{weak
idempotent} in $\sD_G(G)$, defined as an object $e\in\sD_G(G)$ such
that $e*e\cong e$, was also used. The notion of a \emph{minimal weak idempotent} is
defined analogously to the notion of a minimal closed idempotent. By \cite[Thm.~1.49(a)--(b)]{foundations}, the classes of minimal weak
idempotents and of minimal closed idempotents in $\sD_G(G)$ coincide
when $k=\overline{k}$. (Note that, \emph{a priori}, it is neither
obvious that a minimal weak idempotent in $\sD_G(G)$ is a closed
idempotent, nor that a minimal closed idempotent in $\sD_G(G)$ is
also minimal as a weak idempotent.)
\end{rem}

In view of this remark, we may (and will) shorten ``minimal closed
idempotent'' to ``minimal idempotent'' in the future (when
$k=\overline{k}$) without fear of confusion.

\mbr

\textbf{For the rest of the section we assume that $k$ is
algebraically closed.}

\begin{defins}\label{d:min-idempotents-L-packets}
Let $G$ be a perfect unipotent group over
$k$, and let $e\in\sD_G(G)$ be a closed idempotent.
\begin{enumerate}[(1)]
\item The \emph{Hecke subcategory} defined by $e$ is the full subcategory $e\sD_G(G)\subset\sD_G(G)$ consisting of objects $M\in\sD_G(G)$ such that $e*M\cong M$. By \cite[Lem.~2.18]{foundations}, $e\sD_G(G)$ is closed under $*$ and is a monoidal category with unit object $e$.
 \sbr
 In the remainder of the definition, assume that $e$ is
minimal.
 \sbr
\item Let $\sM^{perv}_e$ denote the full subcategory of
$e\sD_G(G)$ consisting of those objects for which the underlying
$\ell$-adic complex is a perverse sheaf \cite[Ch. 4]{bbd} on $G$. The \emph{Lusztig packet
of character sheaves} on $G$ defined by $e$ is the set of
(isomorphism classes of) indecomposable objects of the category
$\sM_e^{perv}$.
 \sbr
\item An object of $\sD_G(G)$ is a \emph{character sheaf} if it lies
in the Lusztig packet of character sheaves defined by some minimal
closed idempotent in $\sD_G(G)$.
\end{enumerate}
\end{defins}

The set of isomorphism classes of character sheaves on $G$ is
denoted by $\cs(G)$. From now on we will write ``$\bL$-packet'' instead of ``Lusztig packet'' for brevity.

\begin{rem}\label{r:langlands}
The conjectural notion of an $L$-packet in the representation theory
of reductive groups over \emph{local} fields was introduced by
R.P.~Langlands in \cite{langlands}. It is hard to compare it with
the notion of an $\bL$-packet because technically the two definitions
are given in quite different terms. However, philosophically the two notions
are closely related. Namely, as explained by R.~Bezrukavnikov, $\bL$-packets are philosophically similar to \emph{geometric} $L$-packets, which are, in general, larger than the $L$-packets defined by Langlands\footnote{Conjecturally, $L$-packets bijectively correspond to ``Langlands parameters.''  Geometric $L$-packets should correspond to Frobenius-invariant ``geometric Langlands parameters'' (one obtains geometric Langlands parameters from the usual ones by replacing the Weil-Deligne group $W'_K$ with $\Ker (W'_K\twoheadrightarrow\bZ)$). Thus each geometric $L$-packet is a union of several ordinary $L$-packets.}.
\end{rem}

\subsection{Properties of $\bL$-packets of character
sheaves}\label{ss:properties-L-packets} We keep the notation of
\S\ref{ss:recollections}. In this subsection we list a few basic
properties of minimal closed idempotents in $\sD_G(G)$ and the
corresponding $\bL$-packets of character sheaves.

\mbr

Let $\iota:G\rar{}G$ be the inversion morphism, $g\mapsto g^{-1}$.
As in \cite[Def.~1.17]{foundations}, we denote by $\bD^-_G:\sD(G)^\circ\rar{}\sD(G)$
the composition $\bD_G^-=\bD_G\circ\iota^*=\iota^*\circ\bD_G$, where
$\bD_G$ is the Verdier duality functor. The functors $\bD_G$ and
$\bD_G^-$ can be naturally ``upgraded'' to functors
$\sD_G(G)^\circ\rar{}\sD_G(G)$, which we also denote by $\bD_G$ and
$\bD_G^-$.

\begin{prop}\label{p:properties-char-sheaves}
Let $G$ be a perfect unipotent group over
$k$ $($where $k=\overline{k}$ as before$)$, and let $e\in\sD_G(G)$ be a minimal idempotent. Then:
\begin{enumerate}[$($a$)$]
\item $\sM_e^{perv}$ is a semisimple abelian category with finitely
many simple objects.
 \sbr
\item There exists a $($necessarily unique$)$ $n_e\in\bZ$ such that $\bD_G^-(e)\cong e[-2n_e]$.
 \sbr
\item We have $e[-n_e]\in\sM_e^{perv}$, and $e[-n_e]$ is a character sheaf.
 \sbr
\item The subcategory $\sM_e=\sM_e^{perv}[n_e]\subset\sD_G(G)$ is
closed under convolution, and is a monoidal category with unit
object $e$.
 \sbr
\item The canonical functor $D^b(\sM_e^{perv})\rar{}e\sD_G(G)$ is an equivalence of categories.
\end{enumerate}
\end{prop}

\begin{proof}
All the assertions above are contained in \cite[Thm.~1.15, Prop.~1.19]{foundations}.
\end{proof}

Note that assertion (a) implies that $\bL$-packets of character
sheaves on $G$ are finite and that character sheaves in the $\bL$-packet defined by $e$ can also be characterized as the simple objects of the category $\sM_e^{perv}$.

\mbr

The next notion \cite[Def.~1.21]{foundations} is the geometric
analogue of the notion of functional dimension in the classical
representation theory of Lie groups.

\begin{defin}
If $e\in\sD_G(G)$ is a minimal closed idempotent, the number
$d_e:=(\dim G-n_e)/2$ is called the \emph{functional dimension} of
$e$.
\end{defin}

We note that $d_e$ may fail to be an integer \cite[Rem.~1.23]{foundations}. One can show that $d_e,n_e\geq 0$ \cite[Thm.~1.15(b)]{foundations}, but we do not need this fact in this article.

\subsection{Connected unipotent groups over finite fields}
\label{ss:connected-unipotent-finite-fields}
Let $q$ be a power of a prime $p$, which we assume to be fixed once and for all.
If $X_0$ is a perfect variety over $\bF_q$ and
$M_0\in\sD(X_0)$, we denote the function associated to $M_0$ via the
functions-sheaves dictionary by $t_{M_0}:X_0(\bF_q)\rar{}\ql$
(see \cite{deligne-weil-2}, \cite[\S4.2]{characters}). Thus for each $n\in\bN$, we also have the corresponding function $t_{M_0\tensqn}:X_0(\bF_{q^n})\rar{}\ql$.

\mbr

Let $G_0$ be a perfect unipotent group over $\bF_q$, and write
$G=G_0\tens_{\bF_q}\bF$. The absolute Frobenius morphism\footnote{It
is defined to be the identity map on the underlying topological
space, and the map $f\mapsto f^q$ on the local sections of the
structure sheaf.} $\Phi_q:G_0\rar{}G_0$ induces an
$\bF$-homomorphism $\Fr_q:G\rar{}G$, called the \emph{Frobenius
endomorphism} of $G$, by extension of scalars. The pullback functor
$\Fr_q^*$ induces an automorphism of the set $\cs(G)$, and
$\cs(G)^{\Fr_q^*}$ will denote the set of character sheaves on $G$
invariant under this automorphism.

\mbr

Let $p^r$ denote the exponent of $G_0$, i.e., $r\geq 1$ is the minimal integer such that $g^{p^r}=1$ for all $g\in G_0(\bF)$.
We write $\bZ[\mu_{p^{2r}},p^{-1}]\subset\bQ(\mu_{p^{2r}})$ for the subring generated by $\mu_{p^{2r}}$ and $\frac{1}{p}$ (cf.~Remark \ref{r:why-exponent-squared}).

\begin{thm}\label{t:main-connected}
Assume that $G_0$ is a connected perfect unipotent group over $\bF_q$.
\begin{enumerate}[$($a$)$]
\item If $M\in\cs(G)^{\Fr_q^*}$, then $M$ arises from an irreducible
perverse sheaf $M_0$ on $G_0$ such that $M_0$ is
pure\footnote{The definition of a pure complex on a perfect variety over $\bF_q$ is recalled in \S\ref{ss:purity-reminder}.} of
weight $0$ and the function $t_{M_0\tensqn}:G_0(\bF_{q^n})\rar{}\ql$ takes
values in the subring $\bZ[\mu_{p^{2r}},p^{-1}]\subset\qab\subset\ql$ for every $n\in\bN$.
 \sbr
\item For each $M\in\cs(G)^{\Fr_q^*}$, fix a choice of $M_0$ subject
to the requirements stated in part $($a$)$. The functions
\[
 \bigl\{ t_{M_0} : G_0(\bF_q)\rar{}\qab \st
 M\in\cs(G)^{\Fr_q^*} \bigr\}
\]
form a basis of the space of conjugation-invariant functions
$G_0(\bF_q)\rar{}\qab$, which is orthonormal with respect to the
inner product
\[
\eval{f_1}{f_2} = \sum_{g\in G_0(\bF_q)}
f_1(g)\,\overline{f_2(g)}.
\]
\end{enumerate}
\end{thm}

This result is a special case of Theorem \ref{t:main} (which yields additional information even for connected $G_0$), and is only formulated separately for expository reasons. (The statement of Theorem \ref{t:main} requires some additional preparations.)

\begin{rem}\label{r:why-exponent-squared}
The irreducible characters of $G_0(\bF_q)$ over $\ql$ take values in the subring $\bZ[\mu_{p^r}]\subset\ql$, and the corresponding minimal central idempotents in the group algebra of $G_0(\bF_q)$ are defined over $\bZ[\mu_{p^r},p^{-1}]$. Thus it is natural to ask whether $\bZ[\mu_{p^{2r}},p^{-1}]$ can be replaced with the smaller ring $\bZ[\mu_{p^r},p^{-1}]$ in the statement of Theorem \ref{t:main-connected}(a). If $p^r=2$, the answer is negative: see Example \ref{ex:exponent-2}.

\mbr

For $p^r>2$, the answer is not known to us. The precise place in our arguments where it becomes necessary to consider $\bZ[\mu_{p^{2r}},p^{-1}]$ is indicated in Remark \ref{r:optimality-bound-exponent}.
\end{rem}

\begin{example}\label{ex:exponent-2}
Suppose that $G_0=\bG_a$ and $q$ is an odd power of $2$, so that $p^r=2$ and $\bZ[\mu_{p^r},p^{-1}]\subset\bQ$. Then $\ql[1]$ is a $\Fr_q^*$-invariant character sheaf on $G$ corresponding to the trivial representation of $\bG_a(\bF_q)$, and if $M_0$ is any pure perverse sheaf of weight $0$ on $G_0$ whose base change to $\bF$ is isomorphic to $\ql[1]$, then $t_{M_0}$ does \emph{not} take values in $\bQ$. In fact, we must necessarily have $\abs{t_{M_0}(0)}^2=\frac{1}{q}$, and since $q$ is an odd power of $2$, there is no element $\la\in\bQ$ with $\abs{\la}^2=\frac{1}{q}$.
\end{example}

\begin{rem}\label{r:absolute-value-square-root-of-p}
On the other hand, we observe that if $p^r>2$, the subring $\bZ[\mu_{p^r}]\subset\qab$ always contains an element $\la$ such that $\abs{\la}^2=\la\cdot\overline{\la}=p$. Indeed, if $p=2$, one can take $\la=1+i$, and if $p>2$, one can take $\la=\sum_{a\in\bF_p^\times}\left(\dfrac{a}{p}\right)\zeta^a$, where $\left(\dfrac{a}{p}\right)$ is the Legendre symbol and $\zeta\in\mu_{p^r}$ is a primitive $p$-th root of $1$. So if $p^r>2$, then for any $s\in\bZ$, the ring $\bZ[\mu_{p^r},p^{-1}]$ contains an element $\la$ such that $\la\cdot\overline{\la}=p^s$.
\end{rem}

\begin{rem}\label{r:uniqueness-basis}
The perverse sheaf $M_0$ in Theorem \ref{t:main-connected}(a) is determined uniquely up to tensor product with a $\ql$-local system on $G_0$ obtained by pullback from a rank $1$ local system on $\Spec\bF_q$ such that the corresponding character $\Gal(\bF/\bF_q)\rar{}\ql^\times$ takes values in the subgroup consisting of elements of $\bQ(\mu_{p^{2r}})$ of absolute value $1$.

\mbr

In particular, the function $t_{M_0}$ is determined up to multiplication by an element of $\bQ(\mu_{p^{2r}})$ of absolute value $1$. Thus the orthonormality assertion in Theorem \ref{t:main-connected}(b) is unambiguous even though the functions $t_{M_0}$ are not uniquely defined.
\end{rem}

\subsection{Easy unipotent groups} Recall from \cite{characters} that a (perfect) unipotent group $G$ is said to be \emph{easy} if every geometric point of $G$ is contained in the neutral connected component of its centralizer. One of the main examples of an easy group (over any field) is the group $UL_n$ of unipotent upper-triangular matrices of size $n$.

\begin{thm}\label{t:main-easy}
Let $G_0$ be an easy perfect unipotent group over $\bF_q$, and let $G$ and $\Fr_q$ be defined as before. Choose a minimal idempotent $e\in\sD_G(G)$ such that $\Fr_q^*e\cong e$.
 \sbr
\begin{enumerate}[$(a)$]
\item There exists a unique $($up to isomorphism$)$ weak idempotent $e_0\in\sD_{G_0}(G_0)$ such that $e$ is obtained from $e_0$ by base change. Moreover, $e_0$ is a closed idempotent.
 \sbr
\item The underlying complex of $e_0$ is pure of weight $0$.
 \sbr
\item The functional dimension $d_e$ is an integer, and $q^{\dim G-d_e}\cdot t_{e_0}:G_0(\bF_q)\rar{}\ql$ is an irreducible character of the group $G_0(\bF_q)$ of degree $q^{d_e}$.
 \sbr
\item Every irreducible character of $G_0(\bF_q)$ over $\ql$ is of the form $q^{\dim G-d_e}\cdot t_{e_0}$ for a $($unique$)$ $\Fr_q^*$-invariant minimal idempotent $e\in\sD_G(G)$.
\end{enumerate}
\end{thm}

This result is proved in \S\ref{ss:proof-t:main-easy} below.

\begin{rem}\label{r:sum-of-squares-of-dimensions}
Theorem \ref{t:main-easy}(c) can be generalized as follows. Suppose that $G_0$ is a connected perfect unipotent group over $\bF_q$, which is not necessarily easy, and let $e_0\in\sD_{G_0}(G_0)$ be a geometrically minimal weak idempotent (\S\ref{ss:L-packets-characters}). Then
\begin{equation}\label{e:sum-of-squares-of-dimensions}
\sum \chi(1)^2 = q^{2d_e},
\end{equation}
where the sum on the left hand side ranges over all irreducible characters of $G_0(\bF_q)$ in the $\bL$-packet defined by $e_0$ (see Definition \ref{d:L-packets-characters}(b) and Remark \ref{r:difference-between-notions-of-L-packet}) and $d_e$ is the functional dimension of the minimal idempotent $e\in\sD_G(G)$ obtained from $e_0$ by base change. For the proof of \eqref{e:sum-of-squares-of-dimensions}, see \S\ref{ss:proof-r:sum-of-squares-of-dimensions}.
\end{rem}


\section{The disconnected case}\label{s:disconnected}

In this section we state the main result of the article, namely,
Theorem \ref{t:main}. It is more informative than Theorem
\ref{t:main-connected}, because it explains the relationship between
$\bL$-packets of character sheaves and those of irreducible characters
(\S\ref{ss:L-packets-characters}) and yields an explicit
construction of $\Fr_q^*$-stable $\bL$-packets of character sheaves in
terms of admissible pairs (\S\ref{ss:admissible-pairs}) defined over
$\bF_q$. Furthermore, it extends Theorem \ref{t:main-connected} to
the case where $G_0$ is an arbitrary (possibly disconnected)
unipotent group over $\bF_q$.

\subsection{Finite discrete groups}\label{ss:finite-discrete}
For motivation, we begin by discussing a simple situation that is the ``opposite'' of the connected case.
Let $\Ga$ be an (abstract) finite group, and let $G_0$ be the corresponding finite discrete group scheme\footnote{As an $\bF_q$-scheme, $G_0$ is the disjoint union of copies of $\Spec\bF_q$ labeled by $\Ga$, and the group structure is induced by that on $\Ga$.} over $\bF_q$. If $\Ga$ is a $p$-group, then $G_0$ is unipotent, but this restriction is unimportant for what follows.

\mbr

The Frobenius action on $G_0(\bF)$ is trivial and we can identify $\Ga$ both with $G_0(\bF)$ and with $G_0(\bF_q)$. The category $\sD_G(G)$ can be identified with the derived category of the category of equivariant $\ql$-sheaves\footnote{That is, sheaves of finite dimensional $\ql$-vector spaces, where the topology on $\Ga$ is discrete.} on $\Ga$ (with respect to the conjugation action of $\Ga$ on itself); the latter is known as the \emph{Drinfeld double} (or \emph{quantum double}) of $\Ga$. There is only one nonzero idempotent in $\sD_G(G)$, namely, the unit object, and it follows that the character sheaves on $G$ are the simple objects in the category of equivariant $\ql$-sheaves on $\Ga$. It is known that these objects are classified up to isomorphism by pairs $(x,\rho)$ up to simultaneous $\Ga$-conjugation, where $x\in\Ga$ and $\rho$ is an isomorphism class of irreducible representations of the centralizer $Z_\Ga(x)$ of $x$ in $\Ga$. Again, the Frobenius action on the set of all character sheaves on $G$ is trivial.

\mbr

In particular, taking $x=1$, we obtain a bijection between the set of irreducible representations of $\Ga=G_0(\bF_q)$ and a certain subset of the set of character sheaves on $G$. However, unless $\Ga$ is trivial, there are other character sheaves on $G$, corresponding to irreducible representations of the centralizers of nontrivial elements of $\Ga$.

\mbr

One way to account for them is to notice that if $x\in\Ga$, then conjugation by $x$ is an automorphism of $\Ga$ which can be thought of as the Frobenius corresponding to some $\bF_q$-structure on $\Ga$ viewed as a discrete group over $\bF$. If $G_0^x$ denotes the corresponding group over $\bF_q$, then $G_0^x(\bF_q)$ is identified with the centralizer $Z_\Ga(x)\subset\Ga$.

\mbr

In the language of algebraic group theory, $G_0^x$ is a pure inner form of $G_0$ defined by the conjugacy class of $x$. (By definition, the pure inner forms $G_0^x$ and $G_0^y$ are the same if and only if $x$ and $y$ are conjugate in $\Ga$.) Thus we see that if we consider not only the irreducible representations of $G_0(\bF_q)$ but also the irreducible representations of the groups of $\bF_q$-points of each of the pure inner forms of $G_0$, we restore the equality between the number of irreducible characters and the number of (Frobenius-invariant) character sheaves. As we explain in this section, the same pattern holds for an arbitrary unipotent group over $\bF_q$.

\subsection{Functions on groupoids}\label{ss:functions-on-groupoids} Following a suggestion of V.~Drinfeld, we introduce a formalism that will allow us to clarify the statement and the proof of Theorem \ref{t:main}. It is inspired by the formalism of ``masses of categories'' used in \cite[\S4]{deligne-flicker}.

\begin{defin}\label{d:finite-groupoid}
A groupoid $\cG$ is \emph{finite} if its set of isomorphism classes $\pi_0(\cG)$ is finite and the automorphism group of any object of $\cG$ is also finite.
\end{defin}

\begin{defin}\label{d:function-on-a-groupoid}
A \emph{function on a groupoid} $\cG$ is a function $\pi_0(\cG)\rar{}\qab$. We write $\Fun(\cG)$ for the vector space of functions on $\cG$.
\end{defin}

\begin{defin}\label{d:standard-inner-product}
If $\cG$ is a finite groupoid, we write $L^2(\cG)$ for the space $\Fun(\cG)$ equipped with the (Hermitian) $L^2$ inner product corresponding to the measure on $\pi_0(\cG)$ such that the measure of the isomorphism class of an object $X\in\cG$ is equal to $\frac{1}{\abs{\Aut_{\cG}(X)}}$.
\end{defin}

\begin{rem}
The total volume of this measure is equal to the mass of $\cG$ as defined in \cite[\S4.10]{deligne-flicker}.
\end{rem}

\begin{defin}\label{d:inertia-groupoid}
Let $\cG$ be a groupoid. The \emph{inertia groupoid} $\cI_{\cG}$ of $\cG$ is the groupoid of pairs $(X,f)$, where $X\in\cG$ and $f\in\Aut_{\cG}(X)$. More concisely, one can define $\cI_{\cG}$ is the groupoid of $1$-morphisms $B\bZ\rar{}\cG$, or as the $2$-fiber product of the diagonal $\cG\rar{}\cG\times\cG$ with itself.
\end{defin}

\begin{defin}\label{d:representations-and-characters-of-groupoids}
\begin{enumerate}[(a)]
\item The category of \emph{representations} of a groupoid $\cG$ is the category of functors from $\cG$ to the category of $\ql$-vector spaces.
 \sbr
\item If $\rho$ is a finite dimensional representation of a finite groupoid $\cG$, the \emph{character} of $\rho$ is the function on $\cI_{\cG}$ given by $(X,f)\mapsto\operatorname{Tr}\rho(f)$.
\end{enumerate}
\end{defin}

\begin{rem}\label{r:character-theory-for-finite-groupoids}
Standard character theory implies that if $\cG$ is a finite groupoid, then the characters of irreducible representations of $\cG$ form an orthonormal basis in $L^2(\cI_{\cG})$.
\end{rem}

\begin{example}\label{ex:groupoids-coming-from-finite-groups}
Let $\Ga$ be a finite group. A representation of $\Ga$ is the same thing as a representation of the groupoid $B\Ga$. The inertia groupoid of $B\Ga$ can be identified with the category whose objects are elements of $\Ga$, a morphism from $x$ to $y$ is an element $\ga\in\Ga$ with $\ga x\ga^{-1}=y$, and the composition of morphisms is the product in $\Ga$. We denote the latter groupoid by $[\Ga]$. The space $L^2\bigl([\Ga])$ can be identified with the space $\fun{\Ga}$ of conjugation-invariant functions $\Ga\rar{}\qab$ equipped with the standard inner product, and the assertion of Remark \ref{r:character-theory-for-finite-groupoids} amounts to the well known orthogonality relations for irreducible characters of finite groups.
\end{example}

\subsection{Pure inner forms}\label{ss:pure-inner-forms} Pure inner forms of algebraic groups over finite fields, as well as some related notions, were discussed in\footnote{The term ``inner form'' was used in \emph{op.~cit.}, but ``pure inner form'' is more appropriate.} \cite[\S6]{characters}. We also summarize all the constructions and results we will need in this article in \S\ref{s:pure-inner-forms} below. For the time being it suffices to mention that if $G_0$ is a perfect group over $\bF_q$, then, given a class $\al\in H^1(\bF_q,G_0)$ in Galois cohomology, one can define the corresponding pure inner form $G_0^\al$; it is another perfect group\footnote{The group $G_0^\al$ is determined uniquely up to an isomorphism, which itself is unique up to an inner automorphism given by conjugation by an element of $G_0^\al(\bF_q)$.} over $\bF_q$ that becomes isomorphic to $G_0$ over $\bF$. Moreover, there is a natural monoidal equivalence $\sD_{G_0}(G_0)\rar{\sim}\sD_{G_0^\al}(G_0^\al)$, which we denote $M_0\longmapsto M_0^\al$ and call the ``transport of equivariant complexes'' (see \S\ref{ss:transport-equivariant-complexes}).

\begin{rem}\label{r:F-q-points-quotient-stack-by-conjugation}
Let $[G_0]$ denote the quotient stack of $G_0$ by the conjugation action of $G_0$ on itself. The groupoid $[G_0](\bF_q)$ can be identified with the disjoint union of the groupoids $[G_0^\al(\bF_q)]$ (see Example \ref{ex:groupoids-coming-from-finite-groups}) as $\al$ ranges over $H^1(\bF_q,G_0)$.
\end{rem}

\subsection{L-packets of irreducible
characters}\label{ss:L-packets-characters} Let $G_0$ be a perfect unipotent group over
$\bF_q$, and write $G=G_0\tens_{\bF_q}\bF$. An object
$e_0\in\sD_{G_0}(G_0)$ is called a \emph{geometrically minimal
idempotent} if $e_0*e_0\cong e_0$ and the corresponding
object $e\in\sD_G(G)$ is a minimal idempotent. Given $\al\in H^1(\bF_q,G_0)$, we have the corresponding pure inner
form $G^\al_0$ and the geometrically minimal idempotent $e^\al_0\in\sD_{G^\al_0}(G^\al_0)$ obtained from $e_0$ via transport of equivariant complexes (see \S\ref{ss:pure-inner-forms}). In particular, the function $t_{e^\al_0}:G^\al_0(\bF_q)\rar{}\ql$ is a central idempotent with respect to convolution.

\begin{defin}\label{d:L-packets-characters}
\begin{enumerate}[$($a$)$]
\item For each $\al\in H^1(\bF_q,G_0)$, consider the set of
(isomorphism classes of) irreducible representations of
$G_0^\al(\bF_q)$ over $\ql$ on which the idempotent $t_{e^\al_0}$
acts via the identity. The disjoint union of these sets is called
the \emph{$\bL$-packet of irreducible representations of $G_0$ defined by
$e_0$}.
 \sbr
\item The set of characters of the representations in (a), viewed as
a subset of\footnote{Equality \eqref{e:functions} follows from Remark \ref{r:F-q-points-quotient-stack-by-conjugation}.}
\begin{equation}\label{e:functions}
\Fun\bigl([G_0](\bF_q)\bigr) = \bigoplus_{\al\in H^1(\bF_q,G_0)}
\Fun(G_0^\al(\bF_q),\qab)^{G_0^{\al}(\bF_q)},
\end{equation}
is called the \emph{$\bL$-packet of irreducible characters of $G_0$ defined by
$e_0$}.
\end{enumerate}
\end{defin}

\begin{rem}\label{r:difference-between-notions-of-L-packet}
This notion is formally different from the definition of an $\bL$-packet of irreducible characters of $G_0(\bF_q)$ given in \cite[Def.~2.7]{characters}. Pure inner forms of $G_0$ were not considered in \emph{op.~cit.}, and $G_0$ was assumed to be connected there. If $G_0$ is connected, then the two notions are equivalent, but this fact requires proof.

\mbr

Indeed, with the definition above, it is not immediate that $\bL$-packets of irreducible characters are nonempty. Nevertheless, in \S\ref{ss:proof-p:equivalence-def-L-packet-characters} we will prove the following result.
\end{rem}

\begin{prop}\label{p:equivalence-def-L-packet-characters}
Let $G_0$ be any perfect unipotent group over $\bF_q$. The $\bL$-packets of irreducible characters of $G_0$ are nonempty and pairwise disjoint. Their union is equal to the disjoint union of the sets of all the irreducible characters of the groups $G^\al_0(\bF_q)$, viewed as subsets of $\Fun\bigl([G_0](\bF_q)\bigr)$. If $G_0$ is connected, the notion of an $\bL$-packet of irreducible characters of $G_0(\bF_q)$ defined above is equivalent to the one given in \cite[Def.~2.7]{characters}.
\end{prop}

\subsection{Admissible pairs and Serre duality}\label{ss:admissible-pairs} We briefly
recall a few more definitions from \cite{characters,foundations}. In
this subsection $k$ can be any perfect field of characteristic
$p>0$. If $H$ is a connected perfect unipotent group over $k$, a
nonzero $\ql$-local system $\cL$ on $H$ is said to be
\emph{multiplicative} if
$\mu^*\cL\cong\cL\boxtimes\cL$, where $\mu:H\times H\rar{}H$ is the
group operation in $H$. Roughly speaking, the \emph{Serre dual},
$H^*$, of $H$ is defined as the moduli space of multiplicative local
systems on $H$ (see \cite[\S3.1]{foundations} as well as \cite[Appendix A]{characters} for the details). We
recall (\emph{loc.~cit.}) that $H^*$ is a possibly disconnected
perfect commutative unipotent group over $k$.

\mbr

If $G$ is a perfect unipotent group over $k$ and $(H,\cL)$ is a pair
consisting of a closed connected subgroup $H\subset G$ and a
multiplicative local system $\cL$ on $H$, the \emph{normalizer},
$G'$, of the pair $(H,\cL)$ in $G$ is defined in two steps. First,
let $N_G(H)$ be the normalizer of $H$ in $G$. Then $N_G(H)$ acts on
$H$ by conjugation, and we have the induced action of $N_G(H)$ on
$H^*$. We then let $G'$ be the stabilizer of $[\cL]\in H^*(k)$ in
$N_G(H)$ with respect to the latter action. Sometimes we write $G'=N_G(H,\cL)$.

\begin{defin}\label{d:admissible-pair}
\begin{enumerate}[$($a$)$]
\item If $k$ is algebraically closed, the pair $(H,\cL)$ as in the
previous paragraph is said to be \emph{admissible} for $G$ if the
following three conditions hold:
 \sbr
\begin{itemize}
\item Let $G'$ be the normalizer of $(H,\cL)$ in $G$, and let $G^{\prime\circ}$ denote its neutral
connected component. Then $G^{\prime\circ}/H$ is commutative.
 \sbr
\item The $k$-group morphism
$\vp_\cL : G^{\prime\circ}/H \rar{} \bigl(G^{\prime\circ}/H\bigr)^*$
induced by $\cL$ (whose construction is explained in \cite[\S{}A.13]{characters} and \cite[\S3.3]{foundations}) is an isogeny.
 \sbr
\item For every $g\in G(k)$ such that $g\not\in G'(k)$, we have
\[
\cL\bigl\lvert_{(H\cap H^g)^\circ} \not\cong
\cL^g\bigl\lvert_{(H\cap H^g)^\circ},
\]
where $H^g=g^{-1}Hg$ and $\cL^g$ is the multiplicative local system
on $H^g$ obtained from $\cL$ by transport of structure (via the map
$h\mapsto g^{-1}hg$).
\end{itemize}
 \sbr
\item If $k$ is an arbitrary perfect field of characteristic $p>0$,
the pair $(H,\cL)$ is said to be \emph{admissible} if the
corresponding pair obtained by base change to an algebraic closure
$\overline{k}$ of $k$ is admissible for $G\tens_k\overline{k}$.
\end{enumerate}
\end{defin}

\begin{rem}
We do not recall the statement of the second condition in the
definition of admissibility in detail, because it will not be
explicitly used in our article.
\end{rem}

\begin{prop}[See
\cite{characters,foundations}]\label{p:adm-pair-min-idemp} Let $k$
be a perfect field of characteristic $p>0$ and let $G$ be a perfect
unipotent group over $k$. Consider an admissible pair $(H,\cL)$ for
$G$ and define $G'\subset G$ to be its normalizer. The object
$e'_{H,\cL}\in\sD_{G'}(G')$ obtained from
$\bK_H\tens\cL\in\sD_{G'}(H)$ via extension by zero is a minimal
closed idempotent in $\sD_{G'}(G')$, and $\ig(e'_{H,\cL})$ is a
minimal closed idempotent in $\sD_G(G)$.
\end{prop}

In this statement, $\bK_H$ denotes the dualizing complex of $H$. The
definition of the functor $\ig:\sD_{G'}(G')\rar{}\sD_G(G)$ is
recalled in \S\ref{ss:induction-functors} below.

\begin{defin}\label{d:adm-pair-min-idemp}
In the setting of Proposition \ref{p:adm-pair-min-idemp}, the object
$\ig(e'_{H,\cL})\in\sD_G(G)$ is called the \emph{minimal idempotent
defined by the admissible pair $(H,\cL)$}. In the special case where
$G=G'$, we call $e'_{H,\cL}$ a \emph{Heisenberg minimal idempotent}.
\end{defin}

\subsection{The main result}\label{ss:main-result} The next theorem is the
central result of our work. As in
\S\ref{ss:connected-unipotent-finite-fields}, we assume that $\bF$
is an algebraic closure of a finite field with $p$ elements, and
recall that $\cs(G)$ denotes the set of character sheaves on a
perfect unipotent group $G$ over $\bF$. The notions of a character
sheaf and of an $\bL$-packet of character sheaves defined by a minimal
idempotent $e\in\sD_G(G)$ were introduced in Definitions
\ref{d:min-idempotents-L-packets}.

\begin{thm}\label{t:main}
Let $G_0$ be any perfect unipotent group over $\bF_q$, let
$G=G_0\tens_{\bF_q}\bF$, and let $\Fr_q:G\rar{}G$ be the Frobenius
endomorphism.
\begin{enumerate}[$($a$)$]
\item If $e\in\sD_G(G)$ is a minimal idempotent such that
$\Fr_q^*(e)\cong e$, there is a unique $($up to isomorphism$)$ weak idempotent $e_0\in\sD_{G_0}(G_0)$ such that $e$ is obtained from $e_0$ by base change. Moreover, $e_0$ is a closed idempotent in $\sD_{G_0}(G_0)$.
 \sbr
\item In the situation of $($a$)$, there exist a pure inner form $G_0^\al$ of $G_0$ and an admissible pair $(H_0,\cL_0)$ for $G_0^\al$ defined over $\bF_q$ such that $e_0^\al\in\sD_{G^\al_0}(G^\al_0)$
is isomorphic to the minimal idempotent defined by $(H_0,\cL_0)$
$($Definition~\ref{d:adm-pair-min-idemp}$)$.
 \sbr
\item Let $e$ and $e_0$ be as in part $($a$)$. Then $e_0$ is pure of weight $0$. Moreover, if $M$ is a character
sheaf in the $\bL$-packet defined by $e$ and $\Fr_q^*M\cong M$, then
$M$ comes from an object $M_0\in e_0\sD_{G_0}(G_0)$ such that $M_0$
is perverse and pure of weight $0$, and the function
$t_{(M_0\tensqn)^\al}:G^\al_0(\bF_{q^n})\rar{}\ql$ takes values in $\bZ[\mu_{p^{2r}},p^{-1}]$ for each $n\in\bN$ and each $\al\in H^1(\bF_{q^n},G_0\tensqn)$, where $p^r$ is the exponent of $G_0$ and we write $G^\al_0(\bF_{q^n})$ in place of $(G_0\tensqn)^\al(\bF_{q^n})$ for brevity.
 \sbr
\item In the situation of $($c$)$, the elements
$T_{M_0}=\bigl(t_{M_0^\al}\bigr)_{\al\in H^1(\bF_q,G_0)}\in\Fun\bigl([G_0](\bF_q)\bigr)$ are linearly
independent, and their span is equal to the span of the set of
elements in the $\bL$-packet of irreducible characters defined by
$e_0$ $($see \eqref{e:functions} and Def.~\ref{d:L-packets-characters}$)$.
 \sbr
\item For each $M\in\cs(G)^{\Fr_q^*}$, fix a choice of $M_0$ subject
to the requirements stated in part $($c$)$. The elements $\bigl\{T_{M_0}\st M\in\cs(G)^{\Fr_q^*}\bigr\}$ form an orthonormal basis of the space
$\Fun\bigl([G_0](\bF_q)\bigr)$ with respect to the inner product
$\eval{F_1}{F_2}=q^{\dim G}\cdot(F_1,F_2)$, where $(\cdot,\cdot)$ is the inner product on $L^2\bigl([G_0](\bF_q)\bigr)$ $($see Def.~\ref{d:standard-inner-product}$)$.
\end{enumerate}
\end{thm}

A natural analogue of Remark \ref{r:uniqueness-basis} applies in
this situation. We also observe that there is some redundancy in the
statement of the theorem: for instance, the linear independence
assertion in part (d) follows from part (e).

\begin{cor}[of Theorem \ref{t:main}(d)]
Let $G_0$ be a perfect unipotent group over $\bF_q$, and let $e_0\in\sD_{G_0}(G_0)$ be a geometrically minimal idempotent. The cardinality of the $\bL$-packet of irreducible characters of $G_0$ defined by $e_0$ is equal to the number of $\Fr_q^*$-invariant elements in the $\bL$-packet of character sheaves defined by $e_0$.
\end{cor}

\begin{rems}\label{rems:t:main}
\begin{enumerate}[(a)]
\item If $G_0$ is connected, then $\Fun\bigl([G_0](\bF_q)\bigr)=\Fun(G_0(\bF_q),\qab)^{G_0(\bF_q)}$ and $\abs{G_0(\bF_q)}=q^{\dim G}$. Therefore Theorem \ref{t:main} implies Theorem \ref{t:main-connected}.
 \sbr
\item In the situation of Theorem \ref{t:main}(a), there may not exist an admissible pair for $G_0$ defined over $\bF_q$ such that the corresponding minimal idempotent (Definition \ref{d:adm-pair-min-idemp}) is isomorphic to $e_0$. Indeed, in \S\ref{ss:example-for-main-theorem} we construct an example in which the function $t_{e_0}$ is identically zero; then one can apply Proposition \ref{p:when-min-idemp-comes-from-adm-pair}. Therefore passing to a pure inner form of $G_0$ in Theorem \ref{t:main}(b) is necessary.
 \sbr
\item If, in the setting of Theorem \ref{t:main}(c), we take $M=e[-n_e]$ (cf.~Prop.~\ref{p:properties-char-sheaves}(c)), then we can take $M_0=e_0[-n_e](-n_e/2)\tens\pr^*\sL$ for a suitable pure rank $1$ local system $\sL$ of weight $0$ on $\Spec\bF_q$, where $\pr:G_0\rar{}\Spec\bF_q$ denotes the structure morphism: see \S\ref{ss:proof-rems:t:main-c}. (The reason we cannot take $M_0=e_0[-n_e](-n_e/2)$ in general is that the corresponding function may not take values in $\bZ[\mu_{p^{2r}},p^{-1}]$.)
\end{enumerate}
\end{rems}

The next result is proved in \S\ref{ss:proof-p:when-min-idemp-comes-from-adm-pair}.

\begin{prop}\label{p:when-min-idemp-comes-from-adm-pair}
Let $G_0$ be a perfect unipotent group over $\bF_q$ and $e_0\in\sD_{G_0}(G_0)$ a geometrically minimal idempotent.
The following are equivalent:
 \sbr
\begin{enumerate}[$($i$)$]
\item there exists an admissible pair $(H_0,\cL_0)$ for $G_0$ such that the minimal idempotent defined by it $($Definition \ref{d:adm-pair-min-idemp}$)$ is isomorphic to $e_0$;
 \sbr
\item the function $t_{e_0}:G_0(\bF_q)\rar{}\ql$ defined by $e_0$ is not identically zero;
 \sbr
\item $t_{e_0}(1)\neq 0$.
\end{enumerate}
\end{prop}

Let us discuss the remainder of the article in more detail. In \S\ref{s:induction} we recall the definition of induction functors from \cite{foundations} and the definition of mixed and pure complexes on a perfect variety over $\bF_q$ \cite{deligne-weil-2,bbd}. We also prove some results relating these notions. In \S\ref{s:pure-inner-forms} we summarize some basic definitions and facts having to do with pure inner forms of perfect unipotent groups over $\bF_q$ and the transport of equivariant complexes. In \S\ref{s:weil} we develop some tools for relating $\ell$-adic complexes on a perfect variety over $\bF_q$ to those on its base change to $\bF$. The main ingredients of the proofs of Theorems \ref{t:main} and \ref{t:main-easy} are collected in \S\ref{s:proofs}. Finally, the more technical steps involved in our arguments, as well as the proofs of various side remarks made in the text, are provided in the Appendix.


\section{Induction functors}\label{s:induction}

\subsection{Definition of averaging
functors}\label{ss:averaging-functors} Let $k$ be any perfect field
of characteristic $p>0$. We begin by recalling certain definitions
and constructions from \cite{foundations}.

\mbr

Let $G$ be a perfect unipotent group over $k$, let $X$ be a perfect
variety over $k$ equipped with a $G$-action, and
let $G'\subset G$ be a closed subgroup. Let us equip $(G/G')\times
X$ with the diagonal action of $G$, let $\pr_2:(G/G')\times
X\rar{}X$ denote the second projection, and let $i:X\into
(G/G')\times X$ be the morphism taking $x\in X$ to $(\bar 1,x)$,
where $\bar 1\in G/G'$ is the image of $1\in G$.

\mbr

Then the pullback functor $i^*:\sD_G((G/G')\times
X)\rar{}\sD_{G'}(X)$ is an equivalence of categories, and we make
the following
\begin{defin}\label{d:averaging}
The \emph{averaging} functor $\Av_{G/G'}:\sD_{G'}(X)\rar{}\sD_G(X)$
and the functor $\av_{G/G'}:\sD_{G'}(X)\rar{}\sD_G(X)$ of
\emph{averaging with compact supports} are defined by
\[
\Av_{G/G'}=\pr_{2*}\circ(i^*)^{-1} \qquad\text{and}\qquad
\av_{G/G'}=\pr_{2!}\circ(i^*)^{-1}.
\]
\end{defin}

By construction, we have a canonical morphism of functors
$\av_{G/G'}\rar{}\Av_{G/G'}$.

\begin{lem}\label{l:averaging-adjunctions}
The functor $\Av_{G/G'}$ is right adjoint, and the functor
$\av_{G/G'}[2d](d)$ is left adjoint, to the forgetful functor
$\sD_G(X)\rar{}\sD_{G'}(X)$, where $d=\dim(G/G')$.
\end{lem}
\begin{proof}
This is \cite[Lemma 1.35]{foundations}.
\end{proof}

\begin{rem}\label{r:functoriality-of-averaging}
The construction of averaging functors is functorial in the sense
that if $f:Y\rar{}X$ is a \emph{smooth} $G$-equivariant morphism of
perfect varieties over $k$ equipped with a $G$-action, then we
have natural isomorphisms of functors
\[
\Av^Y_{G/G'}\circ f^* \cong f^*\circ\Av^X_{G/G'}
\qquad\text{and}\qquad \av^Y_{G/G'}\circ f^* \cong
f^*\circ\av^X_{G/G'}
\]
where $\Av^X$, $\av^X$ are the averaging functors for $X$ and
$\Av^Y$, $\av^Y$ are the averaging functors for $Y$. (Use the proper
and smooth base change theorems.)
\end{rem}

\subsection{Averaging functors and
duality}\label{ss:averaging-duality} If $X$ is any perfect
variety over $k$, we let
$\bD_X:\sD(X)^\circ\rar{\sim}\sD(X)$ denote the Verdier duality functor.
Given an action of a perfect unipotent group $G$ on $X$, the functor
$\bD_X$ naturally lifts to $\sD_G(X)^\circ$.

\begin{lem}\label{l:averaging-duality}
Let $G'\subset G$ be a closed subgroup, and let $d=\dim(G/G')$. Then
the functors $\bD_X\circ\Av_{G/G'}\circ\bD_X$ and
$\av_{G/G'}[2d](d)$ are isomorphic.
\end{lem}

\begin{proof}
This is \cite[Lemma 6.9]{foundations}.
\end{proof}

\subsection{Reminder on pure complexes}\label{ss:purity-reminder}
We now briefly recall a few definitions from
\cite{deligne-weil-2,bbd}. Let $X_0$ be a perfect variety over $\bF_q$ and consider a point $x\in X_0(\bF_{q^n})$, that is, an $\bF_q$-morphism
$x:\Spec\bF_{q^n}\rar{} X_0$. If $\cF$ is a $\ql$-sheaf on $X_0$,
then $x^*\cF$ is a continuous $\ql$-representation of  $\Gal(\bF/\bF_{q^n})$. The \emph{geometric Frobenius} is the
generator $F_{q^n}$ of $\Gal(\bF/\bF_{q^n})$, defined as the
\emph{inverse} of the automorphism $a\longmapsto a^{q^n}$.

\begin{defins}\label{d:purity}
\begin{enumerate}[$($a$)$]
\item A $\ql$-sheaf $\cF$ on $X_0$ is \emph{punctually pure of weight $w\in\bZ$} if for
every $n\geq 1$ and every $x\in X_0(\bF_{q^n})$, the eigenvalues of
$F_{q^n}$ acting on $x^*\cF$ are algebraic numbers, all of whose
conjugates in $\bC$ have absolute value $(q^n)^{w/2}$.
 \sbr
\item A $\ql$-sheaf $\cF$ on $X_0$ is \emph{mixed} if it has a
finite filtration whose successive subquotients are punctually pure
of some weights. The weights of the (nonzero) subquotients in this
filtration are called the \emph{weights of $\cF$}.
 \sbr
\item An object $M\in\sD(X_0)=D^b_c(X_0,\ql)$ is said to be \emph{mixed}
if the cohomology sheaves $\cH^i(M)$ are mixed for all $i\in\bZ$. We
let $\sD_m(X_0)\subset\sD(X_0)$ denote the full subcategory formed
by mixed complexes.
 \sbr
\item Let $w\in\bZ$. An object $M\in\sD_m(X_0)$ is said to \emph{have
weights $\leq w$} if, for every $i\in\bZ$, the weights of $\cH^i(M)$
are $\leq w+i$. We let $\sD_{\leq w}(X_0)\subset\sD_m(X_0)$ denote
the full subcategory formed by complexes whose weights are $\leq w$.
 \sbr
\item An object $M\in\sD_m(X_0)$ is said to \emph{have weights $\geq
w$} if the Verdier dual $\bD_{X_0}(M)$ lies in $\sD_{\leq -w}(X_0)$.
We let $\sD_{\geq w}(X_0)\subset\sD_m(X_0)$ denote the full
subcategory formed by complexes whose weights are $\geq w$.
 \sbr
\item An object $M\in\sD_m(X_0)$ is \emph{pure of weight
$w$} if $M\in\sD_{\leq w}(X_0)\cap\sD_{\geq w}(X_0)$.
\end{enumerate}
\end{defins}

\begin{rem}\label{r:twists-preserve-weights}
For each integral or half-integral\footnote{In order to make sense of the Tate twist $M_0\longmapsto M_0(d)$ for half-integral $d$, one needs to choose a square root of $q$ in $\qab$.} $d$, the functor $M_0\longmapsto M_0[2d](d)$ preserves the subcategories $\sD_{\leq w}(X_0)\subset\sD(X_0)$ and $\sD_{\geq w}(X_0)\subset\sD(X_0)$ for all $w\in\bZ$.
\end{rem}

\begin{thm}[Deligne, \cite{deligne-weil-2}, 3.3.1,
6.2.3]\label{t:deligne} Let $X_0$, $Y_0$ be perfect varieties over
$\bF_q$, and let $f:Y_0\rar{}X_0$ be an $\bF_q$-morphism. For every
$w\in\bZ$, the functor $f_!:\sD(Y_0)\rar{}\sD(X_0)$ takes $\sD_{\leq
w}(Y_0)$ into $\sD_{\leq w}(X_0)$, and $f_*:\sD(Y_0)\rar{}\sD(X_0)$
takes $\sD_{\geq w}(Y_0)$ into $\sD_{\geq w}(X_0)$.
\end{thm}

\begin{cor}\label{c:averaging-purity}
Let $G_0$ be a perfect unipotent group acting on a perfect
variety $X_0$ over $\bF_q$, and let
$G'_0\subset G_0$ be a closed subgroup. Fix $M\in\sD_{G'_0}(X_0)$.
If the underlying complex of $M$ lies in $\sD_{\leq w}(X_0)$ for
some $w\in\bZ$, then so does the underlying complex of
$\av_{G_0/G_0'}(M)$. If the underlying complex of $M$ lies in
$\sD_{\geq w}(X_0)$ for some $w\in\bZ$, then so does the underlying
complex of $\Av_{G_0/G_0'}(M)$. Thus if $M$ is pure of weight $w$
and the canonical morphism
$\av_{G_0/G_0'}(M)\rar{}\Av_{G_0/G_0'}(M)$ is an isomorphism, then
$\av_{G_0/G_0'}(M)$ is pure of weight $w$.
\end{cor}

The corollary follows from Theorem \ref{t:deligne}, Lemma \ref{l:averaging-duality}, Remark \ref{r:twists-preserve-weights} and the following

\begin{lem}\label{l:averaging-purity}
In the situation of Corollary \ref{c:averaging-purity}, let $G_0$ act diagonally on $(G_0/G_0')\times X_0$ and define $i:X_0\into(G_0/G_0')\times X_0$ by $x\mapsto(\overline{1},x)$ as in \S\ref{ss:averaging-functors}.

\mbr

If $N\in\sD_{G_0}((G_0/G_0')\times X_0)$ is such that $($the underlying complex of$)$ $i^*N$ lies in $\sD_{\leq w}(X_0)$ for some $w\in\bZ$, then $N$ lies in $\sD_{\leq w}((G_0/G_0')\times X_0)$.
\end{lem}

The proof of Lemma \ref{l:averaging-purity} is given in \S\ref{ss:proof-l:averaging-purity}.

\subsection{Induction functors}\label{ss:induction-functors} If
$k$ is any perfect field of characteristic $p>0$ and $G'\subset G$
are perfect unipotent groups over $k$, consider the conjugation
actions of $G$ and $G'$ on themselves, and let $j:G'\into G$ denote
the inclusion morphism. Then $j$ induces the functor
$j_!=j_*:\sD_{G'}(G')\rar{}\sD_{G'}(G)$ of extension by zero.

\begin{defin}\label{d:induction-functors}
The \emph{induction functor} $\Ig:\sD_{G'}(G')\rar{}\sD_G(G)$, and
the functor $\ig:\sD_{G'}(G')\rar{}\sD_G(G)$ of \emph{induction with
compact supports}, are defined as
\[
\Ig=\Av_{G/G'}\circ j_! \qquad\text{and}\qquad \ig=\av_{G/G'}\circ
j_!.
\]
\end{defin}

By construction, we have a canonical morphism of functors
$\can:\ig\rar{}\Ig$.

\begin{rem}\label{r:equivalent-def-induction}
In \cite[\S5.2.2]{characters} we gave a different definition of the functor $\ig$, which we now recall. Form the quotient $\widetilde{G}=(G\times G')/G'$ for the right $G'$-action on $G\times G'$ given by $(g,g')\cdot\ga=(g\ga,\ga^{-1}g'\ga)$. The left multiplication action of $G$ on itself induces a $G$-action on $\widetilde{G}$. We have the $G'$-equivariant injection $i:G'\into\widetilde{G}$ induced by $g'\mapsto(1,g')$ and the $G$-equivariant morphism $\pi:\widetilde{G}\rar{}G$ induced by $(g,g')\mapsto gg'g^{-1}$. The functor $i^*:\sD_G(\widetilde{G})\rar{}\sD_{G'}(G')$ is an equivalence, and in \emph{loc.~cit.} we defined $\ig=\pi_!\circ(i^*)^{-1}:\sD_{G'}(G')\rar{}\sD_G(G)$.

\mbr

It will be useful for us to know that the two constructions of $\ig$ lead to naturally isomorphic functors. This assertion is proved in \S\ref{ss:equivalence-def-induction}.
\end{rem}

Let us list some properties of induction functors. The
following facts result from Lemmas
\ref{l:averaging-adjunctions} and \ref{l:averaging-duality} together
with Remark \ref{r:functoriality-of-averaging}.

\begin{prop}\label{p:induction-duality}
Let $G$ be a perfect unipotent group over $k$, let $G'\subset G$ be
a closed subgroup, and let $d=\dim(G/G')$. The functor $\Ig$ is
right adjoint to the restriction functor
$\sD_G(G)\rar{}\sD_{G'}(G')$. Furthermore, the functors
$\bD_G\circ\Ig\circ\bD_{G'}$, $\bD^-_G\circ\Ig\circ\bD^-_{G'}$ and
$\ig[2d](d)$ are isomorphic\footnote{The functor $\bD_G^-:\sD_G(G)^\circ\rar{}\sD_G(G)$ was
introduced in \S\ref{ss:properties-L-packets}.}.
\end{prop}

The next proposition follows immediately from Corollary
\ref{c:averaging-purity}.

\begin{prop}\label{p:induction-purity}
Let $G_0$ be a perfect unipotent group over $\bF_q$, and let
$G'_0\subset G_0$ be a closed subgroup. Fix $M\in\sD_{G'_0}(G'_0)$
and $w\in\bZ$. If $M\in\sD_{\leq w}(G'_0)$, then
$\igz(M)\in\sD_{\leq w}(G_0)$. If $M\in\sD_{\geq w}(G'_0)$, then
$\Igz(M)\in\sD_{\geq w}(G_0)$. Thus if $M$ is pure of weight $w$ and
the canonical morphism $\can_M:\igz(M)\rar{}\Igz(M)$ is an
isomorphism, then $\igz(M)$ is pure of weight $w$.
\end{prop}


\section{Pure inner forms}\label{s:pure-inner-forms}

Throughout this section we work with an algebraic closure $\bF$ of a field of prime order $p$ and a finite subfield $\bF_q\subset\bF$. Pure inner forms\footnote{This notion is not to be confused with the concept of a pure complex recalled in \S\ref{ss:purity-reminder}.} of perfect groups over $\bF_q$ were used in \cite{characters} to study the relationship between the induction functor $\igz$, where $G_0'$ is a closed subgroup of a perfect unipotent group $G_0$ over $\bF_q$, and the operation of induction of class functions from $G'_0(\bF_q)$ to $G_0(\bF_q)$.

\mbr

The two types of induction are compatible ``on the nose'' when $G_0'$ is connected. In general, the relationship is more complicated \cite[Proposition 6.13]{characters}. We give a more general result in Proposition \ref{p:induction-sheaves-functions} below.

\subsection{Definitions}\label{ss:defs-inner-forms} We follow the same conventions regarding the definition of pure inner forms as in \cite{characters}. In particular, if $G_0$ is a perfect group over $\bF_q$, we write $H^1(\bF_q,G_0)$ for the set of isomorphism classes of right $G_0$-torsors.

\begin{defin}\label{d:pure-inner-form-group}
A \emph{pure inner form} of $G_0$ is a pair $(G_1,P)$ consisting of a perfect group $G_1$ over $\bF_q$ and a $(G_1,G_0)$-bitorsor\footnote{Let us recall the definition. One requires that $P$ is a perfect variety over $\bF_q$ equipped with a left $G_1$-action and a right $G_0$-action, that these actions commute, and that they make $P$ a left (respectively, right) torsor under $G_1$ (respectively, $G_0$).} $P$.
\end{defin}

\begin{rems}
\begin{enumerate}[(1)]
\item If we fix a right $G_0$-torsor $P$, there exist a perfect group $G_1$ over $\bF_q$ and a left action of $G_1$ on $P$ such that $P$ is a $(G_1,G_0)$-bitorsor. Moreover, the pair consisting of $G_1$ and this action is determined up to a unique isomorphism.
     \sbr
    Uniqueness follows from the observation that if $P$ is a $(G_1,G_0)$-bitorsor, then $G_1$ represents the functor that sends a perfect scheme $S$ over $\bF_q$ to the group of $S$-scheme automorphisms of $P\times S$ that commute with the right $G_0$-action.
     \sbr
    Existence follows from \cite[Lemma 6.3]{characters}.
 \sbr
\item Let $\al\in H^1(\bF_q,G_0)$, and choose a representative $P$ of the isomorphism class $\al$. As we just mentioned, we obtain a corresponding pure inner form $(G_1,P)$ of $G_0$. By abuse of notation, we will write $G_0^\al$ for $G_1$ and call it the \emph{pure inner form of $G_0$ defined by $\al$}. The class $\al$ only determines $G_0^\al$ up to an isomorphism that is unique up to inner automorphisms. However, our primary interest lies in the space of conjugation-invariant functions on the group $G_0^\al(\bF_q)$, which is canonically determined by $\al$.
 \sbr
\item In view of the first remark, the definition of a pure inner form of $G_0$ given above and the notation $G_0^\al$ agree with those used in \cite{characters}.
\end{enumerate}
\end{rems}

\begin{defin}\label{d:pure-inner-form-G-variety}
If $(G_1,P)$ is a pure inner form of a perfect group $G_0$ over $\bF_q$ and $X$ is a perfect variety over $\bF_q$ equipped with a left $G_0$-action, we write $X^\al=(P\times X)/G_0$, where $\al=[P]\in H^1(\bF_q,G_0)$ as before and the right $G_0$-action on $P\times X$ is given by $(p,x)\cdot g=(p\cdot g,g^{-1}\cdot x)$. Note that $X^\al$ carries a natural left action of $G_0^\al=G_1$.
\end{defin}

\subsection{An interpretation via gerbes}\label{ss:gerby-interpretation} In this subsection we explain a more elegant approach to the definition of pure inner forms and the other notions appearing in this section, which was communicated to us by A.~Beilinson.

\mbr

Let $\fG$ be an algebraic gerbe of finite type over $\bF_q$. In view of \cite[Cor.~6.4.2]{behrend-book}, one can use the following concrete definition: $\fG$ is an algebraic stack over $\bF_q$, which is isomorphic to the classifying stack of some group scheme of finite type over $\bF_q$.

\mbr

If we choose an object $x\in\fG(\bF_q)$ and let $G_0=\Aut_{\fG}(x)$, then $G_0$ is a group scheme of finite type over $\bF_q$ and $x$ determines an isomorphism $BG_0\rar{\sim}\fG$. Then the set of isomorphism classes of objects in $\fG(\bF_q)$ becomes identified with $H^1(\bF_q,G_0)$, and if $y\in\fG(\bF_q)$, then $G_1=\Aut_{\fG}(y)$ is a pure inner form of $G_0$ over $\bF_q$ whose class in $H^1(\bF_q,G_0)$ corresponds to the isomorphism class of $y$ in $\fG(\bF_q)$. Moreover, $\operatorname{Isom}_{\fG}(x,y)$ is naturally a $(G_1,G_0)$-bitorsor. Thus the datum of a triple $(G_0,G_1,P)$, where $G_0$ and $G_1$ are group schemes of finite type over $\bF_q$ and $P$ is a $(G_1,G_0)$-bitorsor, is equivalent to the datum of a triple $(\fG,x,y)$, where $\fG$ is an algebraic gerbe of finite type over $\bF_q$ and $x,y\in\fG(\bF_q)$ are objects.

\mbr

Definition \ref{d:pure-inner-form-G-variety} can also be rephrased in this language. To do so, let us introduce the following terminology for brevity. If $\fY$ is an algebraic stack, we call a \emph{scheme over $\fY$} an algebraic stack $\fX$ equipped with a morphism $\fX\rar{}\fY$ such that for every scheme $T$ the fiber product $\fX\times_{\fY}T$ is also a scheme. Now a scheme over $\bF_q$ with an action of $G_0$ is ``the same thing'' as a scheme over\footnote{Given an action of $G_0$ on a scheme $X$ over $\bF_q$, the quotient stack $G_0\backslash X$ is a scheme over $BG_0$. Conversely, if $\fX$ is a scheme over $BG_0$, then the fiber product of $\fX$ with the canonical morphism $x:\Spec\bF_q\rar{}BG_0$ is a scheme $X$ over $\bF_q$, and we get an action of $G_0=\Aut(x)$ on $X$.} $BG_0$. So with the notation of the previous paragraph, the objects $x$ and $y$ determine equivalences of categories
\[
\left\{
\begin{matrix}
\text{schemes with a } \\
G_0-\text{action}
\end{matrix}
\right\} \rar{\sim}
\left\{
\begin{matrix}
\text{schemes} \\
\text{over } \fG
\end{matrix}
\right\} \rar{\sim}
\left\{
\begin{matrix}
\text{schemes with a } \\
G_1-\text{action}
\end{matrix}
\right\}
\]
whose composition is the functor $X\mapsto(P\times X)/G_0$, where $P=\operatorname{Isom}_{\fG}(x,y)$.

\mbr

Throughout this section we will explain how the other constructions and results can be interpreted from the gerby viewpoint. However, for the reader's convenience, we will also indicate the more \emph{ad hoc} arguments, which avoid the language of stacks.

\subsection{Functoriality of pure inner forms}\label{ss:functoriality-inner-forms} Let $\vp:G_0'\rar{}G_0$ be a homomorphism of perfect groups over $\bF_q$. It induces a natural map $\vp_*:H^1(\bF_q,G_0')\rar{}H^1(\bF_q,G_0)$. If $\be\in H^1(\bF_q,G_0')$ and $\al=\vp_*(\be)$, we obtain the corresponding homomorphism $\vp^\be:G_0'^\be\rar{}G_0^\al$. If $\vp$ is a closed embedding, then so is $\vp^\be$. In the latter case, we tacitly view $G_0'$ as a closed subgroup of $G_0$, and $G_0'^\be$ as a closed subgroup of $G_0^\al$.

\mbr

Next suppose $X'$ (respectively, $X$) is a perfect variety over $\bF_q$ equipped with a left action of $G_0'$ (respectively, $G_0$), and let $f:X'\rar{}X$ be a $G_0'$-equivariant morphism (where the $G_0'$-action on $X$ is induced via $\vp$ by the action of $G_0$). Then we also obtain the corresponding $G_0'^\be$-equivariant morphism $f^\be:X'^\be\rar{}X^\al$.

\mbr

In the special case where $G_0'=G_0$ and $\vp$ is the identity, let us make a more precise statement. For any perfect group $G$ over a perfect field $k$, we write $G\var$ for the category of perfect varieties over $k$ equipped with a $G$-action. (The morphisms in this category are the $G$-equivariant morphisms of $k$-schemes.) The next lemma is a tautology from the gerby viewpoint of \S\ref{ss:gerby-interpretation}. However, it is also very easy to give a direct argument using Definitions \ref{d:pure-inner-form-group} and \ref{d:pure-inner-form-G-variety}, so we omit the proof.

\begin{lem}\label{l:pure-inner-forms-equivalence}
If $(G_1,P)$ is a pure inner form of a perfect group $G_0$ over $\bF_q$, the functor $X\longmapsto (P\times X)/G_0$ described in Definition \ref{d:pure-inner-form-G-variety} is an equivalence of categories
\[
G_0\var \rar{\sim} G_1\var.
\]
A quasi-inverse functor is given by $Y\mapsto G_1\setminus(Y\times P)$, where $G_1$ acts on $Y\times P$ diagonally and the $G_0$-action on the quotient is induced by $g\cdot(y,p)=(y,p\cdot g^{-1})$.
\end{lem}

\subsection{Transport of equivariant complexes}\label{ss:transport-equivariant-complexes} Let $(G_1,P)$ be a pure inner form of a perfect unipotent group $G_0$ over $\bF_q$, let $X_0\in G_0\var$, and define $X_1=(P\times X_0)/G_0$ using the construction of Definition \ref{d:pure-inner-form-G-variety}.

\begin{defin}\label{d:transport-equivariant-complexes}
We define a functor $\sD_{G_0}(X_0)\rar{\sim}\sD_{G_1}(X_1)$, which we call \emph{transport of equivariant complexes}, as follows. Let $\varpi:P\times X_0\rar{}X_1$ denote the quotient morphism, and let $\pi_2:P\times X_0\rar{}X_0$ denote the second projection. Observe that $\pi_2$ is a quotient map for the action of $G_1$ induced by the left $G_1$-action on $P$. Hence the functors $\pi_2^*:\sD_{G_0}(X_0)\rar{}\sD_{G_1\times G_0}(P\times X_0)$ and $\varpi^*:\sD_{G_1}(X_1)\rar{}\sD_{G_1\times G_0}(P\times X_0)$ are equivalences. We choose a quasi-inverse of $\varpi^*$ and define the transport functor $\sD_{G_0}(X_0)\rar{\sim}\sD_{G_1}(X_1)$ as the composition $(\varpi^*)^{-1}\circ \pi_2^*$.
\end{defin}

If $\al\in H^1(\bF_q,G_0)$ denotes the isomorphism class of $P$ (as a right $G_0$-torsor) and we write $X_0^\al=X_1$, as in Definition \ref{d:pure-inner-form-G-variety}, we will denote the transport functor by
\[
\sD_{G_0}(X_0) \rar{} \sD_{G_0^\al}(X_0^\al), \qquad M\longmapsto M^\al.
\]

\begin{rem}\label{r:transport-unique-up-to-isom}
The construction of the transport functor $M\longmapsto M^\al$ involves a choice of a representative $P$ of the class $\al$ as well as a choice of a quasi-inverse of the functor $\varpi^*$ in the definition above. Making different choices in those steps lead to a canonically isomorphic transport functor, so these choices are irrelevant.
\end{rem}

\begin{rem}\label{r:gerby-viewpoint-transport-functor}
It is also easy to define the transport functor using the viewpoint of \S\ref{ss:gerby-interpretation}. Namely, in the situation of Definition \ref{d:transport-equivariant-complexes}, the quotient stacks $G_0\backslash X_0$ and $G_1\backslash X_1$ can be naturally identified, which yields an equivalence $D^b_c(G_0\backslash X_0,\ql)\rar{\sim}D^b_c(G_1\backslash X_1,\ql)$, where we are using the $\ell$-adic formalism of \cite{Las-Ols06}. If we interpret the equivariant derived categories $\sD_{G_0}(X_0)$ and $\sD_{G_1}(X_1)$ as $D^b_c(G_0\backslash X_0,\ql)$ and $D^b_c(G_1\backslash X_1,\ql)$, respectively, we get an equivalence $\sD_{G_0}(X_0)\rar{\sim}\sD_{G_1}(X_1)$, which is the same as the one constructed in Definition \ref{d:transport-equivariant-complexes}.
\end{rem}

\subsection{Functoriality of transport}\label{ss:functoriality-transport}

\begin{lem}\label{l:functoriality-transport}
Let $G_0$ be a perfect unipotent group over $\bF_q$, and let $f:X\rar{}Y$ be a morphism in $G_0\var$. For each $\al\in H^1(\bF_q,G_0)$, the functors
\[
\sD_{G_0}(X) \rar{} \sD_{G_0^\al}(Y^\al), \qquad M\longmapsto (f_!M)^\al \quad\text{and}\quad M\longmapsto (f^\al)_!(M^\al),
\]
are isomorphic, and the functors
\[
\sD_{G_0}(Y) \rar{} \sD_{G_0^\al}(X^\al), \qquad N\longmapsto(f^*N)^\al \quad\text{and}\quad N\longmapsto (f^\al)^*(N^\al),
\]
are isomorphic.
\end{lem}

The lemma follows from the definitions and the proper base change theorem.

\subsection{The conjugation action}\label{ss:conjugation-action-pure-inner-forms}
Let $G_0$ be a perfect group over $\bF_q$. If we let $G_0$ act on itself by conjugation, $G_0$ becomes an object of the category $G_0\var$. The multiplication and inversion maps for $G_0$ are equivariant, and hence they endow $G_0$ with the structure of a group object of the category $G_0\var$.

\mbr

Next suppose $(G_1,P)$ is a pure inner form of $G_0$ and define $\cG_1=(P\times G_0)/G_0$, where $G_0$ acts on $P\times G_0$ on the right via $(p,x)\cdot g=(p\cdot g,g^{-1}xg)$. By Lemma \ref{l:pure-inner-forms-equivalence}, $\cG_1$ has a natural structure of a group object in the category $G_1\var$.

\mbr

Moreover, if $X_0\in G_0\var$ is arbitrary, the action morphism $G_0\times X_0\rar{}X_0$ is $G_0$-equivariant (where $G_0$ acts on $G_0\times X_0$ diagonally, the action on the first factor being the conjugation action). By Lemma \ref{l:pure-inner-forms-equivalence}, we obtain a natural action of $\cG_1$ on $X_1=(P\times X_0)/G_0$ in the category $G_1\var$. On the other hand, we can also view $G_1$ as a group object of $G_1\var$, and it acts on $X_1$ as well.

\begin{prop}\label{p:conjugation-action}
With the above notation, there exists a unique $G_1$-equivariant isomorphism $G_1\rar{\simeq}\cG_1$ of perfect groups over $\bF_q$ that is compatible with the actions of $G_1$ and $\cG_1$ on $X_1=(P\times X_0)/G_0$ for every $X_0\in G_0\var$.
\end{prop}

The proposition is proved in \S\ref{ss:proof-p:conjugation-action}. For us the important implication is that if $G_0$ is a perfect unipotent group over $\bF_q$ and $\al\in H^1(\bF_q,G_0)$, then the meaning of the notation $\sD_{G_0^\al}(G_0^\al)$ is unambiguous\footnote{Namely, we could consider the $G_0^\al$ inside the parentheses as the pure inner form of $G_0$ viewed either as a perfect group over $\bF_q$ or as an object of $G_0\var$ for the conjugation action of $G_0$ on itself; the resulting categories are canonically equivalent by Proposition \ref{p:conjugation-action}.}. Moreover, we have

\begin{cor}\label{c:transport-monoidal-equivalence}
If $G_0$ is a perfect unipotent group over $\bF_q$ and $\al\in H^1(\bF_q,G_0)$, the transport functor $\sD_{G_0}(G_0)\rar{\sim}\sD_{G_0^\al}(G_0^\al)$ has the structure of a \emph{monoidal} equivalence.
\end{cor}

\begin{proof}
With the earlier notation, the transport functor $\sD_{G_0}(G_0)\rar{\sim}\sD_{G_1}(\cG_1)$ is monoidal by Lemma \ref{l:functoriality-transport}. It remains to apply Proposition \ref{p:conjugation-action}.
\end{proof}

\begin{rem}
Proposition \ref{p:conjugation-action} and Corollary \ref{c:transport-monoidal-equivalence} are essentially trivial from the viewpoint of \S\ref{ss:gerby-interpretation}. Namely, let $\fG$ be an algebraic gerbe of finite type over $\bF_q$, and let $\fI$ be its inertia stack, i.e., the $2$-fiber product of the diagonal $\fG\rar{}\fG\times\fG$ with itself (cf.~Definition \ref{d:inertia-groupoid}). Then $\fI$ is a (relative) group scheme over $\fG$.

\mbr

We have the $\ell$-adic constructible derived category $D^b_c(\fI,\ql)$ together with an intrinsically defined operation of convolution with compact supports.

\mbr

Now choose $x\in\fG(\bF_q)$ and put $G_0=\Aut_{\fG}(x)$. We get an isomorphism $BG_0\rar{\sim}\fG$, so, as explained in \S\ref{ss:gerby-interpretation}, $\fI$ determines a group scheme over $\bF_q$ equipped with a $G_0$-action. That group scheme is nothing but $G_0$ with the conjugation action of $G_0$ on itself. Moreover, the convolution functor on $D^b_c(\fI,\ql)$ corresponds to the convolution functor on $\sD_{G_0}(G_0)$ defined in \S\ref{ss:recollections}.

\mbr

Next choose another $y\in\fG(\bF_q)$, let $G_1=\Aut_{\fG}(y)$, and write $P=\operatorname{Isom}_{\fG}(x,y)$, which is a $(G_1,G_0)$-bitorsor. Then $y$ yields an isomorphism $BG_1\rar{\sim}\fG$ and the group scheme $\cG_1$ mentioned at the beginning of \S\ref{ss:conjugation-action-pure-inner-forms} is the same as the group scheme with an action of $G_1$ coming from $\fI\rar{}BG_1$. In view of the previous paragraph, the assertions of Proposition \ref{p:conjugation-action} and Corollary \ref{c:transport-monoidal-equivalence} now become obvious.
\end{rem}

\subsection{Induction} As we saw earlier, an equivariant complex on a perfect group $G_0$ over $\bF_q$ determines an equivariant complex on each of its pure inner forms $G_0^\al$, and hence a class function on the corresponding group of rational points $G_0^\al(\bF_q)$. If we take this viewpoint on the sheaves-to-functions correspondence, the functor $\igz$ becomes compatible with induction of class functions in a natural sense:

\begin{prop}\label{p:induction-sheaves-functions}
Let $G_0$ be a perfect unipotent group over $\bF_q$, let $G_0'\subset G_0$ be a closed subgroup, fix $M\in\sD_{G'_0}(G'_0)$ and put $N=\igz M$. For each $\al\in H^1(\bF_q,G_0)$,
\[
t_{N^\al} = \sum_{\be\mapsto\al} \ind_{G_0'^\be(\bF_q)}^{G_0^\al(\bF_q)} t_{M^\be},
\]
where the sum is taken over all $\be\in H^1(\bF_q,G'_0)$ that map to $\al$ via the natural map $H^1(\bF_q,G_0')\rar{}H^1(\bF_q,G_0)$.
\end{prop}

This result is proved in \S\ref{ss:proof-p:induction-sheaves-functions}. The symbol $\ind_{G_0'^\be(\bF_q)}^{G_0^\al(\bF_q)}$ denotes the usual operation of induction of class functions from the subgroup $G_0'^\be(\bF_q)$ to the group $G_0^\al(\bF_q)$.

\begin{rem}
As a special case of the proposition, we note that if $\al\in H^1(\bF_q,G_0)$ is not in the image of $H^1(\bF_q,G_0')\rar{}H^1(\bF_q,G_0)$, then $t_{N^\al}\equiv 0$ on $G^\al_0(\bF_q)$.
\end{rem}

\begin{rem}\label{r:gerby-interpretation-induction}
Write $[G_0']$, $(\Ad G_0')\backslash G_0$ and $[G_0]$ for the quotient stacks of $G_0'$, $G_0$ and $G_0$ by the conjugation action of $G_0'$, $G_0'$ and $G_0$, respectively. Then the functor $\igz$ can be identified with the functor $\pi_!\circ\iota_!:D^b_c([G_0'],\ql)\rar{}D^b_c([G_0],\ql)$, where $\iota:[G_0']\rar{}(\Ad G_0')\backslash G_0$ and $\pi:(\Ad G_0')\backslash G_0\rar{}[G_0]$ are both induced by the natural inclusion $G_0'\into G_0$. So Proposition \ref{p:induction-sheaves-functions} could be deduced from the (relative version of) the Grothendieck-Lefschetz trace formula for algebraic stacks \cite[Thm.~4.2]{sun}. However, we prefer to give a direct proof.
\end{rem}

The functor of induction is also compatible with transport functors:

\begin{lem}\label{l:induction-transport}
If $G_0$ is a perfect unipotent group over $\bF_q$ and $G_0'\subset G_0$ is a closed subgroup, then for every $\be\in H^1(\bF_q,G_0')$, the functors
\[
\sD_{G_0'}(G_0') \rar{} \sD_{G_0^\al}(G_0^\al), \qquad M\longmapsto (\igz M)^\al \quad\text{and}\quad M\longmapsto \ind_{G_0'^\be}^{G_0^\al} (M^\be),
\]
are isomorphic, where $\al\in H^1(\bF_q,G_0)$ is the image of $\be$.
\end{lem}

The proof is given in \S\ref{ss:proof-l:induction-transport}. We remark that from the viewpoint of \S\ref{ss:gerby-interpretation}, this result is obvious (use the gerby interpretations explained in Remarks \ref{r:gerby-interpretation-induction} and \ref{r:gerby-viewpoint-transport-functor}).

\subsection{Compatibility with duality} The next result is proved in \S\ref{ss:proof-p:transport-duality}.

\begin{prop}\label{p:transport-duality}
Let $G_0$ be a perfect unipotent group over $\bF_q$.
 \sbr
\begin{enumerate}[$($a$)$]
\item If $X\in G_0\var$, then for each $\al\in H^1(\bF_q,G_0)$, the functors
    \[
    \sD_{G_0}(X)^\circ\rar{\sim}\sD_{G_0^\al}(X^\al), \qquad M\longmapsto (\bD_X M)^\al \quad\text{and}\quad M\longmapsto \bD_{X^\al}(M^\al),
    \]
    are isomorphic.
 \sbr
\item For each $\al\in H^1(\bF_q,G_0)$, the functors
\[
\sD_{G_0}(G_0)^\circ\rar{\sim}\sD_{G_0^\al}(G_0^\al), \qquad M\longmapsto (\bD_{G_0}^- M)^\al \quad\text{and}\quad M\longmapsto\bD_{G_0^\al}^-(M^\al),
\]
are isomorphic.
\end{enumerate}
\end{prop}

\subsection{Compatibility with weights} The next result is proved in \S\ref{ss:proof-p:transport-weights}.

\begin{prop}\label{p:transport-weights}
Let $G_0$ be a perfect unipotent group acting on a perfect variety $X$ over $\bF_q$, and let $M\in\sD_{G_0}(X)$. If $w\in\bZ$ and $M\in\sD_{\leq w}(X)$ $($resp., $M\in\sD_{\geq w}(X)${}$)$, then $M^\al\in\sD_{\leq w}(X^\al)$ $($resp., $M^\al\in\sD_{\geq w}(X^\al)${}$)$ for all $\al\in H^1(\bF_q,G_0)$.
\end{prop}

\subsection{Equivariant sheaves on a point}\label{ss:equivariant-sheaves-point} We now illustrate some of the notions discussed in this section with a specific example, which will be used in \S\S\ref{ss:proof-t:main-b}--\ref{ss:proof-orthonormality}. We let our base field be $\bF_q$ and consider $\ell$-adic complexes and equivariant complexes (with respect to some perfect unipotent group $G_0$) on the ``point'' $\Spec\bF_q$. A gerby interpretation of the material we present is explained in Remark \ref{r:gerby-interpretation-equivariant-sheaves-point}. As we will see, what we are considering here is nothing but a Grothendieck-Lefschetz trace formula for the classifying stack of a unipotent algebraic group over $\bF_q$.

\mbr

Objects of $\sD(\Spec\bF_q)$ can be viewed as bounded complexes of finite dimensional $\ql$-vector spaces equipped with a continuous action of $\Gal(\bF/\bF_q)$. If $V^\bullet$ is such a complex, the corresponding ``trace-of-Frobenius function'' on the one-point set $(\Spec\bF_q)(\bF_q)$ takes the value $\sum_{j\in\bZ} (-1)^j \cdot \tr\bigl(F_q;H^j(V^\bullet)\bigr)$, where $F_q\in\Gal(\bF/\bF_q)$ is the geometric Frobenius (the inverse of the automorphism $a\mapsto a^q$ of $\bF$).

\mbr

Now suppose $G_0$ is a perfect unipotent group over $\bF_q$, acting trivially on $\Spec\bF_q$.
Write $G_0^\circ$ for the neutral connected component of $G_0$ and $\pi_0(G_0)=G_0/G_0^\circ$
for the group of components of $G_0$ (a finite \'etale group scheme over $\bF_q$).
Set $\Ga=\pi_0(G_0)(\bF)$. It is a finite (discrete) group equipped with a (continuous) action
of $\Gal(\bF/\bF_q)$. Objects of $\sD_{G_0}(\Spec\bF_q)$ can be viewed as bounded complexes
of finite dimensional $\ql$-vector spaces equipped with a continuous action of the semidirect
product $\widetilde{\Ga}:=\Gal(\bF/\bF_q)\ltimes\Ga$. Let us describe the transport functor in this setting.

\mbr

By Lang's theorem \cite{lang}, the natural map $H^1(\bF_q,G_0)\rar{}H^1(\bF_q,\pi_0(G_0))$ is bijective. One can naturally identify $H^1(\bF_q,\pi_0(G_0))$ with the group cohomology set $H^1(\Gal(\bF/\bF_q),\Ga)$. In turn, since $\Gal(\bF/\bF_q)$ is topologically isomorphic to $\widehat{\bZ}$ via the map taking $F_q\in\Gal(\bF/\bF_q)$ to $1\in\widehat{\bZ}$, we can identify $H^1(\Gal(\bF/\bF_q),\Ga)$ with the set of $\Ga$-conjugacy classes in $\widetilde{\Ga}$ that project to $F_q\in\Gal(\bF/\bF_q)$ under the canonical homomorphism $\widetilde{\Ga}\rar{}\Gal(\bF/\bF_q)$. Fix $\ga\in\Ga$ and let $\al\in H^1(\bF_q,G_0)$ correspond to the conjugacy class of $\ga\cdot F_q\in\widetilde{\Ga}$. One can then identify $\pi_0(G_0^\al)$ with the same group $\Ga$ equipped with the ``new'' action of the geometric Frobenius, given by $x\mapsto \ga F_q(x)\ga^{-1}$, where $F_q$ is the original action of the geometric Frobenius on $\Ga$.

\mbr

In this language, if $M\in\sD_{G_0}(\Spec\bF_q)$ corresponds to a bounded complex $V^\bullet$ of finite dimensional $\ql$-vector spaces equipped with a continuous action of $\widetilde{\Ga}$, then $M^\al$ corresponds to the same complex with the same action of $\widetilde{\Ga}$, but now $\widetilde{\Ga}$ is identified with the semidirect product $\Gal(\bF/\bF_q)\ltimes\Ga$ in a different way, where the geometric
Frobenius in $\Gal(\bF/\bF_q)$ corresponds to the element $\ga\cdot F_q$.

\begin{prop}[see \S\ref{ss:proof-p:equivariant-sheaves-point}]\label{p:equivariant-sheaves-point}
In the situation above, we have
\begin{equation}\label{e:equivariant-sheaves-point}
\sum_{\al\in H^1(\bF_q,G_0)} \frac{q^{\dim G_0}}{\abs{G_0^\al(\bF_q)}} \cdot t_{M^\al}(*) = \sum_{j\in\bZ} (-1)^j\cdot \tr\bigl( F_q; H^j(V^\bullet)^\Ga \bigr),
\end{equation}
where $t_{M^\al}$ is the trace-of-Frobenius function corresponding to the complex $M^\al$, the symbol $*$ denotes the unique $\bF_q$-point of $\Spec\bF_q$, and $H^j(V^\bullet)^\Ga$ denotes the subspace of $H^j(V^\bullet)$ consisting of $\Ga$-invariant elements.
\end{prop}

\begin{rem}\label{r:gerby-interpretation-equivariant-sheaves-point}
Let $\fG$ be an algebraic gerbe of finite type over $\bF_q$, and assume that the automorphism group of any object of $\fG(\bF_q)$ is \emph{unipotent}. Write $\widetilde{\Ga}$ and $\Ga$ for the \'etale fundamental groups of $\fG$ and $\fG\tens_{\bF_q}\bF$, respectively. Then $\widetilde{\Ga}$ is an extension of $\widehat{\bZ}=\Gal(\bF/\bF_q)$ by $\Ga$, and objects of $D^b_c(\fG,\ql)$ can be viewed as bounded complexes
of finite dimensional $\ql$-vector spaces equipped with a continuous action of $\widetilde{\Ga}$.

\mbr

Choose $x\in\fG$ and put $G_0=\Aut_{\fG}(x)$, which by assumption is a unipotent algebraic group over $\bF_q$. Then $x$ identifies $\fG$ with $BG_0$, the group $\Ga$ with $\pi_0(G_0)(\bF)$, and determines a splitting of the extension $1\to\Ga\to\widetilde{\Ga}\to\widehat{\bZ}\to 1$, which allows us to identify $\widetilde{\Ga}$ with the semidirect product $\Gal(\bF/\bF_q)\ltimes\Ga$. If we interpret $\sD_{G_0}(\Spec\bF_q)$ as $D^b_c(\fG,\ql)$, then formula \eqref{e:equivariant-sheaves-point} becomes the Grothendieck-Lefschetz trace formula \cite{sun} for an object $M\in D^b_c(\fG,\ql)$. In fact, since $D^b_c(BG_0,\ql)$ can be identified with $D^b_c(B\pi_0(G_0),\ql)$ by the unipotence of $G_0$, we can even use the earlier trace formula \cite[Cor.~6.4.10]{behrend-book} for the finite stack $B\pi_0(G_0)$.
\end{rem}


\section{Weil formalism}\label{s:weil}

\subsection{Notation}\label{ss:weil-definitions}
We fix an algebraic closure $\bF$ of a field of prime order $p$ and a finite subfield $\bF_q\subset\bF$. We also fix a perfect variety $X_0$ over $\bF_q$, write $X=X_0\tens_{\bF_q}\bF$, and let $\Fr_q:X\rar{}X$ denote the Frobenius endomorphism\footnote{It is in fact an automorphism because $X_0$ is assumed to be perfect.}, obtained from the absolute Frobenius $X_0\rar{}X_0$ by extension of scalars.

\mbr

Let $\sD^{Weil}(X_0)$ be the category consisting of pairs $(M,\vp)$, where $M\in\sD(X)$ and $\vp:\Fr_q^*M\rar{\simeq}M$ is an isomorphism in $\sD(X)$. Morphisms and compositions in $\sD^{Weil}(X_0)$ are defined in the obvious way. (The category $\sD^{Weil}(X_0)$ is very naive, but suffices for the purposes of this article.) One has a natural functor $\sD(X_0)\rar{}\sD^{Weil}(X_0)$, which is neither faithful nor essentially surjective.

\mbr

Let $\Perv^{Weil}(X_0)$ denote the full subcategory of $\sD^{Weil}(X_0)$ consisting of pairs $(M,\vp)$ for which $M$ is a perverse sheaf on $X$. In this setting we have

\begin{lem}\label{l:perverse-weil}
The natural functor $\Perv(X_0)\rar{}\Perv^{Weil}(X_0)$ is fully faithful.
\end{lem}

\begin{proof}
This is the first assertion of \cite[Prop.~5.1.2]{bbd}.
\end{proof}

\subsection{The formalism} In what follows we will tacitly use the following observation. The formalism of the six functors for the categories $\sD(X_0)$ extends to the categories $\sD^{Weil}(X_0)$ without any changes in the notation. In particular, given an $\bF_q$-morphism $f:X_0\rar{}Y_0$ of perfect varieties over $\bF_q$, we have the induced functors
\[
f_*,f_!:\sD^{Weil}(X_0)\rar{}\sD^{Weil}(Y_0) \qquad\text{and}\qquad f^*,f^!:\sD^{Weil}(Y_0)\rar{}\sD^{Weil}(X_0),
\]
as well as the Verdier duality functor $\bD_{X_0}:\sD^{Weil}(X_0)^\circ\rar{\sim}\sD^{Weil}(X_0)$, and all the standard adjunctions hold in this context.

\subsection{Morphism spaces}\label{ss:derived-homs} When we deal with spaces of morphisms we depart from our usual convention of omitting the letters `L' and `R'. Thus, if $A$ and $B$ are object of any category, $\Hom(A,B)$ always denotes the space of morphisms $A\rar{}B$.

\mbr

If $M_0$ and $N_0$ are objects of $\sD(X_0)$, the derived Hom between them will be denoted by $R\Hom_{\sD(X_0)}(M_0,N_0)$. We will also need a version of $R\Hom$ that ``remembers the Frobenius action.'' To define it, let $(M,\vp)$ and $(N,\psi)$ be the objects of $\sD^{Weil}(X_0)$ corresponding to $M_0$ and $N_0$.

\mbr

We can view $R\Hom_{\sD(X)}(M,N)$ as an object of $\sD(\Spec\bF)$, and as such it can be ``upgraded'' to an object of $\sD^{Weil}(\Spec\bF_q)$ using the composition
\[
R\Hom_{\sD(X)}(M,N) \rar{\Fr_q^*} R\Hom_{\sD(X)}(\Fr_q^*M,\Fr_q^*N) \rar{\sim} R\Hom_{\sD(X)}(M,N),
\]
where the second quasi-isomorphism is induced by $\vp$ and $\psi$. We denote this object of $\sD^{Weil}(\Spec\bF_q)$ by $R\Hom_{\sD(X_0)}^{Weil}(M_0,N_0)$.

\begin{rem}
The object $R\Hom_{\sD(X_0)}^{Weil}(M_0,N_0)\in\sD^{Weil}(\Spec\bF_q)$ can also be obtained from an object of $\sD(\Spec\bF_q)$, for example using the construction mentioned in the proof of \cite[Prop.~5.1.2]{bbd}. Namely, let $a:X_0\rar{}\Spec\bF_q$ denote the structure morphism, and consider the inner Hom $R\underline{\Hom}(M_0,N_0)\in\sD(X_0)$. The pushforward $Ra_*R\underline{\Hom}(M_0,N_0)\in\sD(\Spec\bF_q)$ gives rise to $R\Hom_{\sD(X_0)}^{Weil}(M_0,N_0)$.
\end{rem}

\subsection{Equivariant version}\label{ss:weil-equivariant} In what follows we will need an extension of the earlier discussion to the equivariant setting. Let $G_0$ be a perfect unipotent group over $\bF_q$ acting on $X_0$. Note that now $\Fr_q$ could stand either for the Frobenius endomorphism of $X$ or for the Frobenius endomorphism of $G=G_0\tens_{\bF_q}\bF$, but in practice this conflict of notation should not cause any confusion.

\mbr

We let $\sD_{G_0}^{Weil}(X_0)$ be the category consisting of pairs $(M,\vp)$, where $M\in\sD_G(X)$ and $\vp:\Fr_q^*M\rar{\simeq}M$ is an isomorphism in $\sD_G(X)$. As before, we have a natural functor $\sD_{G_0}(X_0)\rar{}\sD_{G_0}^{Weil}(X_0)$, which is neither faithful nor essentially surjective.

\mbr

The formalism of the six functors also extends to the categories $\sD_{G_0}^{Weil}(X_0)$. The construction of \S\ref{ss:derived-homs} can be adapted to the present situation to yield the definition of $R\Hom_{\sD_{G_0}(X_0)}^{Weil}(M,N)\in\sD^{Weil}(\Spec\bF_q)$ for any $M,N\in\sD_{G_0}^{Weil}(X_0)$.

\mbr

Definition \ref{d:transport-equivariant-complexes} can be repeated in the ``Weil setting'' and yields a transport functor $\sD_{G_0}^{Weil}(X_0)\rar{\sim}\sD_{G_0^\al}^{Weil}(X_0^\al)$ for each $\al\in H^1(\bF_q,G_0)$. The sheaves-to-functions correspondence makes sense for objects of $\sD^{Weil}(X_0)$ and Proposition \ref{p:equivariant-sheaves-point} remains true for objects of $\sD_{G_0}^{Weil}(\Spec\bF_q)$ with essentially the same proof.

\mbr

Next define $\Perv_{G_0}(X_0)\subset\sD_{G_0}(X_0)$ and $\Perv_{G_0}^{Weil}(X_0)\subset\sD_{G_0}^{Weil}(X_0)$ to be the full subcategories consisting of objects whose underlying complex is perverse.

\begin{lem}\label{l:equivariant-perverse-weil}
The natural functor $\Perv_{G_0}(X_0)\rar{}\Perv_{G_0}^{Weil}(X_0)$ is fully faithful.
\end{lem}

\begin{proof}
By construction, the forgetful functor $\Perv_{G_0}^{Weil}(X_0)\rar{}\Perv^{Weil}(X_0)$ is faithful. Hence by Lemma \ref{l:perverse-weil} the natural functor $\Perv_{G_0}(X_0)\rar{}\Perv_{G_0}^{Weil}(X_0)$ is also faithful. Next suppose $M_0,N_0\in\Perv_{G_0}(X_0)$ are objects, and let $f:M\rar{}N$ be a morphism between the corresponding objects of $\Perv_{G_0}^{Weil}(X_0)$. By Lemma \ref{l:perverse-weil}, we know that $f$ comes from a unique morphism $f_0:M_0\rar{}N_0$ in $\Perv(X_0)$, and we only need to check that $f_0$ is $G_0$-equivariant.

\mbr

Let $\al:G_0\times X_0\rar{}X_0$ and $\pi:G_0\times X_0\rar{}X_0$ be the action morphism and the projection, respectively. The $G_0$-equivariant structures on $M_0$ and $N_0$ yield identifications $\al^*M_0\rar{\simeq}\pi^*M_0$ and $\al^*N_0\rar{\simeq}\pi^*N_0$. We must show that $\al^*(f_0)=\pi^*(f_0)$ modulo these identifications. By assumption, $\al^*(f)=\pi^*(f)$. It remains to observe that since $\al$ and $\pi$ are smooth morphisms of relative dimension $d=\dim G$, the objects $\al^*M_0,\pi^*M_0,\al^*N_0,\pi^*N_0\in\sD(G_0\times X_0)$ are perverse sheaves up to cohomological shift by $d$. Applying Lemma \ref{l:perverse-weil} once more finishes the proof.
\end{proof}


\section{Proofs of main results}\label{s:proofs}

We prove Theorem \ref{t:main}(a) in \S\ref{ss:proof-t:main-a}, Theorem \ref{t:main}(b) in \S\ref{ss:proof-t:main-b}, and the orthonormality relations for character sheaves (cf.~Theorem \ref{t:main}(e)) in \S\S\ref{ss:gabber}--\ref{ss:proof-orthonormality}. Using the ``geometric Mackey theory'' recalled in \S\ref{ss:geometric-mackey-theory}, we reduce the proof of Theorem \ref{t:main}(c) to the ``Heisenberg case'' (Proposition \ref{p:heis}) in \S\ref{ss:proof-t:main-c}. We state some results on equivariant local systems on homogeneous spaces under unipotent groups in \S\ref{ss:homogeneous-spaces} and use them to complete the proof of Theorem \ref{t:main}(d) in \S\ref{ss:proof-t:main-d}. Theorem \ref{t:main-easy} is proved in \S\ref{ss:proof-t:main-easy}. Finally, Proposition \ref{p:heis} is established in \S\ref{ss:proof-p:heis}.
Certain technical details of the arguments appearing in this section have been relegated to the Appendix.

\subsection{Proof of Theorem \ref{t:main}(a)}\label{ss:proof-t:main-a}
We begin with the uniqueness assertion:

\begin{lem}\label{l:uniqueness-idempotent}
Let $G_0$ be a perfect unipotent group over $\bF_q$, and let $G=G_0\tens_{\bF_q}\bF$. If $e_0,\widetilde{e}_0\in\sD_{G_0}(G_0)$ are weak idempotents such that the base changes of $e_0$ and $\widetilde{e}_0$ to $\bF$ are isomorphic to the same minimal idempotent $e\in\sD_G(G)$, then $e_0\cong \widetilde{e}_0$.
\end{lem}

\begin{proof}
By Proposition \ref{p:properties-char-sheaves}(c), both $e_0[-n_e]$ and $\widetilde{e}_0[-n_e]$ are perverse. Hence, in view of Lemma \ref{l:equivariant-perverse-weil}, it suffices to show that if $E_0,\widetilde{E}_0\in\sD_{G_0}^{Weil}(G_0)$ are the objects corresponding to $e_0,\widetilde{e}_0$, then $E_0$ and $\widetilde{E}_0$ are isomorphic.

\mbr

Now $E_0$ and $\widetilde{E}_0$ both have the form $(e,\Fr_q^*e\rar{\simeq}e)$ for two possibly different isomorphisms $\Fr_q^*e\rar{\simeq}e$. By parts (a) and (c) of Proposition \ref{p:properties-char-sheaves} and Schur's lemma, any two isomorphisms $\Fr_q^*e\rar{\simeq}e$ are proportional. Hence there is a scalar $\la\in\ql^\times$ such that if $\sL_\la\in\sD^{Weil}_{G_0}(\Spec\bF_q)$ corresponds to the $1$-dimensional vector space on which the Frobenius acts via $\la$ and the $G_0$-equivariant structure is trivial, then $\widetilde{E}_0\cong E_0\tens\pr^*\sL_\la$, where $\pr:G_0\rar{}\Spec\bF_q$ is the structure morphism. Thus
\begin{eqnarray*}
\widetilde{E}_0 &\cong& \widetilde{E}_0*\widetilde{E}_0\cong (E_0\tens\pr^*\sL_\la)*(E_0\tens\pr^*\sL_\la) \\
&\cong& (E_0*E_0)\tens\pr^*\sL_\la^{\tens 2} \cong E_0\tens\pr^*\sL_\la^{\tens 2} \cong \widetilde{E}_0\tens\pr^*\sL_\la.
\end{eqnarray*}
Tensoring both sides of the isomorphism $\widetilde{E}_0\cong\widetilde{E}_0\tens\pr^*\sL_\la$ with the inverse of $\pr^*\sL_\la$ yields $E_0\cong\widetilde{E}_0$, as required.
\end{proof}

Now with the notation of Theorem \ref{t:main}, fix a minimal idempotent $e\in\sD_G(G)$ such that $\Fr_q^*(e)\cong e$. By \cite[Thm.~1.41(c)]{foundations}, $e$ comes from some admissible pair $(H,\cL)$ for $G$ defined over $\bF$. But $(H,\cL)$ must already be defined over a finite extension $\bF_q\subset\bF_{q^n}\subset\bF$. So if we write $G_1=G_0\tens_{\bF_q}\bF_{q^n}$, then $e$ comes from a minimal closed idempotent $e_1\in\sD_{G_1}(G_1)$.

\mbr

By a slight abuse of notation, we also write $\Fr_q=\Phi_q\tens\id:G_1\rar{}G_1$ for the morphism obtained from the absolute Frobenius morphism $G_0\rar{}G_0$ by base change. Then $\Fr_q^*(e_1)$ and $e_1$ are both closed idempotents in $\sD_{G_1}(G_1)$ whose base change to $\bF$ is isomorphic to $e$. Applying Lemma \ref{l:uniqueness-idempotent} to $G_1$ in place of $G_0$, we obtain an isomorphism $\vp:\Fr_q^*(e_1)\rar{\simeq}e_1$ in $\sD_{G_1}(G_1)$.

\mbr

Choose an idempotent arrow $\pi_1:\e\rar{}e_1$. Then $\vp\circ\Fr_q^*(\pi_1):\e\rar{}e_1$ is another idempotent arrow, where we tacitly identify $\Fr_q^*(\e)$ with $\e$. By \cite[Cor.~2.40]{foundations}, there exists a unique isomorphism $\sg:e_1\rar{\simeq}e_1$ such that $\pi_1=\sg\circ\vp\circ\Fr_q^*(\pi_1)$. Replacing $\vp$ with $\sg\circ\vp$, we may and do assume that $\vp\circ\Fr_q^*(\pi_1)=\pi_1$.

\mbr

Next observe that $\Fr_q^n:G_1\rar{}G_1$ is equal to the absolute Frobenius morphism for $G_1$ over $\bF_{q^n}$. In particular, we have a canonical isomorphism $(\Fr_q^n)^*(e_1)\rar{\simeq}e_1$ compatible with $\pi_1$. By the uniqueness part of \cite[Cor.~2.40]{foundations}, the composition
\[
\vp\circ\Fr_q^*(\vp)\circ\dotsb\circ(\Fr_q^{n-1})^*(\vp) : (\Fr_q^n)^*(e_1)\rar{\simeq}e_1
\]
is equal to the aforementioned canonical isomorphism. This means that $\vp$ defines a descent datum for $e_1$ relative to the finite \'etale cover $\pr:G_1\rar{}G_0$. Hence $e_1$ comes from an object $e_0\in\sD_{G_0}(G_0)$ and $\pi_1$ comes from an idempotent arrow $\pi_0:\e\rar{}e_0$. To construct $e_0$ and $\pi_0$ explicitly, observe that $\Gal(\bF_{q^n}/\bF_q)$ naturally acts\footnote{The action of the Frobenius in $\Gal(\bF_{q^n}/\bF_q)$ comes from the isomorphism $\vp$.} on $\pr_*(e_1)=\pr_!(e_1)$, take $e_0=\pr_*(e_1)^{\Gal(\bF_{q^n}/\bF_q)}$ and let $\pi_0:\e\rar{}e_0$ be the arrow coming from $\pr_*(\pi_1):\pr_*(\e)\rar{}\pr_*(e_1)$.

\mbr

Thus $e_0\in\sD_{G_0}(G_0)$ is a closed idempotent giving rise to $e$, as claimed.

\subsection{Proof of Theorem \ref{t:main}(b)}\label{ss:proof-t:main-b} We keep the notation of Theorem \ref{t:main}(a). Let $e_0\in\sD_{G_0}(G_0)$ be the closed idempotent constructed in \S\ref{ss:proof-t:main-a}. We must show that there exists $\al\in H^1(\bF_q,G_0)$ such that $e_0^\al\in\sD_{G_0^\al}(G_0^\al)$ comes from an admissible pair for $G_0^\al$ defined over $\bF_q$. By Proposition \ref{p:when-min-idemp-comes-from-adm-pair}, it is enough to prove

\begin{prop}\label{p:implies-main-b}
Given an object $E_0=(e,\vp)\in\sD_{G_0}^{Weil}(G_0)$ such that $e$ is a minimal idempotent in $\sD_G(G)$, there exists $\al\in H^1(\bF_q,G_0)$ such that $t_{E_0^\al}(1)\neq 0$, where $E_0^\al$ denotes the image of $E_0$ in $\sD_{G_0^\al}^{Weil}(G_0^\al)$ under the corresponding transport functor and $t_{E_0^\al}$ is the associated ``trace-of-Frobenius'' function $($cf.~\S\ref{ss:weil-equivariant}$)$.
\end{prop}

\begin{proof}
Let us view $1$ as an $\bF_q$-morphism $\Spec\bF_q\rar{}G_0$ and form the pullback $1^*E_0\in\sD_{G_0}^{Weil}(\Spec\bF_q)$. We can apply Proposition \ref{p:equivariant-sheaves-point} (or, rather, its analogue in the Weil setting) to the object $1^*E_0$. Thus we need to compute the $G$-invariants $(1^*E_0)^G$ as an object of $\sD^{Weil}(\bF_q)$. Writing $\bD$ for the Verdier duality functor on $\Spec\bF_q$ and using the standard adjunctions we obtain a chain of isomorphisms
\begin{equation}\label{e:key-1}
\begin{split}
\bD\bigl((1^*E_0)^G\bigr) &\cong \bigl( \bD(1^*E_0) \bigr)^G \cong \bigl( 1^!\bD_{G_0}^-E_0 \bigr)^G \cong R\Hom_{\sD_{G_0}(\Spec\bF_q)}^{Weil}(\ql,1^!\bD_{G_0}^-E_0) \\
&\cong R\Hom_{\sD_{G_0}(G_0)}^{Weil}(\e,\bD_{G_0}^-E_0).
\end{split}
\end{equation}

We now pause to prove an auxiliary result.

\begin{prop}\label{p:hom-1-to-min-idemp}
If $e\in\sD_G(G)$ is any minimal idempotent, then $\Hom_{\sD_G(G)}(\e,e)$ is
$1$-dimensional and $\Hom_{\sD_G(G)}(\e,e[r])=0$ for all $r\in\bZ$ such that $r\neq 0$.
\end{prop}

\begin{proof}
The functor $\sD_G(G)\rar{}e\sD_G(G)$ given by $M\mapsto e*M$ is left adjoint to the inclusion functor $e\sD_G(G)\into\sD_G(G)$ by \cite[Prop.~2.22(a)]{foundations}. In particular, we have $\Hom_{\sD_G(G)}(\e,e[r])\cong\Hom_{e\sD_G(G)}(e,e[r])$ for all $r\in\bZ$. By Proposition \ref{p:properties-char-sheaves}, the category $\sM_e$ is semisimple, $e$ is a simple object of $\sM_e$, and the natural functor $D^b(\sM_e)\rar{}e\sD_G(G)$ is an equivalence. This implies the proposition.
\end{proof}

\begin{cor}
The total cohomology of the complex $(1^*E_0)^G$ $($where we disregard the Frobenius action$)$ is $1$-dimensional.
\end{cor}

\begin{proof}
In view of \eqref{e:key-1}, this is the same as showing that the total cohomology of $R\Hom_{\sD_G(G)}(\e,\bD_G^-e)$ is $1$-dimensional. However, $\bD_G^-e\cong e[-2n_e]$ by Proposition \ref{p:properties-char-sheaves}. Hence the desired assertion follows from the proposition above.
\end{proof}

We now observe that the trace of an automorphism of a $1$-dimensional vector space cannot be zero. Hence if we apply Proposition \ref{p:equivariant-sheaves-point} to the object $M=1^*E_0\in\sD_{G_0}^{Weil}(\Spec\bF_q)$, the sum on the right hand side of \eqref{e:equivariant-sheaves-point} is nonzero thanks to the corollary we just proved. It follows that $t_{E_0^\al}(1)\neq 0$ for some $\al\in H^1(\bF_q,G_0)$.
\end{proof}

\subsection{Application of a result of Gabber}\label{ss:gabber} We now make a short digression to prove the following fact, which will be used in our demonstration of Theorem \ref{t:main}(e). Let $\overline{\bQ}$ be the algebraic closure of $\bQ$ in $\ql$, fix an embedding $\iota:\overline{\bQ}\into\bC$ and let $z\mapsto\overline{z}$ denote the automorphism of $\overline{\bQ}$ corresponding to complex conjugation via $\iota$.

\begin{lem}\label{l:dual-complex-conjugate}
Let $X$ be a scheme of finite type over $\bF_q$. If $\cF$ is a pure perverse sheaf of weight $0$ on $X$, then $t_{\bD_X\cF}=\overline{t_{\cF}}$, where $\bD_X$ is the Verdier duality functor.
\end{lem}

\begin{rem}
Since $\cF$ is pure, the trace-of-Frobenius function $t_{\cF}$ takes values in the maximal CM extension of $\bQ$ inside $\overline{\bQ}$. On that subfield, the automorphism $z\mapsto\overline{z}$ is independent of the choice of $\iota$, which is consistent with the equality $t_{\bD_X\cF}=\overline{t_{\cF}}$.
\end{rem}

\begin{proof}[Proof of Lemma \ref{l:dual-complex-conjugate}]
Without loss of generality we may assume that $\cF$ is simple. Then by \cite[Cor.~III.5.5]{kiehl-weissauer}, $\cF\cong j_{!*}\bigl(\cL[\dim U](\dim U/2)\bigr)$ for a locally closed embedding $j:U\into X$ and a pure local system $\cL$ on $U$ of weight $0$ such that $U_{red}$ (the reduced scheme associated to $U$) is smooth over $\bF_q$,

\mbr

We have $\bD_X(\cF)\cong j_{!*}\bigl(\cL^\vee[\dim U](\dim U/2)\bigr)$, where $\cL^\vee$ is the dual local system. Since $\cL$ is pure of weight $0$, we have $t_{\cL^\vee\tensqn}=\overline{t_{\cL\tensqn}}$ for all $n\in\bN$. We will deduce from this that $t_{\bD_X\cF}=\overline{t_{\cF}}$ using a result of O.~Gabber given in \cite{fujiwara-gabber-independence}.

\mbr

To this end, let $\iota_1:\overline{\bQ}\into\ql$ be the inclusion map, and let $\iota_2:\overline{\bQ}\into\ql$ be given by $\iota_2(z)=\iota_1(\overline{z})$. Let $I$ be the set consisting of the two pairs $(\ell,\iota_1)$ and $(\ell,\iota_2)$. Then the pure perverse sheaves $\cL[\dim U](\dim U/2)$ and $\cL^\vee[\dim U](\dim U/2)$ on $U$ form a $(\overline{\bQ},I)$-compatible system in the terminology of \cite[\S1.2]{fujiwara-gabber-independence}. By \cite[Thm.~3]{fujiwara-gabber-independence}, $\cF$ and $\bD_X\cF$ also form a $(\overline{\bQ},I)$-compatible system, proving the lemma.
\end{proof}

\subsection{Proof of the orthonormality relations}\label{ss:proof-orthonormality}
In this subsection we assume that Theorem \ref{t:main}(c) holds (it is proved in \S\ref{ss:proof-t:main-c}).
For each $M\in\cs(G)^{\Fr_q^*}$ we fix a choice of $M_0\in\sD_{G_0}(G_0)$ satisfying its requirements. We will show how the orthonormality assertion of Theorem \ref{t:main}(e) can be deduced from these properties. In particular, it will follow that the elements $\bigl\{T_{M_0}\st M\in\cs(G)^{\Fr_q^*}\bigr\}$ are linearly independent. The fact that they span $\Fun\bigl([G_0](\bF_q)\bigr)$ follows from Theorem \ref{t:main}(d), which is proved in \S\ref{ss:proof-t:main-d}, along with Proposition \ref{p:equivalence-def-L-packet-characters}.

\mbr

By Lemma \ref{l:dual-complex-conjugate} and Proposition \ref{p:transport-duality}(a), we have
\begin{equation}\label{e:dual-complex-conjugate}
t_{(\bD_{G_0}M_0)^\al}=\overline{t_{M_0^\al}} \quad\text{for all}\quad \al\in H^1(\bF_q,G_0).
\end{equation}

\begin{rem}
Before delving into the remainder of the proof, let us explain the interpretation of the argument we will give from the viewpoint of the Grothendieck-Lefschetz trace formula for stacks. For each $M\in\cs(G)^{\Fr_q^*}$ we view $M_0$ as an object of $D^b_c([G_0],\ql)$, where $[G_0]$ is the stack quotient of $G_0$ by the adjoint action of $G_0$ on itself. If $\pr:[G_0]\rar{}\Spec\bF_q$ is the structure morphism, it turns out that given $M,N\in\cs(G)^{\Fr_q^*}$, we have $\pr_!(M_0\tens\bD N_0)=\ql$ in degree $0$ if $M\cong N$ and $\pr_!(M_0\tens\bD N_0)=0$ otherwise. Applying the Grothendieck-Lefschetz trace formula to $M_0\tens\bD N_0$ yields an identity, which in view of \eqref{e:dual-complex-conjugate} turns out to be equivalent to the orthonormality relations of Theorem \ref{t:main}(e).
\end{rem}

We now explain how to rephrase this proof without using the language of stacks. Unraveling the definition of the inner product $\eval{\cdot}{\cdot}$ given in Theorem \ref{t:main}(e), we find that the orthonormality relations for character sheaves are equivalent to the following assertion: given any $M,N\in\cs(G)^{\Fr_q^*}$, we have
\begin{equation}\label{e:need-for-orthonormality}
\sum_{\al\in H^1(\bF_q,G_0)} \left( \frac{q^{\dim G}}{\abs{G_0^\al(\bF_q)}}
\sum_{g\in G_0^\al(\bF_q)} t_{M^\al_0\tens(\bD_{G_0}N_0)^\al}(g) \right) =
\begin{cases}
1 & \text{ if } M\cong N, \\
0 & \text{ if } M\not\cong N.
\end{cases}
\end{equation}
To prove \eqref{e:need-for-orthonormality}, let $\pr:G_0\rar{}\Spec\bF_q$ denote the structure morphism. By the usual Grothendieck-Lefschetz trace formula (for schemes), we have
\[
\sum_{g\in G_0(\bF_q)} t_{M_0\tens\bD_{G_0}(N_0)}(g) = t_{\pr_!(M_0\tens\bD_{G_0}N_0)}(*),
\]
where $*$ denotes the unique $\bF_q$-point of $\Spec\bF_q$. In view of Proposition \ref{p:equivariant-sheaves-point}, we see that the left hand side of \eqref{e:need-for-orthonormality} can be rewritten as the trace of Frobenius acting on the complex $\bigl( \pr_!(M_0\tens\bD_{G_0}N_0) \bigr)^G\in\sD(\Spec\bF_q)$. If we view this complex as an object of $\sD^{Weil}(\Spec\bF_q)$, then Verdier duality implies that
\[
\bigl( \pr_!(M_0\tens\bD_{G_0}N_0) \bigr)^G \cong \bD \bigl( R\Hom^{Weil}_{\sD_{G_0}(G_0)} ( M_0,N_0 ) \bigr),
\]
where $\bD$ denotes the Verdier duality functor on $\Spec\bF_q$.
Now if $M\not\cong N$, then we claim that the complex $R\Hom_{\sD_G(G)}(M,N)$ is acyclic, which means that the right hand side of the last formula vanishes. Indeed, there are two possibilities. If $M$ and $N$ belong to $\bL$-packets defined by two nonisomorphic minimal idempotents $e,f\in\sD_G(G)$, then the corresponding Hecke subcategories $e\sD_G(G)$ and $f\sD_G(G)$ are orthogonal (in the sense of $\Hom$), which forces $R\Hom_{\sD_G(G)}(M,N)=0$. If $M$ and $N$ belong to the $\bL$-packet defined by the same minimal idempotent $e\in\sD_G(G)$, then our claim follows from parts (a) and (e) of Proposition \ref{p:properties-char-sheaves} and the fact that $M$ and $N$ are nonisomorphic simple objects of the category $\sM_e^{perv}$.

\mbr

The argument above also shows that $R\Hom_{\sD_G(G)}(M,M)$ is concentrated in degree $0$
and its $0$-th cohomology is $1$-dimensional. On the other hand, $\Hom_{\sD_{G_0}(G_0)}(M_0,M_0)$
contains the identity morphism, which maps to a nonzero $\Fr_q$-invariant element of
$\Hom^{Weil}_{\sD_{G_0}(G_0)}(M_0,M_0)$. We conclude that the complex
$R\Hom^{Weil}_{\sD_{G_0}(G_0)}(M_0,M_0)$ is concentrated in degree $0$ and its $0$-th
cohomology is a $1$-dimensional space on which the Frobenius acts trivially. The same
statement must then hold for the complex $\bD \bigl( R\Hom^{Weil}_{\sD_{G_0}(G_0)} ( M_0,M_0 ) \bigr)$,
and it follows that the associated ``trace of Frobenius'' function is equal to $1$. Thus we proved \eqref{e:need-for-orthonormality}, and with it Theorem \ref{t:main}(e).

\subsection{Geometric Mackey theory}\label{ss:geometric-mackey-theory} We now pause to discuss some auxiliary results, which will be used in the proof of Theorem \ref{t:main}(c). Let us choose a perfect field $k$ of characteristic $p>0$ and an algebraic closure $\kbar$ of $k$.

\begin{defin}\label{d:geometric-mackey-condition}
Let $G$ be a perfect unipotent group over $k$ and $G'\subset G$ a closed subgroup. Fix an object $e'\in\sD(G')$ and let $\overline{e'}$ denote the object of $\sD(G)$ obtained from $e'$ via extension by zero.
 \sbr
\begin{enumerate}[(1)]
\item Suppose $k=\overline{k}$. We say that $e'$ satisfies the \emph{geometric Mackey condition} with respect to $G$ if for each $x\in G(k)$ such that $x\not\in G'(k)$, we have $\overline{e'}*\de_x*\overline{e'}=0$, where $\de_x\in\sD(G)$ denotes the delta-sheaf at $x$.
 \sbr
\item For general $k$, we say that $e'$ satisfies the \emph{geometric Mackey condition} with respect to $G$ if the object of $\sD(G'\tens_k\overline{k})$ obtained from $e'$ by base change to $\kbar$ satisfies the geometric Mackey condition with respect to $G\tens_k\kbar$.
\end{enumerate}
\end{defin}

\begin{prop}\label{p:geometric-Mackey}
Let $G$ be a perfect unipotent group over $k$, let $G'\subset G$ be a closed subgroup, and fix a geometrically minimal\footnote{This means that $e'$ remains minimal after base change from $k$ to $\kbar$.} closed idempotent $e'\in\sD_{G'}(G')$ satisfying the geometric Mackey condition with respect to $G$. Then
 \sbr
\begin{enumerate}[$($a$)$]
\item the object $e=\ig e'$ is a geometrically minimal closed idempotent in $\sD_G(G)$;
 \sbr
\item the functor $\ig$ restricts to an equivalence of categories
\[
e'\sD_{G'}(G') \rar{\sim} e\sD_G(G);
\]
 \sbr
\item $\forall\,M\in e'\sD_{G'}(G')$, the canonical map $\ig M\rar{}\Ig M$ is an isomorphism;
 \sbr
\item if $M\in e'\sD_{G'}(G')$ is such that the underlying complex of $M$ is perverse, then the underlying complex of $(\ig M)[d]$ is also perverse, where $d=\dim(G/G')$.
\end{enumerate}
\end{prop}

\begin{proof}
All the assertions above are contained in \cite[Thm.~1.52]{foundations}.
\end{proof}

\begin{lem}\label{l:geometric-Mackey}
Let $G$ be a perfect unipotent group over $k$ and $(H,\cL)$ an admissible pair for $G$. Define $G'\subset G$ to be the normalizer of $(H,\cL)$, and let $e'\in\sD_{G'}(G')$ be the Heisenberg minimal idempotent defined by $(H,\cL)$ $($cf.~Definition \ref{d:adm-pair-min-idemp}$)$. Then $e'$ satisfies the geometric Mackey condition with respect to $G$.
\end{lem}

\begin{proof}
See \cite[\S9.5]{characters}.
\end{proof}

\subsection{Proof of Theorem \ref{t:main}(c)}\label{ss:proof-t:main-c} For the time being we assume the following result, proved in \S\ref{ss:proof-p:heis}, and explain how to deduce Theorem \ref{t:main}(c) from it.

\begin{prop}\label{p:heis}
Theorem \ref{t:main}(c) holds when $e_0\in\sD_{G_0}(G_0)$ is a Heisenberg minimal idempotent $($Def.~\ref{d:adm-pair-min-idemp}$)$ defined by an admissible pair $(H_0,\cL_0)$ for $G_0$.
\end{prop}

Now let $G_0$ be a perfect unipotent group over $\bF_q$, and let $e_0\in\sD_{G_0}(G_0)$ be a geometrically minimal idempotent. Without loss of generality we may replace $G_0$ with one of its pure inner forms and assume (using Theorem \ref{t:main}(b), which was proved earlier) that $e_0$ is isomorphic to the minimal idempotent defined by some admissible pair $(H_0,\cL_0)$ for $G_0$. Write $G_0'$ for the normalizer of $(H_0,\cL_0)$ in $G_0$ and $e_0'\in\sD_{G_0'}(G_0')$ for the Heisenberg minimal idempotent defined by $(H_0,\cL_0)$.


\mbr

It is clear that $e_0'$ is pure of weight $0$. Since $e_0\cong\igz e_0'$, Lemma \ref{l:geometric-Mackey}, Proposition \ref{p:geometric-Mackey}(c) and Proposition \ref{p:induction-purity} imply that $e_0$ is also pure of weight $0$.

\mbr

Next let $G,e,G',e'$ denote the objects obtained from $G_0,e_0,G_0',e_0'$ via base change to $\bF$, and write $\bL(e)$ and $\bL(e')$ for the $\bL$-packets of character sheaves on $G$ and $G'$ defined by $e$ and $e'$, respectively. Parts (b) and (d) of Proposition \ref{p:geometric-Mackey} yield

\begin{lem}\label{lemma}
The functor $\ig[d]:\sD_{G'}(G')\rar{}\sD_G(G)$ induces a bijection
\[
\bL(e')\rar{\simeq}\bL(e),
\]
which is compatible with the action of $\Fr_q^*$, where $d=\dim(G/G')$.
\end{lem}

Fix $M\in\bL(e)^{\Fr_q^*}$. By Lemma \ref{lemma}, there is $N\in\bL(e')^{\Fr_q^*}$ with $M\cong\bigl(\ig N\bigr)[d]$. By Proposition \ref{p:heis}, we can choose $N_0\in e_0'\sD_{G_0'}(G_0')$ such that $N$ is obtained from $N_0$ by base change to $\bF$; the underlying complex of $N_0$ is perverse and pure of weight $0$; and for all $n\in\bN$ and $\be\in H^1(\bF_{q^n},G_0')$, the function $t_{(N_0\tensqn)^\be}:G_0'^\be(\bF_{q^n})\rar{}\ql$ takes values in the subring $\bZ[\mu_{p^{2r}},p^{-1}]\subset\qab\subset\ql$. Here, as before, we write $G_0'^\be(\bF_{q^n})$ in place of $(G_0'\tensqn)^\be(\bF_{q^n})$ for brevity.

\mbr

By Remark \ref{r:absolute-value-square-root-of-p}, there exists $\la\in\bZ[\mu_{p^{2r}}]$ with $\la\cdot\overline{\la}=p$. Let $\sL_\la$ be the local system on $\Spec\bF_q$ corresponding to the vector space of dimension $1$ over $\ql$ on which the geometric Frobenius $F_q\in\Gal(\bF/\bF_q)$ acts via $\la$. Write $q=p^s$ and let $M_0=\bigl(\igz N_0\bigr)[d]\tens\pr^*(\sL_\la)^{-s\cdot d}$, where $\pr:G_0\rar{}\Spec\bF_q$ is the structure morphism and $\sL_\la$ is equipped with the trivial $G_0$-equivariant structure. Then $M$ is obtained from $M_0$ via base change to $\bF$, and $M_0\in e_0\sD_{G_0}(G_0)$ by Proposition \ref{p:geometric-Mackey}(b). Proposition \ref{p:geometric-Mackey}(d) shows that $M_0$ is perverse. By Proposition \ref{p:geometric-Mackey}(c) and Proposition \ref{p:induction-purity}, $\igz N_0$ is pure of weight $0$. Since $\pr^*(\sL_\la)^{-s\cdot d}$ is a pure local system of weight $-d$ by construction, $M_0$ is also pure of weight $0$.

\mbr

Finally, by Proposition \ref{p:induction-sheaves-functions},
\[
t_{(M_0\tensqn)^\al} = (-1)^d \cdot \la^{-n\cdot s\cdot d} \cdot \sum_{\substack{\be\in H^1(\bF_{q^n},G_0') \\ \text{such that }\be\mapsto\al}} \ind_{G_0'^\be(\bF_{q^n})}^{G_0^\al(\bF_{q^n})} t_{(N_0\tensqn)^\be}
\]
for all $n\in\bN$ and $\al\in H^1(\bF_{q^n},G_0)$, where $\la\in\bZ[\mu_{p^{2r}}]$ is as above. We have $\la^{-1}=\overline{\la}/p\in\bZ[\mu_{p^{2r}},p^{-1}]$, which by our choice of $N_0$ implies that $t_{(M_0\tensqn)^\al}$ takes values in $\bZ[\mu_{p^{2r}},p^{-1}]$, completing the proof.

\subsection{Equivariant local systems on homogeneous spaces}\label{ss:homogeneous-spaces} Let us
state some results that will be used in the proofs of Theorem \ref{t:main}(d) and Proposition \ref{p:heis}.

\begin{defin}\label{d:equiv-loc-sys}
Let $k$ be a perfect field of characteristic $p>0$, and let $U$ be a perfect group over $k$ acting on a perfect variety $X$ over $k$ (see Def.~\ref{defs:perfect}). We write $\Loc_U(X)$ for the category of $U$-equivariant $\ql$-local systems on $X$ and $\operatorname{IrrLoc}_U(X)$ for
the set of isomorphism classes of irreducible objects in $\Loc_U(X)$.
\end{defin}

\begin{prop}\label{p:homog}
Let $U_0$ be a perfect group over $\bF_q$ acting transitively on
a nonempty perfect variety $X_0$ over $\bF_q$. Write
$U=U_0\tens_{\bF_q}\bF$ and $X=X_0\tens_{\bF_q}\bF$, and let
$\Fr_q:X\rar{}X$ denote the Frobenius endomorphism.
 \sbr
\begin{enumerate}[$($a$)$]
\item Every irreducible $\cL\in\Loc_U(X)$ such that $\Fr_q^*\cL\cong\cL$ can be represented
as the base change of an object $\cL_0\in\Loc_{U_0}(X_0)$ such that $\cL_0$ is pure of weight $0$ and the trace-of-Frobenius function\footnote{Here, as always, we write $X_0^\al(\bF_{q^n})$ in place of $(X_0\tensqn)^\al(\bF_{q^n})$ to simplify the notation.} $t_{(\cL_0\tensqn)^\al}:X_0^\al(\bF_{q^n})\rar{}\ql$ takes values in $\qab\subset\ql$ for all $n\in\bN$ and all $\al\in H^1(\bF_{q^n},U_0)$.
 \sbr
\item[$($a$')$] If for some $($equivalently, every$)$ $x\in X(\bF)$, the group $\pi_0(U^x)$ has exponent $p^r$ $($where $U^x$ is the stabilizer of $x$ in $U${}$)$, then in $($a$)$ we can further ensure that $t_{(\cL_0\tensqn)^\al}:X_0^\al(\bF_{q^n})\rar{}\ql$ takes values\footnote{Here, unlike in the statement of Theorems \ref{t:main-connected} and \ref{t:main}, we do not need to replace $p^r$ with $p^{2r}$, and this is important for our proofs.} in $\bZ[\mu_{p^r},p^{-1}]\subset\qab\subset\ql$.
 \sbr
\item For every $\cL\in\operatorname{IrrLoc}_{U}(X)^{\Fr_q^*}$, fix an $\cL_0$ as in $($a$)$. The elements
\[
\Bigl\{ T_{\cL_0} = \bigl( t_{\cL_0^\al} \bigr)_{\al\in
H^1(\bF_q,U_0)} \,\Bigl\lvert\,
\cL\in\operatorname{IrrLoc}_U(X)^{\Fr_q^*} \Bigr\}
\]
form a basis of the space\footnote{Here, $(U_0\backslash X_0)(\bF_q)$ is the groupoid of $\bF_q$-points of the quotient stack $U_0\backslash X_0$, and we refer to \S\ref{ss:functions-on-groupoids} for further explanations of the notation.}
\[
\Fun\bigl((U_0\backslash X_0)(\bF_q)\bigr) = \bigoplus_{\al\in H^1(\bF_q,U_0)}
\Fun(X_0^\al(\bF_q),\qab)^{U_0^{\al}(\bF_q)}.
\]
\end{enumerate}
\end{prop}

This result is proved in \S\ref{ss:proof-p:homog}.

\begin{rem}
All the assertions of Proposition \ref{p:homog} depend only on the quotient stack $U_0\backslash X_0$, which is a gerbe because $U_0$ acts transitively on $X_0$. Thus Proposition \ref{p:homog} could be reformulated as a statement about algebraic gerbes over $\bF_q$. However, it is the form given above that will be most useful for us.
\end{rem}

\begin{rem}
Just as in Remark \ref{r:uniqueness-basis}, the local system $\cL_0$ in Proposition \ref{p:homog}(a) is determined uniquely up to tensor product with a $\ql$-local system on $X_0$ obtained by pullback from a rank $1$ local system on $\Spec\bF_q$ such that the corresponding character $\Gal(\bF/\bF_q)\rar{}\ql^\times$ takes values in the subgroup consisting of elements of $\bQ(\mu_{p^r})$ of absolute value $1$ (and the $U_0$-equivariant structure is trivial).
\end{rem}

\begin{defin}\label{d:quasi-equivariant}
In the setting of Definition \ref{d:equiv-loc-sys}, consider a central extension
\[
1 \rar{} A \rar{} \Ut \rar{} U \rar{} 1
\]
of perfect groups over $k$, where $A$ is finite and discrete.
We have the induced action of $\Ut$ on $X$, so we can form the category\footnote{The natural functor $\Loc_U(X)\rar{}\Loc_\Ut(X)$
is in general \emph{not} an equivalence.} $\Loc_{\Ut}(X)$. Since $A\subset\Ut$
acts trivially on $X$, for every $M\in\Loc_{\Ut}(X)$ we
obtain an action of $A$ on $M$ by automorphisms. If $\chi:A\rar{}\ql^\times$ is a homomorphism,
we write $\Loc_{\Ut}^\chi(X)$ for the full subcategory of $\Loc_{\Ut}(X)$ consisting of objects
on which $A$ acts via the scalar $\chi$. The set of isomorphism classes of simple objects in $\Loc_{\Ut}^\chi(X)$ will be denoted by $\operatorname{IrrLoc}_{\Ut}^\chi(X)$.
\end{defin}

Next we would like to have an analogue of Proposition \ref{p:homog} for the
categories $\Loc_\Ut^\chi(X)$. In order to state it we need some auxiliary constructions.

\begin{defin}\label{d:aux}
Consider a central extension
\[
1 \rar{} A \rar{} \Ut_0 \rar{f} U_0 \rar{} 1
\]
of perfect groups over $\bF_q$, where $A$ is finite and discrete.
Let \[\de:H^0(\bF_q,U_0)=U_0(\bF_q)\rar{}H^1(\bF_q,A)\cong A\] denote the
connecting homomorphism in the corresponding long exact sequence of Galois cohomology.
If $\chi:A\rar{}\ql^\times$ is a homomorphism and $X_0$ is a perfect variety over
$\bF_q$ equipped with an action of $U_0$,
we write $\Fun(X_0(\bF_q),\qab)^{U_0(\bF_q),\chi}$ for the space of functions
$X_0(\bF_q)\rar{}\qab$ on which $U_0(\bF_q)$ acts\footnote{Note
that the image of $\chi$ consists of roots of unity and hence lies in $\qab$.} via $\chi\circ\de:U_0(\bF_q)\rar{}\ql^\times$.
\end{defin}

\begin{rems}\label{rems:aux}
\begin{enumerate}[(1)]
\item In the situation of Definition \ref{d:aux},
the long exact sequence of Galois cohomology yields an exact sequence
\begin{equation}\label{e:aux}
1 \rar{} A \rar{} \Ut_0(\bF_q) \rar{} U_0(\bF_q) \rar{} A \rar{} H^1(\bF_q,\Ut_0) \rar{} H^1(\bF_q,U_0),
\end{equation}
where we identified both $A(\bF_q)$ and $H^1(\bF_q,A)$ with $A$. In particular,
we see that the image of $\Ut_0(\bF_q)$ under $f$ is a normal subgroup of $U_0(\bF_q)$ and
the quotient $U_0(\bF_q)/f(\Ut_0(\bF_q))$ is identified with a subgroup of $A$.
 \sbr
\item Suppose that, in addition, $A$ is contained in the neutral connected
component $\Ut^\circ_0$ of $\Ut_0$. Then the map $\pi_0(\Ut_0)\rar{}\pi_0(U_0)$ induced
by $f$ is an isomorphism, whence the map $H^1(\bF_q,\Ut_0) \rar{} H^1(\bF_q,U_0)$
in \eqref{e:aux} is a bijection.
 \sbr
\item If $A\subset\Ut_0^\circ$, then the last remark and the exactness of \eqref{e:aux} imply
that the induced map $U_0(\bF_q)/f(\Ut_0(\bF_q))\rar{}A$
is a group isomorphism.
 \sbr
\item Suppose that $\Ut_1$ is another perfect group over $\bF_q$ and $P$ is a
$(\Ut_1,\Ut_0)$-bitorsor. Since $A$ is central in $\Ut_0$, there exists a unique
embedding $A\into\Ut_1$ such that the left and right actions of $A$ on $P$ induced
by the embeddings $A\into\Ut_0$ and $A\into\Ut_1$,
coincide. Moreover, the quotient $P/A$ becomes a $(\Ut_1/A,U_0)$-bitorsor, and
hence $\Ut_1/A$ is identified with a pure inner form of $U_0$.

\mbr

Using this observation, we may and will canonically identify $A$ with a discrete central subgroup of any pure inner form of $\Ut_0$ in all that follows.
\end{enumerate}
\end{rems}

\begin{defin}\label{d:aux2}
In the situation of Definition \ref{d:aux}, assume that $A\subset\Ut_0^\circ$. Let
$\chi:A\rar{}\ql^\times$ be a homomorphism, and let $X_0$ be a perfect variety over
$\bF_q$ equipped with an action of $U_0$. We put
\[
\Fun(X_0)^{U_0,\chi} = \bigoplus_{\al\in H^1(\bF_q,U_0)} \Fun(X^\al_0(\bF_q),\qab)^{U^\al_0(\bF_q),\chi}.
\]
Here we identify $H^1(\bF_q,\Ut_0)=H^1(\bF_q,U_0)$ as in Remark \ref{rems:aux}(2).
In addition, for each $\al\in H^1(\bF_q,\Ut_0)$, we identify $A$ with a subgroup
of the pure inner form $\Ut_0^\al$, and we identify $U_0^\al$ with the corresponding
quotient $\Ut_0^\al/A$, as in Remark \ref{rems:aux}(4). The spaces
$\Fun(X^\al_0(\bF_q),\qab)^{U^\al_0(\bF_q),\chi}$ were constructed in Definition \ref{d:aux}.
\end{defin}

\begin{cor}\label{c:homog}
Let $U_0$ be a perfect group over $\bF_q$ acting transitively on
a nonempty perfect variety $X_0$ over $\bF_q$. Write
$U=U_0\tens_{\bF_q}\bF$ and $X=X_0\tens_{\bF_q}\bF$, and let
$\Fr_q:X\rar{}X$ denote the Frobenius endomorphism. In addition, consider
a central extension $1\rar{}A\rar{}\Ut_0\rar{}U_0\rar{}1$, where $A$ is a finite
discrete subgroup of $\Ut_0^\circ$; write $\Ut=\Ut_0\tens_{\bF_q}\bF$; and let $\chi:A\rar{}\ql^\times$ be a homomorphism.
 \sbr
\begin{enumerate}[$($a$)$]
\item Every irreducible $\cL\in\Loc_\Ut^\chi(X)$ such that $\Fr_q^*\cL\cong\cL$ can be represented as the base change of an object $\cL_0\in\Loc_{\Ut_0}^\chi(X_0)$ such that $\cL_0$ is pure of weight $0$ and the trace-of-Frobenius function\footnote{Here, as always, we write $X_0^\al(\bF_{q^n})$ in place of $(X_0\tensqn)^\al(\bF_{q^n})$ to simplify the notation.} $t_{(\cL_0\tensqn)^\al}:X_0^\al(\bF_{q^n})\rar{}\ql$ takes values in $\qab\subset\ql$ for all $n\in\bN$ and all $\al\in H^1(\bF_{q^n},U_0)$.
 \sbr
\item[$($a$')$] If for some $($equivalently, every$)$ $x\in X(\bF)$, the group $\pi_0(\Ut^x)$ has exponent $p^r$ $($where $\Ut^x$ is the stabilizer of $x$ in $\Ut${}$)$, then in $($a$)$ we can further ensure that $t_{(\cL_0\tensqn)^\al}:X_0^\al(\bF_{q^n})\rar{}\ql$ takes values in $\bZ[\mu_{p^r},p^{-1}]\subset\qab\subset\ql$.
 \sbr
\item For every $\cL\in\operatorname{IrrLoc}_{\Ut}^\chi(X)^{\Fr_q^*}$, fix an $\cL_0$ as in $($a$)$. The elements
\begin{equation}\label{e:basis}
\Bigl\{ T_{\cL_0} = \bigl( t_{\cL_0^\al} \bigr)_{\al\in
H^1(\bF_q,U_0)} \,\Bigl\lvert\,
\cL\in\operatorname{IrrLoc}_\Ut^\chi(X)^{\Fr_q^*} \Bigr\}
\end{equation}
form a basis of the space $\Fun(X_0)^{U_0,\chi}$.
\end{enumerate}
\end{cor}

This result, which can be reduced to Proposition \ref{p:homog}, is proved in \S\ref{ss:proof-c:homog}.

\subsection{Proof of Theorem \ref{t:main}(d)}\label{ss:proof-t:main-d} The nontrivial part of the proof is the following

\begin{prop}\label{p:implies-main-d}
Let $G_0$ be any perfect unipotent group over $\bF_q$, let
$G=G_0\tens_{\bF_q}\bF$, and let $\Fr_q:G\rar{}G$ be the Frobenius
endomorphism. For every character sheaf $M$ on $G$ such that $\Fr_q^*M\cong M$, choose $M_0\in\sD_{G_0}(G_0)$ satisfying the requirements of Theorem \ref{t:main}(c). Then $\Fun\bigl([G_0](\bF_q)\bigr)$ is spanned by $\bigl\{ T_{M_0} \st M\in\cs(G)^{\Fr_q^*}\bigr\}$.
\end{prop}

Indeed, the fact that the elements $\bigl\{ T_{M_0} \st M\in\cs(G)^{\Fr_q^*}\bigr\}$ are linearly independent follows from the orthonormality relations proved in \S\ref{ss:proof-orthonormality}. Hence Proposition \ref{p:implies-main-d} implies that these elements form a basis of $\Fun\bigl([G_0](\bF_q)\bigr)$. Also, if $M_0\in e_0\sD_{G_0}(G_0)$ for some geometrically minimal idempotent $e_0\in\sD_{G_0}(G_0)$, then $T_{M_0}$ belongs to the span of the set of
elements in the $\bL$-packet of irreducible characters defined by
$e_0$. Hence Theorem \ref{t:main}(d) follows from the previous observations and the obvious

\begin{lem}\label{l:basis-direct-sum}
Let $V=\bigoplus_{i\in I}V_i$ be a direct sum of vector spaces over a field.
For each $i\in I$, suppose that we are given a subset $S_i\subset V_i$. If $\bigcup_{i\in I}S_i$
is a basis of $V$, then $S_i$ is a basis of $V_i$ for every $i\in I$.
\end{lem}

It remains to prove Proposition \ref{p:implies-main-d}. We begin with

\begin{lem}
As $N_0$ ranges over \emph{all} objects of $\sD_{G_0}(G_0)$, the elements $T_{N_0}=\bigl(t_{N_0^\al}\bigr)_{\al\in H^1(\bF_q,G_0)}$ span $\Fun\bigl([G_0](\bF_q),\ql\bigr)$.
\end{lem}

\begin{proof}
It suffices to check that an element of $\Fun\bigl([G_0](\bF_q),\ql\bigr)$, which is supported on a single geometric conjugacy class in $G_0$, can be written as a linear combination of elements of the form $T_{N_0}$ with $N_0\in\sD_{G_0}(G_0)$. To this end, let $X\subset G$ be any $\Fr_q$-stable conjugacy class. Then $X$ comes from a subscheme $X_0\subset G_0$ and Proposition \ref{p:homog} implies that $\Fun\bigl((G_0\backslash X_0)(\bF_q),\ql\bigr)$ is spanned by elements of the form $T_{N_0}$ as $N_0$ ranges over $G_0$-equivariant local systems on $X_0$. Extending all such $N_0$ by zero to $G_0$, we obtain the desired claim.
\end{proof}

Next choose a complete collection of representatives $e^1_0,\dotsc,e^m_0$ of the isomorphism classes of geometrically minimal idempotents in $\sD_{G_0}(G_0)$. In view of Proposition \ref{p:equivalence-def-L-packet-characters}, this collection is finite and the sum $\sum_{i=1}^m t_{(e^i_0)^\al}$ is equal to the delta-function at $1\in G_0^\al(\bF_q)$ for every $\al\in H^1(\bF_q,G_0)$. Hence for every $N_0\in\sD_{G_0}(G_0)$, we have
\[
T_{N_0} = \sum_{i=1}^m T_{e_0^i*N_0} \qquad\text{in}\quad \Fun\bigl([G_0](\bF_q),\ql\bigr).
\]
So Proposition \ref{p:implies-main-d} follows from the previous lemma and

\begin{lem}
Let $e_0\in\sD_{G_0}(G_0)$ be a geometrically minimal idempotent, and let $N_0\in e_0\sD_{G_0}(G_0)$. Then $T_{N_0}$ belongs to the span of the elements $T_{M_0}$, where $M$ ranges over the $\Fr_q^*$-invariant character sheaves on $G$ in the $\bL$-packet defined by $e=e_0\tens_{\bF_q}\bF$ and $M_0$ is chosen as in Theorem \ref{t:main}(c).
\end{lem}

\begin{proof}
Write $N_0^{Weil}=(N,\vp_N)$ for the object of $\sD_{G_0}^{Weil}(G_0)$ coming from $N_0$. Then $N\in e\sD_G(G)$, so by Proposition \ref{p:properties-char-sheaves}, $N$ is isomorphic to a finite direct sum of cohomological shifts of character sheaves in the $\bL$-packet defined by $e$. Without loss of generality, we may assume that $N$ is a direct sum of character sheaves without shifts. We can then write $N=N'\oplus N''$, where $N'\subset N$ is the sum of $\Fr_q^*$-invariant character sheaves appearing in $N$ and $N''$ is the sum of character sheaves appearing in $N$ that are not $\Fr_q^*$-invariant. We obtain the corresponding decomposition $N_0=N'_0\oplus N''_0$, and it is clear that $T_{N''_0}=0$.
Moreover, Schur's lemma implies that $N'^{Weil}_0$ is a direct sum of objects of the form $M^{Weil}_0\tens\pr^*(\sE_M)$, where $M$ ranges over the $\Fr_q^*$-invariant character sheaves in the $\bL$-packet defined by $e$, the $M_0$ are chosen as in Theorem \ref{t:main}(c), each $\sE_M$ is an object of $\sD^{Weil}_{G_0}(\Spec\bF_q)$ corresponding to a finite dimensional vector space $V_M$ together with an automorphism $\phi_M:V_M\rar{\simeq}V_M$ (where the $G_0$-equivariant structure is trivial), and $\pr:G_0\rar{}\Spec\bF_q$ denotes the structure morphism.

%

\mbr

Hence
\[
T_{N_0} = T_{N'_0} = \sum_M \tr(\phi_M;V_M)\cdot T_{M_0},
\]
which proves the lemma, and with it finishes the proof of Theorem \ref{t:main}(d).
\end{proof}

\subsection{Proof of Theorem \ref{t:main-easy}}\label{ss:proof-t:main-easy} Let $e\in\sD_G(G)$ be a minimal idempotent such that $\Fr_q^*e\cong e$. Assertions (a) and (b) of Theorem \ref{t:main-easy} follow from parts (a) and (c) of Theorem \ref{t:main}. In particular, we let $e_0\in\sD_{G_0}(G_0)$ denote the (unique) closed idempotent whose base change to $\bF$ is isomorphic to $e$.

\mbr

Next we will prove parts (c) and (d) of Theorem \ref{t:main-easy}. We use a more general

\begin{lem}\label{l:aux-easy}
Suppose $G_0$ is any connected perfect unipotent group over $\bF_q$ and $e_0\in\sD_{G_0}(G_0)$ is a geometrically minimal idempotent. Then $t_{e_0}(1)=q^{-n_e}$, where $e\in\sD_G(G)$ is the minimal idempotent obtained from $e$ via base change to $\bF$ and $n_e$ is as defined in Proposition \ref{p:properties-char-sheaves}(b).
\end{lem}

\begin{proof}
We use a calculation, which is similar to the chain of isomorphisms \eqref{e:key-1}, but easier.
By Theorem \ref{t:main}(b), there exists an admissible pair for $G_0$ (defined over $\bF_q$) that
gives rise\footnote{Note that since $G_0$ is connected, it has no pure inner forms other than itself.}
to $e_0$. By\footnote{Even though the base field in \cite{foundations} is assumed to be algebraically closed, this fact is irrelevant in the proof of the cited proposition.} \cite[Prop.~1.46]{foundations}, we have $\bD_{G_0}^-(e_0)\cong e_0[-2n_e](-n_e)$.
Hence we obtain the following chain of isomorphisms\footnote{As usual, we denote by $1:\Spec\bF_q\rar{}G_0$
the identity element of $G_0$, and by $\iota:G_0\rar{}G_0$ the inversion map $g\mapsto g^{-1}$.} in $\sD(\Spec\bF_q)$:
\begin{eqnarray*}
\bD_{\Spec\bF_q}(1^*e_0) &\cong& \bD_{\Spec(\bF_q)}(1^*\iota^*e_0) \cong 1^!(\bD_{G_0}\iota^*e_0)
= 1^!(\bD_{G_0}^-e_0) \\
 &\cong& 1^!\bigl(e_0[-2n_e](-n_e)\bigr) \cong (1^!e_0)[-2n_e](-n_e).
\end{eqnarray*}

Now if we view $1^!e_0$ as an object of $\sD^{Weil}(\Spec\bF_q)$ (see \S\ref{ss:weil-definitions}), we have
\[
1^!e_0 \cong R\Hom^{Weil}_{\sD(\Spec\bF_q)}(\ql,1^!e_0) \cong R\Hom^{Weil}_{\sD(G_0)}(\e,e_0).
\]
Since $G$ is connected, the forgetful functor $\sD_G(G)\rar{}\sD(G)$ is fully faithful. By
Proposition \ref{p:hom-1-to-min-idemp}, the complex $R\Hom^{Weil}_{\sD(G_0)}(\e,e_0)$ has $1$-dimensional
cohomology in degree $0$ and is acyclic in all nonzero degrees. Hence $1^!e_0\cong\ql$ (with the
trivial Frobenius action) in $\sD^{Weil}(\Spec\bF_q)$, because by assumption we have an idempotent
arrow $\e\rar{}e_0$, which maps to a nonzero $\Fr_q$-invariant element of $\Hom(\e,e)$.

\mbr

We conclude that $\bD_{\Spec\bF_q}(1^*e_0)\cong\ql[-2n_e](-n_e)$ in $\sD^{Weil}(\Spec\bF_q)$, whence
$1^*e_0\cong\ql[2n_e](n_e)$ in $\sD^{Weil}(\Spec\bF_q)$. \emph{A fortiori}, this implies that $t_{e_0}(1)=q^{-n_e}$.
\end{proof}

We return to the setting of Theorem \ref{t:main-easy}; in particular, we assume that $G_0$ is easy. If $e\in\sD_G(G)$ is a minimal idempotent such that $\Fr_q^*(e)\cong e$, Proposition \ref{p:properties-char-sheaves}(c) implies that $e[-n_e]$ is a $\Fr_q^*$-invariant character sheaf in the $\bL$-packet defined by $e$. On the other hand, it follows from \cite[\S9.4]{characters} that the $\bL$-packet of irreducible characters of $G_0(\bF_q)$ defined by $e_0$ contains only one element. By Theorem \ref{t:main}(d), $t_{e_0}$ must be proportional to an irreducible character $\chi$ of $G_0(\bF_q)$ over $\ql$. Since $\frac{\chi(1)}{\abs{G_0(\bF_q)}}\cdot\chi$ is an idempotent under convolution of functions on $G_0(\bF_q)$, this forces $\chi(1)^2=t_{e_0}(1)\cdot\abs{G_0(\bF_q)}$. By Lemma \ref{l:aux-easy}, the last identity can be rewritten as $\chi(1)^2=q^{-n_e+\dim G}$, which implies that $\chi(1)=q^{d_e}$, and therefore $\chi=q^{\dim G-d_e}\cdot t_{e_0}$. The fact that $d_e$ is an integer follows from the identity $\chi(1)=q^{d_e}$ combined with \cite[Theorem~2.5]{characters}. Thus we proved Theorem \ref{t:main-easy}(c).

\begin{rem}
The argument of the previous paragraph implies that $e[-n_e]$ is the unique character sheaf in the $\bL$-packet defined by $e$.
\end{rem}

Finally, let $e$ range over all minimal idempotents in $\sD_G(G)$ satisfying $\Fr_q^*(e)\cong e$. By the last remark, $e[-n_e]$ then ranges over all $\Fr_q^*$-invariant character sheaves on $G$. Theorem \ref{t:main}(d) implies that the corresponding functions $t_{e_0}$ form a basis for the space $\fun{G_0(\bF_q)}$. We already saw that the functions $q^{\dim G-d_e}\cdot t_{e_0}$ are irreducible characters of $G_0(\bF_q)$, and now we conclude that there cannot be any other irreducible characters, proving Theorem \ref{t:main-easy}(d).

\subsection{Proof of Proposition \ref{p:heis}}\label{ss:proof-p:heis} The argument consists of several steps. We begin by explaining the general idea of the proof, and then supply all the details.

\subsubsection{Overview}\label{sss:overview}
By assumption, $H_0$ is normal in $G_0$, so we can consider the action of $G_0$ on $H_0$ by conjugation and form the semidirect product $U_0=G_0\ltimes H_0$. We let $U_0$ act on $G_0$ in such a way that $H_0\subset U_0$ acts by left translations and $G_0\subset U_0$ acts by conjugation. As usual we write $H$, $G$ and $U=G\ltimes H$ for the perfect unipotent groups over $\bF$ obtained from $H_0$, $G_0$ and $U_0$ via base change. Further, we write $e\in\sD_G(G)$ for the Heisenberg minimal idempotent obtained from $e_0$.

\begin{rems}\label{r:action-semidirect-product}
\begin{enumerate}[(1)]
\item It is proved in \cite[\S4.1]{tanmay} that there exist finitely many $U$-orbits in $G$ such that the support of every object of $e\sD_{G}(G)$ is contained in their union. We caution the reader that the notation used in \cite{tanmay} slightly differs from the one we are using. Namely, the meaning of $G,\cL,U$ is the same, while our $G^\circ$ and $H$ are denoted in \emph{op.~cit.} by $H$ and $N$, respectively.
 \sbr
\item All $U$-orbits in $G$ are closed because $U$ is unipotent and $G$ is affine.
\end{enumerate}
\end{rems}

Using the local system $\cL_0$ on $H_0$, we will find a central extension $\Ut_0$ of $U_0$ by a finite discrete $p$-group $A$ and a homomorphism $\chi:A\rar{}\ql^\times$ such that the Hecke subcategory $e_0\sD_{G_0}(G_0)\subset\sD_{G_0}(G_0)$ equals\footnote{This step in the proof is essentially the same as the argument in \cite[\S3.5]{tanmay}.} $\sD_{\Ut_0}^\chi(G_0)$ (cf.~Definition \ref{d:quasi-equivariant}). Even though the action of $U_0$ on $G_0$ is not transitive, Remarks \ref{r:action-semidirect-product} will allow us to apply Corollary \ref{c:homog} in this situation. Proposition \ref{p:heis} will then be deduced from it.

\mbr

Some preliminary results appear in \S\S\ref{sss:construction-Ht0}--\ref{sss:alternative-Hecke-preliminary}. We construct $\Ut_0$ in \S\ref{sss:construction-Ut0} and give alternative descriptions of $e_0\sD_{G_0}(G_0)$ and $e\sD_{G}(G)$ in \S\ref{sss:alternative-Hecke}. The proof of Proposition \ref{p:heis} itself is given in \S\ref{sss:proof}.

\subsubsection{Construction of $\widetilde{H}_0$}\label{sss:construction-Ht0}

\begin{lem}\label{l:Ht0}
There exists a central extension
\begin{equation}\label{e:Ht0}
1 \rar{} A \rar{} \Ht_0 \rar{} H_0 \rar{} 1
\end{equation}
of perfect unipotent groups over $\bF_q$, where $\Ht_0$ is connected and $A$ is finite and discrete, together with an injective homomorphism $\chi:A\into\ql^\times$, such that if we view $\Ht_0$ as an $A$-torsor over $H_0$, then the $\ql$-local system on $H_0$ obtained from this torsor via $\chi$ is isomorphic to $\cL_0$.
\end{lem}

\begin{proof}
See \cite[Lemma~7.3]{characters}.
\end{proof}

For the remainder of this section we fix a central extension \eqref{e:Ht0} and an injective homomorphism $\chi:A\into\ql^\times$ satisfying the conclusions of the lemma.

\begin{rem}\label{r:function-L0}
If $\de:H_0(\bF_q)\rar{}A$ is the connecting homomorphism in the long exact Galois cohomology sequence corresponding to \eqref{e:Ht0}, then $t_{\cL_0}=\chi\circ\de$.
\end{rem}

\subsubsection{Alternative description of $e_0\sD(G_0)\subset\sD(G_0)$}\label{sss:alternative-Hecke-preliminary}
The action of $H_0$ on $G_0$ by left translation induces an action of $\Ht_0$ on $G_0$. Since $A\subset\Ht_0$ acts trivially on $G_0$, the full subcategory $\sD_{\Ht_0}^\chi(G_0)\subset\sD_{\Ht_0}(G_0)$ can be constructed as in Definition \ref{d:quasi-equivariant}. Namely, it consists of objects of $\sD_{\Ht_0}(G_0)$ on which $A$ acts via $\chi$.

\mbr

The next result is standard (cf.~\cite[Proposition~3.11(a)]{tanmay}).
\begin{lem}\label{l:alternative-Hecke-preliminary}
The forgetful functor $\sD_{\Ht_0}(G_0)\rar{}\sD(G_0)$ is fully faithful, and it induces an equivalence of categories $\sD_{\Ht_0}^\chi(G_0)\rar{\sim}e_0\sD(G_0)$.
\end{lem}

\subsubsection{Construction of $\widetilde{U}_0$}\label{sss:construction-Ut0}
Let $c:G_0\times H_0\rar{}H_0$ be the conjugation action map: $c(g,h)=ghg^{-1}$. The central extension \eqref{e:Ht0} is invariant under $G_0$-conjugation (because the local system $\cL_0$ on $H_0$ is $G_0$-invariant by assumption). In other words, let us form the fiber product of the quotient map $\Ht_0\rar{}H_0$ with the morphism $c:G_0\times H_0\rar{}H_0$ and view it as a family of central extensions of $H_0$ parameterized by $G_0$. Then this family is constant. This means that there is a unique morphism $\widetilde{c}:G_0\times\Ht_0\rar{}\Ht_0$ such that $\widetilde{c}(g,a)=a$ for all $g\in G_0$ and all $a\in A$, and the following diagram commutes:
\[
\xymatrix{
 G_0 \times \Ht_0 \ar[d] \ar[rr]^{\ \ \ \widetilde{c}} & & \Ht_0 \ar[d] \\
 G_0 \times H_0 \ar[rr]^{\ \ \ c} & & H_0
   }
\]
The connectedness of $\Ht_0$ further implies that $\widetilde{c}$ defines an action of $G_0$ on $\Ht_0$ by group automorphisms. Using this action, we form the semidirect product $\Ut_0=G_0\ltimes\Ht_0$. Thus we obtain a central extension of perfect unipotent groups over $\bF_q$:
\begin{equation}\label{e:Ut0}
1 \rar{} A \rar{} \Ut_0 \rar{} U_0 \rar{} 1,
\end{equation}
where $U_0=G_0\ltimes H_0$, as defined in \S\ref{sss:overview}.

\subsubsection{Alternative descriptions of $e_0\sD_{G_0}(G_0)$ and $e\sD_{G}(G)$}\label{sss:alternative-Hecke} The action of $U_0$ on $G_0$ described in \S\ref{sss:overview} induces an action of $\Ut_0$ on $G_0$. Similarly, we have actions of $U=U_0\tens_{\bF_q}\bF$ and $\Ut=\Ut_0\tens_{\bF_q}\bF$ on $G=G_0\tens_{\bF_q}\bF$.

\begin{lem}\label{l:bound-on-exponent}
Fix $x\in G(\bF)$ and let $\Ut^x$ be the stabilizer of $x$ in $\Ut$. Then the exponent of $\pi_0(\Ut^x)$ is at most equal to $p^{2r}$, where $p^r$ is the exponent of $G$.
\end{lem}

\begin{proof}
Let $p^s$ denote the exponent of $H_0$. Then the map $h\mapsto h^{p^s}$ takes $\Ht_0$ into $A$, so since $\Ht_0$ is connected, that map is constant. Thus $p^s$ is also the exponent of $\Ht_0$. So the exponent of $\Ut$ is at most $p^r\cdot p^s\leq p^{2r}$, which implies the lemma.
\end{proof}

\begin{rem}\label{r:optimality-bound-exponent}
It is not known to us whether the bound in Lemma \ref{l:bound-on-exponent} is optimal. If one could improve this bound, one could also replace the ring $\bZ[\mu_{p^{2r}},p^{-1}]$ with a smaller one in the statements of Theorems \ref{t:main-connected} and \ref{t:main}.
\end{rem}

\begin{lem}\label{l:alternative-Hecke}
\begin{enumerate}[$($a$)$]
\item The natural functor
\begin{equation}\label{e:forg}
\sD_{\Ut_0}^\chi(G_0)\rar{}\sD_{G_0}(G_0)
\end{equation}
is fully faithful, and its essential image is equal to $e_0\sD_{G_0}(G_0)$.
 \sbr
\item The natural functor $\sD_{\Ut}^\chi(G)\rar{}\sD_{G}(G)$ is fully faithful, and its essential image is equal to $e\sD_{G}(G)$.
\end{enumerate}
\end{lem}

\begin{proof}[Proof $($cf.~Proposition~3.11(b)~in~\cite{tanmay}$)$.]
We only consider assertion (a) (the proof of (b) is completely analogous). The fact that \eqref{e:forg} is fully faithful follows from the fact that $\Ht_0$ is connected (and unipotent). To show that the essential image of \eqref{e:forg} equals $e_0\sD_{G_0}(G_0)$, use Lemma \ref{l:alternative-Hecke-preliminary} and the observation that an object of $\sD_{G_0}(G_0)$ belongs to $e_0\sD_{G_0}(G_0)$ if and only if the corresponding object of $\sD(G_0)$ (obtained by forgetting the $G_0$-equivariant structure) belongs to $e_0\sD(G_0)$.
\end{proof}

\subsubsection{Identifications of Galois cohomology and pure inner forms}\label{sss:identifications}
The inclusion maps $G_0\into U_0$ and $G_0\into\Ut_0$ yield bijections $H^1(\bF_q,G_0)\rar{\simeq}H^1(\bF_q,U_0)$ and $H^1(\bF_q,G_0)\rar{\simeq}H^1(\bF_q,\Ut_0)$. For the remainder of the argument, we tacitly identify $H^1(\bF_q,U_0)$ and $H^1(\bF_q,\Ut_0)$ with $H^1(\bF_q,G_0)$ using these bijections.

\begin{rem}\label{r:ambiguity}
If $\be\in H^1(\bF_q,G_0)$, we now have three different interpretations of $G_0^\be$. Namely, we can consider $G_0$ as an object of $G_0\var$, or as an object of $U_0\var$, or as an object of $\Ut_0\var$, and in each case we can consider the corresponding pure inner form of $G_0$ defined by $\be$. However, all three of these are canonically identified with each other, so this ambiguity is irrelevant for the arguments that follow.
\end{rem}

\subsubsection{Proof of Proposition \ref{p:heis}}\label{sss:proof} By Remarks \ref{r:action-semidirect-product}, there is a closed $\bF_q$-subvariety $X_0\subset G_0$ such that $X=X_0\tens_{\bF_q}\bF$ is a union of finitely many $U$-orbits in $G$, and for every $N\in e\sD_{G}(G)$, the support of $N$ is contained in $X$.

\mbr

By Lemma \ref{l:alternative-Hecke}, we obtain equivalences of categories
\[
\sD_{\Ut_0}^\chi(X_0)\rar{\sim}e_0\sD_{G_0}(G_0) \qquad\text{and}\qquad \sD_{\Ut}^\chi(X)\rar{\sim}e\sD_{G}(G).
\]
If $\cM\in\Loc_{\Ut}^\chi(X)$, the support of $\cM$ is a smooth closed subvariety of $G$. Let $s_{\cM}$ denote its dimension. It follows that the map $\cM\mapsto\cM[s_{\cM}]$ induces a bijection between $\IrrLoc_\Ut^\chi(X)$ and the $\bL$-packet $\bL(e)$ of character sheaves on $G$ defined by the minimal idempotent $e$. This bijection is compatible with the action of $\Fr_q^*$.

\mbr

Now choose $M\in\bL(e)^{\Fr_q^*}$. Then $M\cong\cM[s_{\cM}]$, where $\cM\in\IrrLoc_{\Ut}^\chi(X)^{\Fr_q^*}$ is supported on a single $\Fr_q$-stable $U$-orbit in $X$. Hence by Lemma \ref{l:bound-on-exponent} and Corollary \ref{c:homog}(a$'$), $\cM$ is isomorphic to the base change of some object $\cM_0\in\Loc_{\Ut_0}^\chi(X_0)$ such that $\cM_0$ is pure of weight $0$ and the function $t_{(\cM_0\tensqn)^\al}:X_0^\al(\bF_{q^n})\rar{}\ql$ takes values in $\bZ[\mu_{p^{2r}},p^{-1}]$ for each $n\in\bN$ and each $\al\in H^1(\bF_q,G_0)$, where $p^r$ is the exponent of $G_0$.

\mbr

By Remark \ref{r:absolute-value-square-root-of-p}, there exists $\la\in\bZ[\mu_{p^{2r}},p^{-1}]$ with $\la\cdot\overline{\la}=q^{-s_{\cM}}$.

\mbr

Finally, let $\sL_\la$ be the $\ql$-local system on $\Spec\bF_q$ corresponding to the vector space of dimension $1$ on which the geometric Frobenius acts via $\la$, equip $\sL_\la$ with the trivial $G_0$-equivariant structure, write $\pr:G_0\rar{}\Spec\bF_q$ for the structure morphism, and put $M_0=\cM_0[s_{\cM}]\tens\pr^*\sL_\la\in e_0\sD_{G_0}(G_0)$. It is clear from the construction that $M_0$ satisfies all the requirements of Theorem \ref{t:main}(c).

\mbr

Finally, the fact that $e_0$ itself is pure of weight $0$ follows from its definition as extension-by-zero of $\cL_0[2\dim H](\dim H)$.


\appendix

\section*{Appendix}

\setcounter{section}{1}
\setcounter{thm}{0}

The goal of this appendix is to supply the proofs of the more technical results and the auxiliary claims appearing in the article.

\subsection{Proof of the assertion of Remark \ref{r:sum-of-squares-of-dimensions}}\label{ss:proof-r:sum-of-squares-of-dimensions} Let $\bL(e_0)$ denote the $\bL$-packet of irreducible characters of $G_0(\bF_q)$ defined by $e_0$. For each $\chi\in\bL(e_0)$, the function $\frac{\chi(1)}{\abs{G_0(\bF_q)}}\cdot\chi$ on $G_0(\bF_q)$ is an idempotent under convolution, and
\[
\sum_{\chi\in\bL(e_0)} \frac{\chi(1)}{\abs{G_0(\bF_q)}}\cdot\chi = t_{e_0}.
\]
Evaluating both sides at $1\in G_0(\bF_q)$ yields $\sum_{\chi\in\bL(e_0)} \chi(1)^2=q^{\dim G-n_e}$ in view of Lemma \ref{l:aux-easy} and the fact that $\abs{G_0(\bF_q)}=q^{\dim G}$ because $G_0$ is connected. Recalling that $\dim G-n_e=2d_e$ by definition, we obtain \eqref{e:sum-of-squares-of-dimensions}.

\subsection{Positive functions on finite groups}\label{ss:positive-functions} Let us introduce a simple auxiliary notion, which will be helpful in \S\ref{ss:proof-p:equivalence-def-L-packet-characters} below.

\begin{lem}
Let $\Ga$ be a finite group and $f\in\Fun(\Ga,\ql)^\Ga$ $($here, as usual, $\Ga$ acts on itself by conjugation$)$. The following are equivalent:
 \sbr
\begin{enumerate}[$($i$)$]
\item $f$ acts as a nonnegative rational scalar in each irreducible representation of $\Ga$;
 \sbr
\item $f$ is a linear combination of the irreducible characters of $\Ga$ with nonnegative rational coefficients;
 \sbr
\item $f$ is a linear combination of the minimal idempotents $($under convolution$)$ in $\Fun(\Ga,\ql)^\Ga$ with nonnegative rational coefficients.
\end{enumerate}
\end{lem}

The lemma follows at once from the standard facts that $\Fun(\Ga,\ql)^\Ga$ is isomorphic to a product of copies of $\ql$ as an algebra, and that the map $\chi\mapsto\frac{\chi(1)}{\abs{\Ga}}\cdot\chi$ is a bijection between irreducible characters of $\Ga$ and minimal idempotents in $\Fun(\Ga,\ql)^\Ga$.

\mbr

We call a function $f:\Ga\rar{}\ql$ \emph{positive} if it is conjugation-invariant and satisfies the equivalent conditions of the lemma.

\begin{prop}\label{p:properties-positive-functions}
Let $\Ga$ be a finite group.
 \sbr
\begin{enumerate}[$($a$)$]
\item If $f_1,f_2:\Ga\rar{}\ql$ are positive functions, so is $f_1+f_2$.
 \sbr
\item If $f_1,f_2:\Ga\rar{}\ql$ are positive functions and $f_1$ acts nontrivially on a given irreducible representation $\rho$ of $\Ga$ over $\ql$, then so does $f_1+f_2$.
 \sbr
\item If $f:\Ga\rar{}\ql$ is a positive function, then $f=0$ if and only if $f(1)=0$.
 \sbr
\item If $\Ga'\subset\Ga$ is a subgroup and $f:\Ga'\rar{}\ql$ is a positive function, then $\iga(f):\Ga\rar{}\ql$ is also positive.
\end{enumerate}
\end{prop}

Properties (a) and (b) follow from the characterization (i) of positive functions, while properties (c) and (d) follows from the characterization (ii).

\subsection{Proof of Proposition \ref{p:equivalence-def-L-packet-characters}}\label{ss:proof-p:equivalence-def-L-packet-characters}
Fix a perfect unipotent group $G_0$ over $\bF_q$. The fact that the $\bL$-packets of irreducible characters of $G_0$ are nonempty follows from Proposition \ref{p:implies-main-b}, while the fact that they are pairwise disjoint follows from

\begin{lem}\label{l:convolution-geom-min-weak-idemp}
Let $G_0$ be a perfect unipotent group over $\bF_q$, and let $e_0,f_0\in\sD_{G_0}(G_0)$ be nonisomorphic geometrically minimal idempotents. Then\footnote{The reason this fact requires proof is that we defined a geometrically minimal idempotent in $\sD_{G_0}(G_0)$ to be a weak idempotent, which becomes a minimal idempotent after base change to $\bF$.} $e_0*f_0=0$.
\end{lem}

\begin{proof}
Let $G=G_0\tens_{\bF_q}\bF$, and let $e,f\in\sD_G(G)$ be the minimal idempotents obtained from $e_0,f_0$. If $e_0*f_0\neq 0$, then $e*f\neq 0$, whence $e\cong f$ by minimality. Then $e_0\cong f_0$ by Lemma \ref{l:uniqueness-idempotent}.
\end{proof}

To show that the union of the $\bL$-packets of irreducible characters of $G_0$ equals the disjoint union of the sets of irreducible characters of the groups $G_0^\al(\bF_q)$, it suffices to show that if $\rho$ is an irreducible representation of $G_0(\bF_q)$ over $\ql$, then $\rho$ belongs to some $\bL$-packet (Definition \ref{d:L-packets-characters}(a)). By \cite[Thm.~7.1]{characters}, there exists an admissible pair $(H_0,\cL_0)$ for $G_0$ such that the restriction of $\rho$ to $H_0(\bF_q)$ has as a direct summand the $1$-dimensional representation given by $t_{\cL_0}:H_0(\bF_q)\rar{}\ql^\times$. Write $G_0'$ for the normalizer of $(H_0,\cL_0)$ in $G_0$ and let $e'_0\in\sD_{G_0'}(G_0)$ be the Heisenberg minimal idempotent defined by $(H_0,\cL_0)$. Put $e_0\cong\igz e_0'$, a geometrically minimal idempotent in $\sD_{G_0}(G_0)$. We claim that $\rho$ belongs to the $\bL$-packet defined by $e_0$.

\mbr

To this end, observe that by Proposition \ref{p:induction-sheaves-functions}, we have
\[
t_{e_0}=\sum_{\be\in\operatorname{Ker}(H^1(\bF_q,G_0')\rar{}H^1(\bF_q,G_0))} \ind_{G_0'^\be(\bF_q)}^{G_0(\bF_q)}(t_{e_0'^\be}).
\]
Each $t_{e_0'^\be}\in\Fun(G_0'^\be(\bF_q),\ql)^{G_0'^\be(\bF_q)}$ is an idempotent, whence it is a positive function on $G_0'^\be(\bF_q)$ (see \S\ref{ss:positive-functions}). Since $t_{e_0'}$ equals the function $q^{-\dim H_0}\cdot t_{\cL_0}$ extended by zero to $G_0'(\bF_q)$, it follows that $\ind_{G_0'(\bF_q)}^{G_0(\bF_q)}(t_{e_0'})$ acts nontrivially on $\rho$. By parts (a), (b), (d) of Proposition \ref{p:properties-positive-functions}, $t_{e_0}$ also acts nontrivially on $\rho$, as desired.

\mbr

Finally, let us prove the last assertion of Proposition \ref{p:equivalence-def-L-packet-characters}. If $G_0$ is connected, it has no pure inner forms apart from itself. So every geometrically minimal idempotent in $\sD_{G_0}(G_0)$ comes from an admissible pair for $G_0$ by Theorem \ref{t:main}(b). Now the fact that Definition \ref{d:L-packets-characters}(b) is equivalent to the definition of $\bL$-packets of irreducible characters of $G_0(\bF_q)$ given in \cite[Def.~2.7]{characters} follows from \cite[Prop.~9.1(b)]{characters} and \cite[Thm.~2.14]{characters}.

\subsection{An example for Remark \ref{rems:t:main}(b)}\label{ss:example-for-main-theorem} In this subsection we will construct a unipotent group $G_0$ over $\bF_q$ and a geometrically minimal idempotent $e_0\in\sD_{G_0}(G_0)$ such that $t_{e_0}\equiv 0$. Let $U_0\subset GL_3$ be the group of matrices of the form
\[
\left(
\begin{array}{ccc}
1 & a & b \\
0 & 1 & c \\
0 & 0 & 1
\end{array}
\right)
\]
where $a\in\bF_p$ and $b,c$ are arbitrary. Write $H_0\subset U_0$ for the subgroup defined by $a=0$. Then $H_0$ is the neutral connected component of $U_0$ and $U_0$ can be identified with a semidirect product of the finite discrete group $\bZ/p\bZ$ and $H_0$.

\mbr

Moreover, $H_0\cong\bG_a^2$, so the Serre dual of $H_0$ can also be identified with $\bG_a^2$. The induced action of $\bZ/p\bZ=U_0/H_0$ on $H_0^*(\bF_q)$ is nontrivial, so we can choose a multiplicative local system $\cL_0$ on $H_0$, defined over $\bF_q$, which is not $U_0$-invariant.

\mbr

It is clear that $(H_0,\cL_0)$ is an admissible pair for $U_0$, and by construction, its normalizer in $U_0$ is equal to $H_0$. Put $f_0=\ind_{H_0}^{U_0}(\cL_0[4](2))$, which is the geometrically minimal idempotent in $\sD_{U_0}(U_0)$ defined by $(H_0,\cL_0)$. Choose any nontrivial $\al\in H^1(\bF_q,U_0)$ and let $G_0=U_0^\al$ and $e_0=f_0^\al$. Then $t_{e_0}\equiv 0$ by Proposition \ref{p:induction-sheaves-functions}.

\subsection{Proof of the assertion of Remark \ref{rems:t:main}(c)}\label{ss:proof-rems:t:main-c} By Remark \ref{r:absolute-value-square-root-of-p}, there exists $\la\in\bZ[\mu_{p^{2r}}]$ with $\la\cdot\overline{\la}=q$. Let $\sL_\la$ be the local system on $\Spec\bF_q$ corresponding to the vector space of dimension $1$ over $\ql$ on which the geometric Frobenius acts via $\la$ and put $M_0=e_0[-n_e]\tens\pr^*(\sL_\la)^{n_e}$, where $\pr:G_0\rar{}\Spec\bF_q$ is the structure morphism. Then $e$ comes from $M_0$ by base change to $\bF$, and it is clear that we can write $M_0=e_0[-n_e](-n_e/2)\tens\pr^*\sL$ for a suitable pure rank $1$ local system $\sL$ of weight $0$ on $\Spec\bF_q$. In particular, $M_0$ is perverse and pure of weight $0$. Moreover, for each $n\in\bN$ and each $\al\in H^1(\bF_{q^n},G_0\tensqn)$, the function $t_{(e_0\tensqn)^\al}$ on $G_0^\al(\bF_{q^n})$ is a central idempotent under convolution, so it takes values in $\bZ[\mu_{p^r},p^{-1}]$. Hence $t_{(M_0\tensqn)^\al}=(-\la)^{n_e}\cdot t_{(e_0\tensqn)^\al}$ takes values in $\bZ[\mu_{p^{2r}},p^{-1}]$.

\subsection{Proof of Proposition \ref{p:when-min-idemp-comes-from-adm-pair}}\label{ss:proof-p:when-min-idemp-comes-from-adm-pair}
Let us prove that (i) implies (iii). Suppose that $(H_0,\cL_0)$ is an admissible pair for $G_0$ that gives rise to $e_0$. It is clear that one can find some irreducible representation $\rho$ of $G_0(\bF_q)$ over $\ql$ whose restriction to $H_0(\bF_q)$ has as a direct summand the $1$-dimensional representation given by $t_{\cL_0}:H_0(\bF_q)\rar{}\ql^\times$. By the argument in \S\ref{ss:proof-p:equivalence-def-L-packet-characters}, $t_{e_0}$ acts on $\rho$ via the identity, and \emph{a fortiori}, $t_{e_0}\neq 0$.

\mbr

Next, (iii) implies (ii) trivially, so it remains to check that (ii) implies (i).

\mbr

If $t_{e_0}\not\equiv 0$, there is an irreducible representation $\rho$ of $G_0(\bF_q)$ over $\ql$ on which $t_{e_0}$ acts nontrivially. By the argument in \S\ref{ss:proof-p:equivalence-def-L-packet-characters}, there exists a minimal weak idempotent $f_0\in\sD_{G_0}(G_0)$ coming from some admissible pair for $G_0$ such that $t_{f_0}$ acts on $\rho$ via the identity. So $t_{e_0}*t_{f_0}\not\equiv 0$ (here, $*$ denotes the usual convolution of functions on the finite group $G_0(\bF_q)$), and \emph{a fortiori}, $e_0*f_0\neq 0$. Hence $e_0\cong f_0$ by Lemma \ref{l:convolution-geom-min-weak-idemp}, completing the proof of Proposition \ref{p:when-min-idemp-comes-from-adm-pair}.

\subsection{Proof of Lemma \ref{l:averaging-purity}}\label{ss:proof-l:averaging-purity}
We begin with an auxiliary

\begin{lem}\label{l:auxiliary-weights-pullback}
Let $f:X_0\rar{}Y_0$ be a morphism of perfect varieties over $\bF_q$, and suppose that the induced map\footnote{Recall that $\bF$ denotes a fixed algebraic closure of a field with $p=\operatorname{char}\bF_q$ elements, and $\bF_q$ is viewed as a subfield of $\bF$.} $X_0(\bF)\rar{}Y_0(\bF)$ is surjective. If $M\in\sD(Y_0)$ is such that $f^*M\in\sD_{\leq w}(X_0)$ for some $w\in\bZ$, then $M\in\sD_{\leq w}(Y_0)$.
\end{lem}

\begin{proof}
The functor $f^*:\sD(Y_0)\rar{}\sD(X_0)$ is exact with respect to the usual (non-perverse) t-structure. Hence it suffices to show that if $M$ is a $\ql$-sheaf on $Y_0$ such that $f^*M$ is punctually pure of weight $w$, then so is $M$. Let $n\geq 1$ and $y\in Y_0(\bF_{q^n})$. By assumption, there exist $m\geq 1$ and $x\in X_0(\bF_{q^m})$ such that $n$ divides $m$ and $f(x)=y$. The geometric Frobenius $F_{q^m}$ acting on $x^*(f^*M)$ can be identified with the $(m/n)$-th power of the geometric Frobenius $F_{q^n}$ acting on $y^*M$. Hence the lemma follows from the fact that if $a\in\ql$ is such that $a^{m/n}$ is a Weil number of absolute value $(q^m)^{w/2}$, then $a$ is a Weil number of absolute value $(q^n)^{w/2}$.
\end{proof}

\begin{cor}\label{c:auxiliary-weights-pullback}
Let $G_0$ be a perfect $($unipotent\footnote{The unipotence assumption is irrelevant here; we only impose it because we did not discuss the correct definition of the equivariant derived category $\sD_{G_0}(Y_0)$ in the non-unipotent case.}$)$ group over $\bF_q$ acting on a perfect variety $Y_0$ over $\bF_q$, and let $\al:G_0\times Y_0\rar{}Y_0$ denote the action map. Suppose $X_0$ is another perfect variety over $\bF_q$ and $f:X_0\rar{}Y_0$ is a morphism such that the induced map
\[
G_0\times X_0 \xrar{\ \ \id\times f\ \ } G_0\times Y_0 \xrar{\ \ \al\ \ } Y_0
\]
is surjective at the level of $\bF$-points. If $M\in\sD_{G_0}(Y_0)$ is such that $f^*M\in\sD_{\leq w}(X_0)$ for some $w\in\bZ$, then $M\in\sD_{\leq w}(Y_0)$.
\end{cor}

\begin{proof}
Let $\pi:G_0\times X_0\rar{}X_0$ denote the second projection. Since $f^*M\in\sD_{\leq w}(X_0)$, we have $\pi^*f^*M\in\sD_{\leq w}(G_0\times X_0)$. The $G_0$-equivariant structure on $M$ yields an isomorphism $(\id\times f)^*\al^*M\rar{\simeq}\pi^*f^*M$. In view of our assumption, Lemma \ref{l:auxiliary-weights-pullback} implies that $M\in\sD_{\leq w}(Y_0)$.
\end{proof}

\begin{proof}[Proof of Lemma \ref{l:averaging-purity}]
Apply Corollary \ref{c:auxiliary-weights-pullback} to $Y_0=(G_0/G_0')\times X_0$ and $f=i$.
\end{proof}

\subsection{Equivalence of two definitions of $\ig$}\label{ss:equivalence-def-induction} In this subsection we justify the assertion of Remark \ref{r:equivalent-def-induction}. Consider the diagram
\begin{equation}\label{e:diagram-for-ind}
\xymatrix{
  G' \ar[d]_j \ar[rr]^i & & \Gt \ar[d]^J \ar[drr]^{\pi} & & \\
  G \ar[rr]^{I\ \ \ \ \ \ \ } & & (G/G')\times G \ar[rr]^{\ \ \ \ \ \pr_2} & & G
   }
\end{equation}
where the notation is explained as follows. The variety $\Gt$ and the morphisms $i$ and $\pi$ are defined as in Remark \ref{r:equivalent-def-induction}. The map $j$ is the natural inclusion of $G'$ into $G$. The map $I$ is given by $g\mapsto (\overline{1},g)$ and the map $J$ is induced by the map $G\times G'\rar{}(G/G')\times G$ given by $(g,g')\mapsto(\overline{g},gg'g^{-1})$, where $\overline{g}$ denotes the image of $g$ in $G/G'$. Finally, $\pr_2$ is the second projection.

\mbr

The functors
\[
i^*:\sD_G(\Gt)\rar{\sim}\sD_{G'}(G') \qquad \text{and} \qquad I^*:\sD_G\bigl((G/G')\times G\bigr)\rar{\sim}\sD_{G'}(G)
\]
are equivalences. Unraveling Definition \ref{d:induction-functors}, we see that the functor $\ig$ is defined as the composition $\pr_{2!}\circ (I^*)^{-1}\circ j_!$. On the other hand, in \cite{characters} the functor $\ig$ was defined as the composition $\pi_!\circ(i^*)^{-1}$. We need to check that these two compositions are isomorphic. To this end, observe that diagram \eqref{e:diagram-for-ind} is commutative, and, in addition, the square on the left is cartesian. This implies that $\pr_{2!}\circ J_!\cong\pi_!$ and, by the proper base change theorem, $I^*\circ J_!\cong j_!\circ i^*$. Therefore
\[
\pr_{2!}\circ (I^*)^{-1}\circ j_! \cong \pr_{2!}\circ J_!\circ (i^*)^{-1} \cong \pi_!\circ(i^*)^{-1},
\]
as required.

\subsection{Proof of Proposition \ref{p:conjugation-action}}\label{ss:proof-p:conjugation-action}
We first verify the uniqueness assertion. Suppose $\vp:G_1\rar{\simeq}\cG_1$ is an isomorphism of perfect groups over $\bF_q$ satisfying the requirement of the proposition. Let $X_0\in G_0\var$ be $G_0$ equipped with the action of $G_0$ given by left multiplication. The corresponding $X_1\in G_1\var$ can be identified with $P$ viewed as a left $G_1$-torsor. Let
\[
a:G_1\times P\rar{}P \qquad\text{and}\qquad \al:\cG_1\times P\rar{}P
\]
denote the action maps, and let
\[
p_2:G_1\times P\rar{}P \qquad\text{and}\qquad \pi_2:\cG_1\times P\rar{}P
\]
denote the second projections. By assumption, $\al\circ(\vp\times\id_P)=a$, and we also have $\pi_2\circ(\vp\times\id_P)=p_2$. Hence the following diagram commutes:
\[
\xymatrix{
  G_1\times P \ar[dr]_{(a,p_2)} \ar[rr]^{\vp\times\id_P} & & \cG_1\times P \ar[dl]^{(\al,\pi_2)} \\
   & P\times P &
   }
\]
However, in this diagram all arrows are isomorphisms because $P$ is a left $G_1$-torsor. This implies that $\vp\times\id_P=(\al,\pi_2)^{-1}\circ(a,p_2)$, and this identity determines the morphism $\vp$ uniquely.

\mbr

For the existence, one could construct $\vp$ by reversing the proof of uniqueness given above, but it is more convenient to proceed directly as follows. Consider the diagram
\[
G_1\times P \xrightarrow{\ \ \simeq\ \ } P\times P \xleftarrow{\ \ \simeq\ \ } P\times G_0
\]
where the isomorphism on the left is given by $(g_1,p)\mapsto(g_1\cdot p,p)$ and the isomorphism on the right is given by $(p,g_0)\mapsto(p\cdot g_0,p)$. Both of these isomorphisms are equivariant with respect to the left action of $G_1$ and the right action of $G_0$, where:
 \sbr
\begin{itemize}
\item $G_0$ acts on $P$ on the right in the given way\footnote{Recall that, by assumption, $P$ is a $(G_1,G_0)$-bitorsor.};
 \sbr
\item $G_1$ acts on $P$ on the left in the given way;
 \sbr
\item $G_0$ acts on itself via $\ga_0:g_0\mapsto\ga_0^{-1}g\ga$ and acts trivially on $G_1$;
 \sbr
\item $G_1$ acts on itself via $\ga_1:g_1\mapsto\ga_1 g_1\ga_1^{-1}$ and acts trivially on $G_0$.
\end{itemize}

\mbr

Taking the quotient by the right action of $G_0$, we obtain an isomorphism \[(G_1\times P)/G_0 \rar{\simeq} (P\times G_0)/G_0\] in the category $G_1\var$. But $(G_1\times P)/G_0$ is canonically identified with $G_1$, and by construction, $\cG_1=(P\times G_0)/G_0$. Thus we obtain an isomorphism $G_1\rar{\simeq}\cG_1$. It is straightforward to check that this is an isomorphism of perfect groups over $\bF_q$ and that the requirement of Proposition \ref{p:conjugation-action} is satisfied.

\subsection{Proof of Proposition \ref{p:induction-sheaves-functions}}\label{ss:proof-p:induction-sheaves-functions} We will consider two cases separately. Suppose first that there exists $\be\in H^1(\bF_q,G_0')$ that maps to $\al\in H^1(\bF_q,G_0)$. In this case, without loss of generality, we may replace $G_0'\subset G_0$ with $G_0'^\be\subset G_0^\al$ and $M$ with $M^\be$, and hence we may assume that $\al$ is trivial. (Here we implicitly used Lemma \ref{l:induction-transport}.) Then Proposition \ref{p:induction-sheaves-functions} follows from Remark \ref{r:equivalent-def-induction} and \cite[Prop.~6.13]{characters}.

\mbr

Next suppose that $\al$ is not in the image of the natural map $H^1(\bF_q,G_0')\rar{}H^1(\bF_q,G_0)$. In this case we must prove that $t_{N^\al}\equiv 0$ on $G_0^\al(\bF_q)$.

\mbr

It will be more convenient to work with the definition of the functor $\igz$ used in \cite{characters}, which was recalled in Remark \ref{r:equivalent-def-induction}. Thus let $\Gt_0$ denote the quotient $(G_0\times G_0')/G_0'$, where the right action of $G_0'$ is given by $(g,g')\cdot\ga'=(g\ga',\ga'^{-1}g'\ga')$. If $i:G_0'\into\Gt_0$ is induced by $g'\mapsto(1,g')$ and $\pi:\Gt_0\rar{}G_0$ is induced by $(g,g')\mapsto gg'g^{-1}$, then, by Remark \ref{r:equivalent-def-induction}, we have $N\cong\pi_!\bigl((i^*)^{-1}(M)\bigr)$. Hence, by Lemma \ref{l:functoriality-transport}, we have $N^\al\cong\pi^\al_!\bigl(((i^*)^{-1}(M))^\al\bigr)$. If we can show that $\Gt_0^\al(\bF_q)=\varnothing$, then the proof will be complete, in view of the Grothendieck-Lefschetz trace formula.

\begin{lem}
With the notation above, if $\al$ is not in the image of the natural map $H^1(\bF_q,G_0')\rar{}H^1(\bF_q,G_0)$, then $\Gt_0^\al(\bF_q)=\varnothing$.
\end{lem}

\begin{proof}
Let $P$ be a right $G_0$-torsor whose isomorphism class equals $\al$. Unraveling the definition of $\Gt_0^\al$, we see that it can be identified with the quotient
\[
(P\times G_0\times G_0') / (G_0\times G_0'),
\]
where the right action of $G_0\times G_0'$ on $P\times G_0\times G_0'$ is given by
\[
(p,g,g')\cdot(\ga,\ga') = (p\cdot\ga,\ga^{-1}g\ga',\ga'^{-1}g'\ga').
\]
Suppose, for the sake of contradiction, that $\Gt_0^\al(\bF_q)\neq\varnothing$. This means that there exist $p\in P(\bF)$, $g,\ga\in G_0(\bF)$ and $g',\ga'\in G_0'(\bF)$ such that
\[
(\Fr_q(p),\Fr_q(g),\Fr_q(g')) = (p\cdot\ga,\ga^{-1}g\ga',\ga'^{-1}g'\ga').
\]
The identity above implies that $\ga=g\ga'\Fr_q(g)^{-1}$, so that
\[
\Fr_q(p)=p\cdot\ga = p\cdot g\ga'\Fr_q(g)^{-1},
\]
and therefore $\Fr_q(p\cdot g)=(p\cdot g)\cdot\ga'$. However, the last equality means that $P$ comes from a right $G_0'$-torsor, which contradicts the assumption that $\al$ is not in the image of the natural map $H^1(\bF_q,G_0')\rar{}H^1(\bF_q,G_0)$.
\end{proof}

\subsection{Proof of Lemma \ref{l:induction-transport}}\label{ss:proof-l:induction-transport} Recall that $\Igz$ is right adjoint to the restriction functor $\Res^{G_0}_{G'_0}:\sD_{G_0}(G_0)\rar{}\sD_{G_0'}(G_0')$. This implies that the functors
\[
\sD_{G_0'}(G_0') \rar{} \sD_{G_0^\al}(G_0^\al), \qquad M\longmapsto (\Igz M)^\al \quad\text{and}\quad M\longmapsto \Ind_{G_0'^\be}^{G_0^\al} (M^\be),
\]
are isomorphic, because they have isomorphic left adjoints. Now Lemma \ref{l:induction-transport} follows from Proposition \ref{p:transport-duality} and the last assertion of Proposition \ref{p:induction-duality}.

\subsection{Proof of Proposition \ref{p:transport-duality}}\label{ss:proof-p:transport-duality}
Fix $X\in G_0\var$ and $\al\in H^1(\bF_q,G_0)$.
Let $(G_1,P)$ be a pure inner form of $G_0$ such that $P$, as a right $G_0$-torsor, represents the cohomology class $\al$. Put $X_1=(P\times X)/G_0$, where the right action of $G_0$ is given by $(p,x)\cdot\ga=(p\cdot\ga,\ga^{-1}\cdot x)$. The left action of $G_1$ on $P$ induces a left action of $G_1$ on $X_1$, and we can identify $(G_0^\al,X^\al)$ with $(G_1,X_1)$. If $\varpi:P\times X\rar{}X_1$ is the quotient map and $\pi_2:P\times X\rar{}X$ is the second projection, then the functors
\[
\varpi^*:\sD_{G_1}(X_1)\rar{\sim}\sD_{G_1\times G_0}(P\times X) \qquad\text{and}\qquad \pi_2^*:\sD_{G_0}(X)\rar{\sim}\sD_{G_1\times G_0}(P\times X)
\]
are equivalences, and the transport functor $\sD_{G_0}(X)\rar{\sim}\sD_{G_1}(X_1)$ is defined as $(\varpi^*)^{-1}\circ \pi_2^*$. Both $\varpi$ and $\pi_2$ are smooth morphisms of relative dimension $d=\dim P$, which implies that $\bD_{P\times X}\circ \varpi^*\cong \varpi^*\circ\bD_{X_1}[2d](d)$ and $\bD_{P\times X}\circ \pi_2^*\cong \pi_2^*\circ\bD_{X}[2d](d)$. This yields assertion (a) of the proposition.

\mbr

Assertion (b) follows from part (a), Lemma \ref{l:functoriality-transport}, and the fact that the inner form $\iota^\al:G_0^\al\rar{}G_0^\al$ of the inversion map $\iota:G_0\rar{}G_0$ (given by $g\mapsto g^{-1}$) is naturally identified with the inversion map for the group $G_0^\al$.

\subsection{Proof of Proposition \ref{p:transport-weights}}\label{ss:proof-p:transport-weights}
Fix a pure inner form $(G_1,P)$ of $G_0$ and put $X_1=(P\times X)/G_0$, where the right $G_0$-action on $P\times X$ is given by $(p,x)\cdot g=(p\cdot g,g^{-1}\cdot x)$. Write $\varpi:P\times X\rar{}X_1$ for the quotient morphism and $\pi_2:P\times X\rar{}X$ for the second projection.

\mbr

If $M\in\sD_{G_0}(X)$, then, by definition, the object $M_1\in\sD_{G_1}(X_1)$ obtained from $M$ via transport of equivariant complexes (Definition \ref{d:transport-equivariant-complexes}) is determined uniquely (up to isomorphism) by the requirement that
$\pi_2^* M \cong \varpi^* M_1$ in $\sD_{G_1\times G_0}(P\times X)$
(here, as usual, $G_1$ acts on $P\times X$ via its given action on $P$ and the trivial action on $X$).

\mbr

Now if $M\in\sD_{\leq w}(X)$, then $\pi_2^*M\in\sD_{\leq w}(P\times X)$. Since $\varpi$ is surjective at the level of $\bF$-points, Lemma \ref{l:auxiliary-weights-pullback} implies that $M_1\in\sD_{\leq w}(X_1)$.

\mbr

The fact that if $M\in\sD_{\geq w}(X)$, then $M_1\in\sD_{\geq w}(X_1)$ follows from what we just proved, using the definition of $\sD_{\geq w}$ (cf.~Definition \ref{d:purity}(e)) and Proposition \ref{p:transport-duality}(a).

\subsection{Proof of Proposition \ref{p:equivariant-sheaves-point}}\label{ss:proof-p:equivariant-sheaves-point} Choose representatives $\ga_1\cdot F_q, \ga_2\cdot F_q,\dotsc,\ga_r\cdot F_q$ of the conjugacy classes\footnote{We remark that if $\ga,\ga'\in\Ga$, then $\ga\cdot F_q\in\widetilde{\Ga}$ and $\ga'\cdot F_q\in\widetilde{\Ga}$ are $\Ga$-conjugate if and only if they are $\widetilde{\Ga}$-conjugate. So there is no reason to distinguish between the two kinds of conjugacy.} in $\widetilde{\Ga}=\Gal(\bF/\bF_q)\ltimes\Ga$ that project onto $F_q$ in $\Gal(\bF/\bF_q)$, and write $\al_1,\dotsc,\al_r\in H^1(\bF_q,G_0)$ for the corresponding cohomology classes. Then
\[
t_{M^{\al_i}}(*) = \sum_{j\in\bZ} (-1)^j\cdot \tr\bigl( \ga_i\cdot F_q; H^j(V^\bullet) \bigr),
\]
and hence we are reduced to proving that if $W$ is a finite dimensional continuous representation of $\widetilde{\Ga}$ over $\ql$, then
\begin{equation}\label{e:need-auxiliary}
\sum_{i=1}^r \frac{q^{\dim G_0}}{\abs{G_0^{\al_i}(\bF_q)}} \cdot \tr \bigl( \ga_i\cdot F_q; W \bigr) = \tr \bigl( F_q; W^\Ga \bigr).
\end{equation}

Now recall that if $H$ is a perfect connected unipotent group over $\bF_q$, then $\abs{H(\bF_q)}=q^{\dim H}$. By Lang's theorem \cite{lang}, if $H$ is an arbitrary (not necessarily connected) perfect group over $\bF_q$, the sequence
\[
1 \rar{} H^\circ(\bF_q) \rar{} H(\bF_q) \rar{} (H/H^\circ)(\bF_q) \rar{} 1
\]
is exact, which implies that if $H$ is unipotent, then $\abs{H(\bF_q)}=q^{\dim H}\cdot\abs{(H/H^\circ)(\bF_q)}$. We already saw that $\pi_0(G_0^{\al_i})$ can be identified with the group $\Ga$ equipped with the new geometric Frobenius action given by $x\mapsto \ga_i F_q(x)\ga_i^{-1}$. In particular, the group $\pi_0(G_0^{\al_i})$ can be identified with the centralizer $Z_{\Ga}(\ga_i\cdot F_q)$ of $\ga_i\cdot F_q$ in $\Ga$. Hence the desired identity \eqref{e:need-auxiliary} can be rewritten as
\begin{equation}\label{e:need-auxiliary-2}
\sum_{i=1}^r \frac{1}{\abs{Z_{\Ga}(\ga_i\cdot F_q)}} \cdot \tr \bigl( \ga_i\cdot F_q; W \bigr) = \tr \bigl( F_q; W^\Ga \bigr).
\end{equation}

Recalling that the elements $\ga_i\cdot F_q$ were chosen as representatives of $\Ga$-conjugacy classes of all elements of the form $\ga\cdot F_q$ in $\widetilde{\Ga}$, where $\ga\in\Ga$, we see that the left hand side of \eqref{e:need-auxiliary-2} equals $\abs{\Ga}^{-1}\cdot \tr\bigl( \sum_{\ga\in\Ga} \ga\cdot F_q; W\bigr)$. Now \eqref{e:need-auxiliary-2} follows from the fact that the element $\abs{\Ga}^{-1}\cdot\sum_{\ga\in\Ga}\ga$ in the group algebra of $\Ga$ acts as a projector onto the subspace $W^\Ga$ of $\Ga$-invariants (for any representation $W$ of $\Ga$).

\subsection{Proof of Proposition \ref{p:homog}}\label{ss:proof-p:homog}
Let us first reduce the proposition to the case where $X_0=\Spec\bF_q$.
Since $X_0$ is nonempty and $U_0$ acts transitively on $X_0$,
there exists $\al\in H^1(\bF_q,U_0)$ such that $X_0^\al(\bF_q)\neq\varnothing$.
Without loss of generality we may replace $U_0$ and $X_0$ by the pure inner
forms $U_0^\al$ and $X_0^\al$, and thus assume that $X_0(\bF_q)\neq\varnothing$.
Fix a point $x\in X_0(\bF_q)$ and let $U_0^x\subset U_0$
denote the stabilizer of $x$.

\mbr

If we view $x$ as a morphism $\Spec\bF_q\rar{}X_0$, then, since $U_0$ acts transitively on $X_0$, we obtain an equivalence of stacks $(U_0^x)\backslash\Spec\bF_q\rar{\sim}U_0\backslash X_0$. Since all the assertions of Proposition \ref{p:homog} only depend of the quotient stack $U_0\backslash X_0$, we see that if the proposition is true when $X_0=\Spec\bF_q$, then it is true in general\footnote{It is straightforward to rephrase this argument without using the language of stacks. However, it becomes more cumbersome, so we skip it.}.

\mbr

It remains to prove Proposition \ref{p:homog} in the case where $X_0=\Spec\bF_q$.
Here the result can be reformulated more concretely as a statement about
representations of finite groups. Namely, let $\Ga=\pi_0(U_0)(\bF)$,
a finite group equipped with a (continuous) action of $\Gal(\bF/\bF_q)$,
and form the semidirect product $\widetilde{\Ga}=\Gal(\bF/\bF_q)\ltimes\Ga$.

\mbr

The category $\Loc_U(\Spec\bF)$ is naturally identified with the category
$\Rep(\Ga,\ql)$
of finite dimensional representations of $\Ga$ over $\ql$, so that the
action of $\Fr_q^*$ on $\Loc_U(\Spec\bF)$ becomes identified with the automorphism
of $\Rep(\Ga,\ql)$ induced by the action of the geometric Frobenius
$F_q\in\Gal(\bF/\bF_q)$ on $\Ga$. On the other hand, the category $\Loc_{U_0}(\Spec\bF_q)$
is naturally identified with the category $\Rep(\widetilde{\Ga},\ql)$ of continuous
finite dimensional representations of $\widetilde{\Ga}$ over $\ql$.

\begin{prop}\label{p:extends-finite-image}
Every $F_q$-invariant irreducible representation of $\Ga$ over $\ql$ can be extended to a continuous representation of $\widetilde{\Ga}$ with finite image.
\end{prop}

\begin{proof}
Let $\rho:\Ga\rar{}GL(V)$ be an $F_q$-invariant irreducible representation of $\Ga$ over $\ql$. Then there exists a linear automorphism $\vp:V\rar{\simeq}V$ with $\rho(F_q(\ga))=\vp\circ\rho(\ga)\circ\vp^{-1}$ for all $\ga\in\Ga$. If $N$ is the order of the automorphism of $\Ga$ given by $F_q$, then $\vp^N$ commutes with $\rho(\Ga)$, whence $\vp^N$ is scalar by Schur's lemma. Rescaling $\vp$, we may assume that $\vp^N=\id_V$. Then $\rho$ extends to a continuous representation $\widetilde{\rho}$ of $\widetilde{\Ga}$ on $V$ determined by $\widetilde{\rho}(F_q)=\vp$, and by construction, $\widetilde{\rho}$ has finite image.
\end{proof}

We now state two other results used in the proof of Proposition \ref{p:homog}. They will be proved in \S\S\ref{ss:proof-p:twisted-characters-form-a-basis}--\ref{ss:proof-p:drinfeld-representations}, following unpublished notes by Drinfeld.

\begin{prop}\label{p:twisted-characters-form-a-basis}
For every $F_q$-invariant irreducible representation $\rho$ of $\Ga$ over $\ql$,
choose some extension $\widetilde{\rho}$ of $\rho$ to a continuous
representation of $\widetilde{\Ga}$ with finite image,
and form the function $\widetilde{\chi}_\rho:\Ga\rar{}\qab$ defined
by $\widetilde{\chi}_\rho(\ga)=\tr\bigl(\widetilde{\rho}(\ga\cdot F_q)\bigr)$.
Then the functions $\widetilde{\chi}_\rho$ form a basis for the space
of functions $\Ga\rar{}\qab$ that are invariant under $F_q$-conjugation.
\end{prop}

To explain the terminology,
we recall that $F_q$-conjugation is the action of $\Ga$ on itself defined by
$\ga:x\mapsto F_q(\ga)x\ga^{-1}$. The set of orbits for this action is identified
with the cohomology $H^1(\Gal(\bF/\bF_q),\Ga)\cong H^1(\bF_q,U_0)$. Similarly, for each $n\in\bN$, the cohomology $H^1(\bF_{q^n},U_0)$ is identified with the set of $F_q^n$-conjugacy classes in $\Ga$.

\begin{prop}\label{p:drinfeld-representations}
Suppose that $\Ga$ has exponent $p^r$. Every $F_q$-invariant irreducible representation of $\Ga$ over $\ql$ can be extended to a continuous representation of $\widetilde{\Ga}$, which has finite image and whose character takes values in $\bZ[\mu_{p^r},p^{-1}]\subset\ql$.
\end{prop}

Let us show how Proposition \ref{p:homog} follows from the results above. Suppose $\rho$ is an $F_q$-invariant irreducible representation of $\Ga$ over $\ql$ and let $\widetilde{\rho}$ be an extension of $\rho$ to a representation of $\widetilde{\Ga}$ with finite image. If $\Ga$ has exponent $p^r$, assume moreover that the character of $\widetilde{\rho}$ takes values in $\bZ[\mu_{p^r},p^{-1}]\subset\ql$.

\mbr

Let $\cL_0\in\Loc_{U_0}(\Spec\bF_q)$ correspond to $\widetilde{\rho}$. Choose $\ga\in\Ga$ and $n\in\bN$, and let $\al\in H^1(\bF_{q^n},U_0)$ correspond to the $F_q^n$-conjugacy class of $\ga$ in $\Ga$. Then the value of the trace-of-Frobenius function $t_{(\cL_0\tensqn)^\al}$ on the $1$-element set $X_0^\al(\bF_{q^n})$ is equal to the trace of $\widetilde{\rho}(\ga\cdot F_q^n)$ (this is a special case
of the remarks in \S\ref{ss:equivariant-sheaves-point}). Since $\widetilde{\rho}$ has finite image, we see that $\cL_0$ is pure of weight $0$ and $\operatorname{trace}\bigl(\widetilde{\rho}(\ga\cdot F_q^n)\bigr)\in\qab$, whence Proposition \ref{p:homog}(a) results from Proposition \ref{p:extends-finite-image}. Similarly, Proposition \ref{p:homog}(a$'$) follows from Proposition \ref{p:drinfeld-representations}. Furthermore, if $n=1$, the value of $t_{\cL_0^\al}$ on the $1$-element set $X_0^\al(\bF_q)$
is equal to $\widetilde{\chi}_\rho(\ga)$, so Proposition \ref{p:homog}(b) follows from Proposition \ref{p:twisted-characters-form-a-basis}.

\subsection{Proof of Corollary \ref{c:homog}}\label{ss:proof-c:homog}
Write $A^*=\Hom(A,\ql^\times)$. If $\sC$ is any
$\ql$-linear Karoubi complete\footnote{This means that if $M\in\sC$ and $P\in\End_{\sC}(M)$
satisfies $P^2=P$, then $P$ has a kernel.} category and $M$ is an object of $\sC$ equipped
with an action of $A$ by automorphisms, we get a canonical decomposition
$M=\bigoplus_{\chi\in A^*}M^\chi$, where each $M^\chi$ is a direct summand of $M$ on
which $A$ acts via $\chi$.

\mbr

In particular, we obtain decompositions
\[
\Loc_\Ut(X) = \bigoplus_{\chi\in A^*} \Loc_\Ut^\chi(X) \qquad\text{and}\qquad
\Loc_{\Ut_0}(X_0) = \bigoplus_{\chi\in A^*} \Loc_{\Ut_0}^\chi(X_0)
\]
(direct sums of $\ql$-linear categories). Moreover, $\IrrLoc_\Ut(X)$ becomes
identified with the disjoint union $\bigcup_{\chi\in A^*}\IrrLoc_\Ut^\chi(X)$
(in a way compatible with the action of $\Fr_q^*$), so parts
(a)--(a$'$) of Corollary \ref{c:homog} follow from Proposition \ref{p:homog}(a)--(a$'$).

\mbr

Next, by Remark \ref{rems:aux}(3), the connecting homomorphism $\de:U_0(\bF_q)\rar{}A$
used in Definition \ref{d:aux} is surjective, and its kernel is equal to the image
of $\Ut_0(\bF_q)\rar{}U_0(\bF_q)$. Hence we obtain a decomposition
\[
\Fun(X_0(\bF_q),\qab)^{\Ut_0(\bF_q)} = \bigoplus_{\chi\in A^*}
\Fun(X_0(\bF_q),\qab)^{U_0(\bF_q),\chi}.
\]
The same argument applies to each pure inner form of $\Ut_0$, so we
deduce that
\[
\Fun\bigl((\Ut_0\backslash X_0)(\bF_q)\bigr) = \bigoplus_{\chi\in A^*} \Fun(X_0)^{U_0,\chi}
\]
Since for every $\chi\in A^*$ the elements \eqref{e:basis} belong to
the subspace $\Fun(X_0)^{U_0,\chi}\subset\Fun(X_0)^{\Ut_0}$, we see that
part (b) of Corollary \ref{c:homog} follows from Proposition \ref{p:homog}(b)
together with the remarks in the first part of the proof and Lemma \ref{l:basis-direct-sum}.

\subsection{Proof of Proposition \ref{p:twisted-characters-form-a-basis} (V.~Drinfeld)}\label{ss:proof-p:twisted-characters-form-a-basis}
Let $\Fun(\Ga)$ be the space of functions $\Ga\rar{}\ql$ and let
$A=\ql[\Ga]$ be the group algebra of $\Ga$ over $\ql$. The space $\Fun(\Ga)$ is naturally identified with $\Hom_{\ql}(A,\ql)$, i.e., the space of linear functionals on $A$ (namely, $\la\in\Hom_{\ql}(A,\ql)$ corresponds to
the function $f\in\Fun(\Ga)$ which is the restriction of $\lambda:A\rar{}\ql$ to $\Ga\subset A$). The algebra automorphism of $A$
induced by $F_q\in\Aut(\Ga)$ will be denoted by $\phi$.

\begin{lem}\label{l:twisted-invariant}
A function $f\in\Fun(\Ga)$ is invariant under $F_q$-conjugation if and only if
the corresponding linear functional $\la:A\rar{}\ql$ satisfies the
identity
\begin{equation}\label{e:twtr}
\la(\phi(a)b)=\la(ba) \mbox{ for all } a,b\in A.
\end{equation}
\end{lem}
\begin{proof}
$f$ is invariant under $F_q$-conjugation if and only if
$f(F_q(\ga)x\ga^{-1})=f(x)$ for all $x,\ga\in\Ga$, which is
equivalent to the condition that $f(\phi(\ga)y)=f(y\ga)$ for all
$y,\ga\in\Ga$ (make the change of variables $x\ga^{-1}=y$). Now use
the linearity of $\la$ and $\phi$.
\end{proof}

Proposition \ref{p:twisted-characters-form-a-basis} follows from Lemma \ref{l:twisted-invariant} and
the next fact applied to $A=\ql[\Ga]$.

\begin{prop}\label{p:tw-conjug-alg}
Let $A$ be a finite dimensional semisimple associative algebra over
$\ql$ equipped with an automorphism $\phi:A\rar{\simeq}A$. Let
$\widehat{A}$ denote the set of isomorphism classes of simple $A$-modules. For each
$\rho\in\bigl(\widehat{A}\bigr)^{\phi}$ choose a realization
$\rho:A\rar{}\End(V_\rho)$ and an invertible operator $\phi_\rho:V_\rho\rar{}V_\rho$ so that
\begin{equation}\label{e:intertw}
\phi_\rho\circ\rho(\phi(a))=\rho(a)\circ\phi_\rho \text{ for all }
a\in A.
\end{equation}
Then the functionals $\widetilde{\chi}_\rho:A\rar{}\ql$,
$\rho\in\bigl(\widehat{A}\bigr)^{\phi}$, defined by
$\widetilde{\chi}_\rho (a):=\tr(\phi_{\rho}\circ\rho(a))$ form a
basis of the space of all linear functionals $\la:A\rar{}\ql$
satisfying \eqref{e:twtr}.
\end{prop}

\begin{proof}
If the proposition holds for two pairs $(A_1,\phi_1)$ and $(A_2,\phi_2)$, then it also holds for $(A_1\times A_2,\phi_1\times\phi_2)$. Since any
semisimple associative $\ql$-algebra $A$ is a product of matrix
algebras $A_{\rho}$, where $\rho\in\widehat{A}$, it suffices to consider two cases:
\begin{enumerate}[$($i$)$]
\item the case where $\bigl(\widehat{A}\bigr)^{\phi}=\varnothing$, and
 \sbr
\item the case where $A$ has only one simple module.
\end{enumerate}
In the first case we have to prove that any $\la\in\Hom_{\ql}(A,\ql)$
satisfying \eqref{e:twtr} equals $0$. Choose $\rho\in\widehat{A}$ and let $e_{\rho}\in A$ be such
that the image of $e_{\rho}$ in $A_{\rho'}$ equals $1$ if
$\rho'=\rho$ and $0$ if $\rho'\ne\rho$. Apply \eqref{e:twtr} for
$a=e_{\rho}$ and $b\in e_{\rho}A$. Then $\phi(a)b=\phi(e_{\rho})b=e_{\phi(\rho)}b=e_{\phi(\rho)}e_{\rho}b=0$ because $\phi(\rho)\neq\rho$ by assumption. So $\la(b)=0$ for all $b\in e_{\rho}A$. Since $\rho\in\widehat{A}$ is arbitrary, we conclude that $\la=0$.

\mbr

It remains to consider the case where $\widehat{A}$ has only one element, call it $\rho$. Then $\rho:A\rar{}\End_{\ql}(V_{\rho})$ is an isomorphism. Let us use it to identify $A$ with $\End_{\ql}(V_{\rho})$. Then
\eqref{e:intertw} can be rewritten as $\phi(a)=\phi_\rho^{-1}a\phi_\rho$, and we have to prove that any linear functional $\la:\End(V_{\rho})\rar{}\ql$ such that
\begin{equation}\label{e:trtw}
\la (\phi_\rho^{-1}a\phi_\rho b)=\la(ba) \text{ for all } a,b\in\End(V_\rho)
\end{equation}
is proportional to the functional $\widetilde{\chi}_\rho:\End(V_{\rho})\rar{}\ql$ defined by $\widetilde{\chi}_\rho(a):=\tr(\phi_{\rho} a)$. Put $\mu(a):=\la(\phi_{\rho}^{-1}a)$. Then \eqref{e:trtw} is equivalent to the condition that $\mu(aa')=\mu(a'a)$ for all $a,a'\in\End(V_{\rho})$. So $\mu$ is proportional to $\tr$ and $\la$ is proportional to $\widetilde{\chi}_\rho$.
\end{proof}

\subsection{Proof of Proposition \ref{p:drinfeld-representations} (V.~Drinfeld)}\label{ss:proof-p:drinfeld-representations} Throughout this subsection we write $A=\bZ[\mu_{p^r},p^{-1}]$ and $E=\bQ(\mu_{p^r})$ for brevity. Since $p^r$ is the exponent of $\Ga$, a general fact from character theory of finite groups implies that the group algebra $E[\Ga]$ is isomorphic to a product of matrix algebras over $E$. Thus extension of scalars from $E$ to $\ql$ is a bijection between the sets of isomorphism classes of irreducible representations of $\Ga$ over $E$ and over $\ql$.

\mbr

Consider the groupoid $\cC$ whose objects are free finite rank $A$-modules $M$ equipped with an action of $\Ga$ and a $\Ga$-invariant Hermitian form $\langle\cdot,\cdot\rangle:M\times M\rar{}A$, and whose morphisms are $\Ga$-equivariant isometries.

\mbr

Let $\overline{\cC}\subset\cC$ be the full subcategory consisting of objects $M$ such that
\begin{itemize}
\item the $\Ga$-module $E\tens_A M$ is irreducible, and
 \sbr
\item for every subgroup $H\subset\Ga$ and every character $\chi:H\rar{}E^\times$ such that $E\tens_A M\cong\Ind_H^\Ga\chi$ as $\Ga$-representations, the $A$-submodule \[M^{H,\chi}:=\bigl\{ m\in M \st hm=\chi(h)m\ \forall\,h\in H\bigr\}\] has a generator $m^{H,\chi}$ such that $\langle m^{H,\chi},m^{H,\chi}\rangle=1$.
\end{itemize}

\mbr

We will deduce Proposition \ref{p:drinfeld-representations} from the following result.

\begin{prop}\label{p:drinfeld-2}
\begin{enumerate}[$($a$)$]
\item Every irreducible representation of $\Ga$ over $\ql$ is isomorphic to $\ql\tens_A M$ for some $M\in\overline{\cC}$, and $M$ is determined uniquely up to isomorphism.
 \sbr
\item The automorphism group of any $M\in\overline{\cC}$ consists of multiplications by roots of unity in $E$.
\end{enumerate}
\end{prop}

To see that Proposition \ref{p:drinfeld-2} implies Proposition \ref{p:drinfeld-representations}, let $\rho$ be an irreducible representation of $\Ga$ over $\ql$ whose isomorphism class is invariant under $F_q$. Choose $M\in\overline{\cC}$ so that $\rho\cong\ql\tens_A M$. By the uniqueness assertion of Proposition \ref{p:drinfeld-2}, the isomorphism class of $M$ is also invariant under $F_q$. So there exists an $A$-linear isometry $\phi:M\rar{\simeq}M$ such that $\phi\ga\phi^{-1}(m)=F_q(\ga)(m)$ for all $m\in M$ and all $\ga\in\Ga$. Then some power of $\phi$ is an automorphism of $M$ in $\cC$, so by Proposition \ref{p:drinfeld-2}(b), some (possibly larger) power of $\phi$ is equal to the identity. Hence the action of $\Ga$ on $M$ extends to a continuous action of $\widetilde{\Ga}=\Gal(\bF/\bF_q)\ltimes\Ga$ with finite image, obtained by sending $F_q\in\Gal(\bF/\bF_q)$ to $\phi$. If $\widetilde{\rho}$ is the resulting representation of $\widetilde{\Ga}$ on $\ql\tens_A M$, then $\widetilde{\rho}$ satisfies all the requirements of Proposition \ref{p:drinfeld-representations}.

\begin{proof}[Proof of Proposition \ref{p:drinfeld-2}]
Assertion (b) follows from Schur's lemma and the fact that any element $\la\in A$ such that $\la\cdot\overline{\la}=1$ is a root of unity\footnote{To see this, note that since $E$ has only one place over $p$, the condition $\la\cdot\overline{\la}=1$ implies that $\la$ is integral over $\bZ$.}.

\mbr

Let us prove (a). Let $\rho$ be an irreducible representation of $\Ga$ over $\ql$ and choose $H\subset\Ga$ and $\chi:H\rar{}\ql^\times$ so that $\rho\cong\Ind_H^\Ga\chi$. Define $M_0\in\cC$ as follows. As an $A$-module, $M_0$ consists of functions $f:\Ga\rar{}A$ such that $f(\ga h)=\chi(h)^{-1}f(\ga)$ for all $\ga\in\Ga$ and all $h\in H$. The action of $\Ga$ is given by left translations. The Hermitian form is given by $\langle f_1,f_2\rangle=\sum_{\ga\in\Ga/H} f_1(\ga)\cdot\overline{f_2(\ga)}$.

\mbr

Note that $\ql\tens_A M_0$ is isomorphic to $\rho$ as a representation of $\Ga$ by construction. Moreover, if $M\in\overline{\cC}$ is such that $\ql\tens_A M\cong\rho$, then $M$ is $\cC$-isomorphic to $M_0$: a choice of $m\in M^{H,\chi}$ with $\langle m,m\rangle=1$ defines mutually inverse morphisms
\[
f \longmapsto \sum_{\ga\in\Ga/H} f(\ga)\cdot \ga(m), \qquad f\in M_0;
\]
\[
x \longmapsto f_x, \quad f_x(\ga):=\langle x,\ga(m) \rangle, \qquad x\in M.
\]

\mbr

It remains to check that $M_0\in\overline{\cC}$. That is, we have to show that if $\rho\cong\Ind_{H'}^\Ga\chi'$ for some $H'\subset\Ga$ and some $\chi':H'\rar{}\ql^\times$, then the $A$-module $M_0^{H',\chi'}$ has a generator $m'$ such that $\langle m',m'\rangle=1$. If $p^r\leq 2$, then $\Ga$ is abelian, so that $H'=H=\Ga$ and the existence of $m'$ is obvious. If $p^r>2$, one can argue as follows. It is easy to see that $M_0^{H',\chi'}$ has a generator $m''$ such that $\langle m'',m''\rangle=p^s$ for some $s\in\bZ$. By Remark \ref{r:absolute-value-square-root-of-p}, we can find $\la\in A^\times$ such that $\la\cdot\overline{\la}=p$. Then we can put $m'=\la^{-s}\cdot m''$.
\end{proof}

The proof of the next consequence of Proposition \ref{p:drinfeld-2} is of independent interest.

\begin{cor}[V.~Drinfeld]\label{c:drinfeld}
Let $\Ga$ be a finite $p$-group, let $\rho$ be an irreducible representation of $\Ga$ over $\ql$, and let $\Aut_\rho(\Ga)$ be the stabilizer of the isomorphism class of $\rho$ in $\Aut(\Ga)$. Thus $\rho$ defines a projective representation of $\Aut_\rho(\Ga)$, and therefore a central extension of $\Aut_\rho(\Ga)$ by $\ql^\times$. We write $\al_\rho\in H^2\bigl(\Aut_\rho(\Ga),\ql^\times\bigr)$ for its cohomology class. Then $\al_\rho$ belongs to the image of the natural map
\[
H^2\bigl(\Aut_\rho(\Ga),\mu_{p^r}\bigr) \rar{} H^2\bigl(\Aut_\rho(\Ga),\ql^\times\bigr),
\]
where $p^r$ is the exponent of $\Ga$.
\end{cor}

\begin{proof}
If $\mu_E\subset E^\times$ denotes the group of all roots of unity in $E=\bQ(\mu_{p^r})$, then Proposition \ref{p:drinfeld-2} implies that $\al_\rho$ belongs to the image of the natural map
\[
H^2\bigl(\Aut_\rho(\Ga),\mu_E\bigr) \rar{} H^2\bigl(\Aut_\rho(\Ga),\ql^\times\bigr).
\]
The group $\mu_E$ is isomorphic to $\mu_{p^r}\oplus B$, where $B$ has order $2$ if $p\neq 2$ and $B=0$ if $p=2$. So to finish the proof of the corollary, it suffices to show that $\al_\rho$ is annihilated by a power of $p$. Now $\al_\rho$ comes from a central extension
\[
1 \rar{} E^\times \rar{} \Pi \rar{} \Aut_\rho(\Ga) \rar{} 1.
\]
The homomorphism $\det:\Pi\rar{}E^\times$ induces on $E^\times$ the map $\la\mapsto\la^{\dim(\rho)}$. So $\al_\rho$ is annihilated by $\dim(\rho)$, which is a power of $p$.
\end{proof}


\end{document}